\documentclass[oneside,a4paper]{article}%
\usepackage[final]{microtype}%
\usepackage{amsthm,mathtools}%
\usepackage{xcolor}%
\usepackage{float}%
\usepackage[export]{adjustbox}%
\usepackage{tikz, tikz-cd, mathtools, amssymb, stmaryrd}\usepackage{tangle}\usetikzlibrary{calc}
\usetikzlibrary{decorations.pathmorphing}
\usetikzlibrary{spath3}
\usepackage{quiver}
\usepackage{tikzit}
\usepackage[style=alphabetic,maxnames=99,maxalphanames=4,maxbibnames=10]{biblatex}
\tikzstyle{blob}=[fill={rgb,255: red,240; green,240; blue,255}, draw=black, shape=circle, scale=0.8, minimum size=0.4cm, inner sep=0.05cm]
\tikzstyle{Blob}=[fill={rgb,255: red,240; green,240; blue,255}, draw=black, shape=rectangle, scale=0.8, minimum size=0.6cm, inner sep=0.1cm, rounded corners, anchor=center]
\tikzstyle{Gap}=[fill=none, dashed, draw=black, shape=circle, scale=0.8, minimum size=1cm, inner sep=0.1cm, anchor=center]
\tikzstyle{intxtB}=[fill=white, draw=black, shape=rectangle, scale=0.8, minimum size=0.5cm, inner sep=0.1cm, rounded corners, anchor=center]
\tikzstyle{spot}=[fill=white, draw=black, shape=circle]
\tikzstyle{iblob}=[fill=white, draw=none, shape=circle, scale=0.8, minimum size=0.6cm, inner sep=0.05cm]
\tikzstyle{iBlob}=[fill=white, draw=none, shape=rectangle, scale=0.8, minimum size=0.6cm, inner sep=0.1cm, rounded corners]
\tikzstyle{tblob}=[fill=white, draw=black, shape=circle, scale=0.8, minimum size=0.3cm, inner sep=0.05cm]
\tikzstyle{small-label}=[fill=none, draw=none, shape=rectangle, tikzit category=label, font={\footnotesize}, scale=0.8]
\tikzstyle{norm-label}=[fill=none, draw=none, shape=rectangle, tikzit category=label, font={\normalsize}, scale=0.8]
\tikzstyle{big-label}=[fill=none, draw=none, shape=rectangle, tikzit category=label, font={\large}, scale=0.8]
\tikzstyle{ZeroCell}=[fill=none, shape=rectangle, draw=black, tikzit category=label, font={\footnotesize}, scale=1]
\tikzstyle{over-label}=[fill=white, draw=none, shape=rectangle, tikzit category=label, font={\footnotesize}, scale=0.8]
\tikzstyle{equals}=[fill=none, draw=none, shape=rectangle, font={$=$}]
\tikzstyle{eBlob}=[fill=white, draw=black, shape=circle, scale=0.8, minimum size=0.5cm, inner sep=0.05cm, font={$=$}]
\tikzstyle{starb}=[fill=white, draw=black, shape=circle, scale=0.8, minimum size=0.6cm, inner sep=0.05cm, font={$*$}]
\tikzstyle{mapsto}=[fill=none, draw=none, shape=rectangle, font={$\mapsto$}]
\tikzstyle{leadsto}=[fill=none, draw=none, shape=rectangle, font={\Large $\leadsto$}]
\tikzstyle{compos}=[fill=none, draw=none, shape=rectangle, font={$\circ$}]
\tikzstyle{Box}=[fill=white, draw=black, shape=rectangle, anchor=center, inner sep=0.1cm, align=center, minimum width=1.3cm, minimum height=1cm, thick, font={\footnotesize \scshape}]
\tikzstyle{miniblob}=[fill=white, draw=black, shape=circle, scale=0.8, minimum size=0.4cm, inner sep=0.05cm]
\tikzstyle{genA}=[fill=white, draw=black, shape=circle, inner sep=0pt, minimum width=0.25cm, tikzit fill=red, thick, outer sep=-0.1cm]
\tikzstyle{genAn}=[fill=white, draw=black, shape=circle, path picture={{ \draw[black]
(path picture bounding box.south) -- (path picture bounding box.north) (path picture bounding box.west) -- (path picture bounding box.east);}}, inner sep=0pt, minimum width=0.25cm, tikzit fill={rgb,255: red,130; green,0; blue,0}, thick, outer sep=-0.1cm]
\tikzstyle{genB}=[fill=white, draw=black, shape=isosceles triangle, shape border rotate=90, isosceles triangle apex angle=90, anchor=lower side, minimum width=0.5cm, inner sep=0pt, outer sep=-0.1pt, tikzit fill={rgb,255: red,255; green,128; blue,0}, thick, yshift=-0.034cm]
\tikzstyle{genBn}=[fill=white, draw=black, shape=isosceles triangle, shape border rotate=270, isosceles triangle apex angle=90, anchor=lower side, minimum width=0.5cm, inner sep=0pt, outer sep=-0.1pt, tikzit fill={rgb,255: red,171; green,85; blue,0}, thick, yshift=0.034cm]
\tikzstyle{genC}=[fill=white, draw=black, shape=diamond, anchor=center, minimum width=0.5cm, inner sep=0pt, minimum height=0.5cm, tikzit fill=yellow, thick]
\tikzstyle{genZn}=[fill=white, draw=black, shape=regular polygon, regular polygon sides=4, anchor=center, minimum width=0.5cm, inner sep=0pt, outer sep=-1pt, minimum height=0.5cm, tikzit fill={rgb,255: red,18; green,18; blue,127}, thick]
\tikzstyle{genD}=[fill=white, draw=black, shape=diamond, path picture={{ \draw[black]
(path picture bounding box.south) -- (path picture bounding box.north) (path picture bounding box.west) -- (path picture bounding box.east);
}}, tikzit fill=green, thick]
\tikzstyle{genDn}=[fill=white, draw=black, shape=diamond, path picture={{ \draw[black]
(path picture bounding box.south west) -- (path picture bounding box.north east) (path picture bounding box.north west) -- (path picture bounding box.south east);
}}, tikzit fill=white, thick]
\tikzstyle{genE}=[fill=white, draw=black, shape=circular sector, shape border rotate=270, circular sector angle=90, anchor=center, inner sep=0.02, minimum size=0.33cm, yshift=-0.026cm, tikzit fill=blue, thick]
\tikzstyle{genEn}=[fill=white, draw=black, shape=circular sector, shape border rotate=90, circular sector angle=90, anchor=center, inner sep=0.02, minimum size=0.33cm, yshift=0.026cm, tikzit fill={rgb,255: red,11; green,243; blue,255}, thick]
\tikzstyle{genX}=[fill=white, draw=black, shape=trapezium, trapezium angle=60, shape border rotate=180, anchor=bottom side, inner sep=0.02, minimum size=0.25cm, yshift=0.034cm, tikzit fill={rgb,255: red,150; green,243; blue,155}, thick]

\tikzstyle{string}=[-]
\tikzstyle{dirstring}=[-,postaction=decorate, decoration={markings, 
	mark= at position 0.5 with {\arrow{to}}}]
\tikzstyle{border}=[-, very thick]
\tikzstyle{divider}=[-, tikzit draw={rgb,255: red,191; green,0; blue,64}, dashed, draw=black]
\tikzstyle{dots}=[-, tikzit draw={rgb,255: red,191; green,0; blue,64}, dotted, draw=black]
\tikzstyle{invis}=[-, tikzit draw={rgb,255: red,0; green,255; blue,64}, draw=none, pattern={crosshatch}]
\tikzstyle{ident}=[-, tikzit draw={rgb,255: red,191; green,0; blue,64}, dashed, draw=black]
\tikzstyle{outline}=[-, fill=white, tikzit fill={rgb,255: red,166; green,205; blue,255}, draw=black, tikzit draw={rgb,255: red,3; green,5; blue,115}]
\tikzstyle{Grn}=[-, draw=mySeaGreen, ultra thick, tikzit draw={rgb,255: red,0; green,126; blue,18}]
\tikzstyle{Prp}=[-, draw=myPurple, ultra thick, tikzit draw={rgb,255: red,180; green,40; blue,240}]
\tikzstyle{Inrt}=[-, draw=myBlue, ultra thick, tikzit draw={rgb,255: red,20; green,20; blue,180}]
\tikzstyle{PrpId}=[-, draw={rgb,255: red,144; green,92; blue,203}, very thick, dashed]
\tikzstyle{GrnId}=[-, draw=mySeaGreen, very thick, dashed]
\tikzstyle{Grns}=[-, draw=mySeaGreen, thick, double, tikzit draw={rgb,255: red,129; green,208; blue,0}]
\tikzstyle{GPs}=[-, draw=myAmbig, thick, double, tikzit draw={rgb,255: red,61; green,149; blue,208}]
\tikzstyle{SGr}=[-, draw=black, double=mySeaGreen, double distance=2pt, tikzit draw={rgb,255: red,129; green,208; blue,0}]
\tikzstyle{SPr}=[-, draw=black, double=myPurple, double distance=2pt, tikzit draw={rgb,255: red,208; green,31; blue,161}]
\tikzstyle{Grnsvar}=[-, draw=mySeaGreen, very thick, decorate, decoration={coil,aspect=0}, tikzit draw={rgb,255: red,129; green,208; blue,0}]
\tikzstyle{Prps}=[-, draw=myPurple, thick, double, tikzit draw={rgb,255: red,240; green,54; blue,237}]

  \usetikzlibrary{nfold}
  \tikzset{
  popleft/.style={
  rounded corners,
  to path={-- ([xshift=-#1]\tikztostart.west) |- (\tikztotarget)[pos=0.25]\tikztonodes}
  },
  popright/.style={
  rounded corners,
  to path={-- ([xshift=#1]\tikztostart.east) |- (\tikztotarget)[pos=0.25] \tikztonodes}
  },
  popup/.style={
  rounded corners,
  to path={-- ([yshift=#1]\tikztostart.north) -| (\tikztotarget)[pos=0.25] \tikztonodes}
  },
  popdown/.style={
  rounded corners,
  to path={-- ([yshift=-#1]\tikztostart.south) -| (\tikztotarget)[pos=0.25] \tikztonodes}
  },
  myNArrow/.style={double equal sign distance,>={Implies},->},
  commutative diagrams/Rightarrow/.style={myNArrow, nfold},
  commutative diagrams/equal/.style={double equal sign distance, nfold, no head},
  commutative diagrams/equals/.style={double equal sign distance, nfold, no head},
  triple/.style={myNArrow,scaling nfold=3},
  quadruple/.style={myNArrow,scaling nfold=4},
  mybold/.style={line width=1.2pt, -{To[length=3pt]}},
  commutative diagrams/eql/.style={double equal sign distance, dash, nfold}
  }

\tikzset{curve/.style={settings={#1},to path={(\tikztostart)
    .. controls ($(\tikztostart)!\pv{pos}!(\tikztotarget)!\pv{height}!270:(\tikztotarget)$)
    and ($(\tikztostart)!1-\pv{pos}!(\tikztotarget)!\pv{height}!270:(\tikztotarget)$)
    .. (\tikztotarget)\tikztonodes}},
    settings/.code={\tikzset{quiver/.cd,#1}
        \def\pv##1{\pgfkeysvalueof{/tikz/quiver/##1}}},
    quiver/.cd,pos/.initial=0.35,height/.initial=0}
\tikzset{between/.style n args={2}{/tikz/spath/at end path construction={
    \tikzset{spath/split at keep middle={current}{#1}{#2}}
}}}

\tikzset{tail reversed/.code={\pgfsetarrowsstart{tikzcd to}}}
\tikzset{2tail/.code={\pgfsetarrowsstart{Implies[reversed]}}}
\tikzset{2tail reversed/.code={\pgfsetarrowsstart{Implies}}}
\tikzset{no body/.style={/tikz/dash pattern=on 0 off 1mm}}
    \usepackage{mathpartir}
\newtheorem{theorem}{Theorem}[section]%
\newtheorem{lemma}[theorem]{Lemma}%
\newtheorem{corollary}[theorem]{Corollary}%
\theoremstyle{definition}%
\newtheorem{definition}[theorem]{Definition}%
\newtheorem{notation}[theorem]{Notation}%
\newtheorem{explication}[theorem]{Explication}%
\newtheorem{proposition}[theorem]{Proposition}%
\newtheorem{example}[theorem]{Example}%
\newtheorem{remark}[theorem]{Remark}%
\usepackage[colorlinks=true,linkcolor={blue!30!black}]{hyperref}%
\usepackage{newpxmath,newpxtext}%
\usepackage{cleveref}%
\usepackage[mode=buildmissing]{standalone}%
\DeclareFontFamily{U}{dmjhira}{}
\DeclareFontShape{U}{dmjhira}{m}{n}{
  <-> dmjhira
}{}
\DeclareFontSubstitution{U}{dmjhira}{m}{n}

\newcommand{\yo}{{\usefont{U}{dmjhira}{m}{n}\symbol{"48}}}

    \title{Comparing loose bimodules and double barrels using pseudo-models of enhanced sketches}\author{
        Jason Brown
      \and{}
        Kevin Carlson
      \and{}
        Sophie Libkind
      \and{}
        David Jaz Myers
      }
\begin{document}

\maketitle
\begin{abstract}
        (Pseudo) double categories have two sorts of morphisms: tight ones which compose strictly, and loose ones which compose up to coherent isomorphism. In this paper, we consider bimodules between double categories in the loose direction. We provide two formulation of this concept --- first as pseudo-bimodules between pseudo-categories in the 2-category of categories, and second as double barrels generalizing Joyal's definition of bimodules between categories as functors into the walking arrow --- and prove these two formulations equivalent. In order to prove this equivalence, we define a notion of \emph{pseudo-model} of an enhanced sketch, which may be of independent interest. We then consider some double category theory unlocked by the theory of loose bimodules: loose adjunctions, and loose limits.
    \end{abstract}

    \tableofcontents
    \section{Introduction}\label{djm-00JP}\par{}
 	The central lesson of category theory is that one should always consider mathematical objects together with their appropriate form of \emph{morphism}. When the objects are themselves categories, the morphisms are functors. But to develop the theory of categories categorically, it is not sufficient to work with functors; most central notions in category theory are better expressed by combining functors with \emph{bimodules} \(M : \mathsf {C}^{\mathsf {op}} \times  \mathsf {D} \to  \mathsf {Set}\) between categories (also known as \emph{profunctors}  or \emph{distributors} ). For example, adjunctions and colimits are defined via isomorphisms between bimodules:
 	\begin{equation}\mathsf {C}(Fx, y) \cong  \mathsf {D}(x, Gy), \quad \quad  \mathsf {C}(\operatorname {colim} D, y) \cong  \mathsf {C}^{J}(D(-), y)\end{equation}
 	While the above notions can be defined in ordinary category theory using functors and natural transformations, their definitions using bimodules are more general and more robust. Indeed, \emph{weighted} colimits and \emph{pointwise} Kan extensions in enriched and internal settings cannot generally be expressed 2-categorically, and so must be expressed in terms of bimodules, or some equivalent structure. Simply put, bimodules between categories are central because \emph{representability} is central in category theory.
 \par{}
 	 The ur-bimodule is the hom-bimodule \(\mathsf {C}(-, -)\) of \emph{morphisms} on an object \(\mathsf {C}\); many other bimodules are constructed from hom bimodules by \emph{restriction}, composition, and mapping. With restrictions of hom bimodules, we can define universal properties by isomorphism of bimodules as in the examples above. It is for this reason that \emph{formal} category theory is naturally expressed in the (virtual) double category of categories, functors, and bimodules (as in \cite{cruttwell-2010-unified} and \cite{koudenburg-2024-formal}), though it can  also be expressed via Yoneda structures (as in \cite{gray-1974-formal} or \cite{di-liberti-2019-unicity}).
\par{}
 	Where categories have a single notion of morphism, double categories, definitionally, have two: the “tight” morphisms which compose strictly and the “loose” morphisms whose composition is only unital and associative up to coherent isomorphism. Double categories, double functors, and tight transformations form a 2-category, and one can do much of double category theory working in this 2-category. This is a key advantage of
  double categories over the older concept of bicategory (\cite{bénabou-1967-introduction}). Taking the \emph{tight} transformations as 2-cells means that this notion of double category theory privileges tight morphisms. The corresponding notion of \emph{tight} bimodule was given by Bob Paré, at least implicitly, in \emph{Yoneda theory for double categories} \cite{pare-2011-yoneda} and explicitly by Bertalan in his thesis \cite{bertalan-2012-hidak} (there called “double profunctors”) as lax double functors
 	\begin{equation}M : \mathbb {C}^{\mathsf {op}} \times  \mathbb {D} \to  \mathbb {S}\mathsf {pan}(\mathsf {Set}).\end{equation}
 	The intuition is that \(M(c, d)\) is a set of \emph{tight heteromorphisms} from \(c \in  \mathbb {C}\) to \(d \in  \mathbb {D}\); these are acted upon by pre- and post-composition in the tight direction.
 	These \emph{tight} bimodules suffice to express tight universal properties, as Paré shows in the paper
  just cited.
 \par{}
 	But as pseudo-categories internal to \(\mathcal {C}\mathsf {at}\), double categories support another notion of bimodule which directly categorifies bimodules between ordinary categories, viewed as categories internal to \(\mathsf {Set}\). Bimodules of this sort are called \emph{loose} bimodules. In a loose bimodule, \(M(c, d)\) is instead a \emph{category} of “loose” heteromorphisms from \(c\) to \(d\), acted upon by pre- and post-composition in the loose direction --- but this time
  the actions are only unital and associative up to coherent isomorphism. More explicitly, if we see a double category \(\mathbb {C}\) as a pseudo-monad in spans of categories:
   
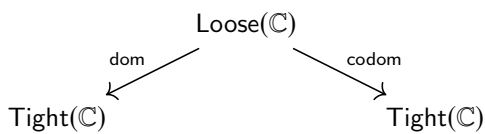
\begin{center}\label{djm-00JP-fig0}
 		\begin {tikzcd}
 	& {\mathsf {Loose}({\mathbb {C}})} \\
 	{ \mathsf {Tight}({\mathbb {C}})} && { \mathsf {Tight}({\mathbb {C}})}
 	\arrow ["{\mathsf {dom}}"', from=1-2, to=2-1]
 	\arrow ["{\mathsf {codom}}", from=1-2, to=2-3]
 \end {tikzcd}
 	\end{center}
 	then a loose bimodule \(M : \mathbb {C} \mathrel {\mkern 3mu\vcenter {\hbox {$\shortmid $}}\mkern -10mu{\to }} \mathbb {D}\) should be \emph{pseudo-bimodule} between these pseudo-monads, consisting of a span of categories
 	
\begin{center}
 	\begin {tikzcd}
 	& {\mathsf {Car} (\mathbb {M})} \\
 	{ \mathsf {Tight}({\mathbb {C}})} && { \mathsf {Tight}({\mathbb {D}})}
 	\arrow ["{\mathsf {dom}}"', from=1-2, to=2-1]
 	\arrow ["{\mathsf {codom}}", from=1-2, to=2-3]
 \end {tikzcd}
 	\end{center}

 	together with a (coherent) action by \(\mathbb {C}\) and \(\mathbb {D}\) as expressed by the following commutative diagram of categories:
    
\begin{center}
 		\begin {tikzcd}[row sep=0.5em, column sep=0.1em]
 	{ \mathsf {Tight}({\mathbb {C}})} && { \mathsf {Tight}({\mathbb {C}})} && { \mathsf {Tight}({\mathbb {D}})} && { \mathsf {Tight}({\mathbb {D}})} \\
 	& {\mathsf {Loose}({\mathbb {C}})} && {\mathsf {Car} (\mathbb {M})} && {\mathsf {Loose}({\mathbb {D}})} \\
 	&& \bullet  && \bullet  \\
 	&&& \bullet  \\
 	&&& \\
 	{ \mathsf {Tight}({\mathbb {C}})} &&&{\mathsf {Car} (\mathbb {M})} &&& { \mathsf {Tight}({\mathbb {D}})}
 	\arrow [equals, from=1-1, to=6-1]
 	\arrow [equals, from=1-7, to=6-7]
 	\arrow ["{\mathsf {dom}}"', from=2-2, to=1-1]
 	\arrow ["{\mathsf {codom}}", from=2-2, to=1-3]
 	\arrow ["{\mathsf {dom}}"', from=2-4, to=1-3]
 	\arrow ["{\mathsf {codom}}", from=2-4, to=1-5]
 	\arrow ["{\mathsf {dom}}"', from=2-6, to=1-5]
 	\arrow ["{\mathsf {codom}}", from=2-6, to=1-7]
 	\arrow ["\lrcorner "{anchor=center, pos=0.125, rotate=135}, draw=none, from=3-3, to=1-3]
 	\arrow [from=3-3, to=2-2]
 	\arrow [from=3-3, to=2-4]
 	\arrow ["\lrcorner "{anchor=center, pos=0.125, rotate=135}, draw=none, from=3-5, to=1-5]
 	\arrow [from=3-5, to=2-4]
 	\arrow [from=3-5, to=2-6]
 	\arrow ["\lrcorner "{anchor=center, pos=0.125, rotate=135}, draw=none, from=4-4, to=2-4]
 	\arrow [from=4-4, to=3-3]
 	\arrow [from=4-4, to=3-5]
 	\arrow [dashed, from=4-4, to=6-4]
 	\arrow ["{\mathsf {dom}}"',from=6-4, to=6-1]
 	\arrow ["{\mathsf {codom}}", from=6-4, to=6-7]
 \end {tikzcd}
 	\end{center}

 	We refer to \(\mathsf {Car}(\mathbb {M})\) as the \emph{carrier} of the loose bimodule; it displays the category \(M(f,g)\) of \emph{heterosquares}
 	
\begin{center}
 	\begin {tikzcd}
 	c & d \\
 	c' & d'
 	\arrow [""{name=0, anchor=center, inner sep=0}, "{m_0}", "\shortmid "{marking}, from=1-1, to=1-2]
 	\arrow ["f"', from=1-1, to=2-1]
 	\arrow ["g", from=1-2, to=2-2]
 	\arrow [""{name=1, anchor=center, inner sep=0}, "{m_1}"', from=2-1, to=2-2]
 	\arrow ["\sigma ", shorten <=4pt, shorten >=4pt, Rightarrow, from=0, to=1]
 \end {tikzcd}
 	\end{center}

 	(composing in the tight direction) over the categories \(\mathsf {Tight}({\mathbb {C}})\) and \(\mathsf {Tight}({\mathbb {D}})\).
 \par{}
 	Loose bimodules can be used to define loose universal properties. For example, there is an evident equivalence \(\mathbb {S}\mathsf {pan}(\mathsf {C})(a \times  b, c) \simeq  \mathbb {S}\mathsf {pan}(\mathsf {C})(a, b \times  c)\) given by moving one leg to the other side:
 	
\begin{center}
 	\begin {math}
 		\left \{
 		\begin {tikzcd}[row sep=0.8em, column sep=0.8em]
 	& S \\
 	{A \times  B} && C
 	\arrow ["{(a, b)}"', from=1-2, to=2-1]
 	\arrow ["c", from=1-2, to=2-3]
 \end {tikzcd}
 		\right \}
 		\quad \simeq \quad 
 		\left \{
 		\begin {tikzcd}[row sep=0.8em, column sep=0.8em]
 	& S \\
 	A && {B \times  C}
 	\arrow ["a"', from=1-2, to=2-1]
 	\arrow ["{(b, c)}", from=1-2, to=2-3]
 \end {tikzcd}
 		\right \}
 		\end {math}
 	\end{center}

 	Similarly, \(\mathbb {P}\mathsf {rof}(\mathsf {A} \times  \mathsf {B}, \mathsf {C}) \simeq  \mathbb {P}\mathsf {rof}(\mathsf {A}, \mathsf {B}^{\mathsf {op}} \times  \mathsf {C})\) in the double category of bimodules of categories. Both of these equivalences can be lifted to equivalences of loose bimodules, and they express a \emph{loose adjunction} between double functors. These particular adjunctions were considered by Patterson as constitutional of a good notion of \emph{compact closed double category} \cite{patterson-2024-toward} \cite{patterson-2024-toward-2}.
 \par{}
 	As another example of a loose universal property which can be expressed by an equivalence of loose bimodules, consider the coproduct \(+ : \mathsf {C} \times  \mathsf {C} \to  \mathsf {C}\) in a lextensive category \(\mathsf {C}\) \cite{carboni-1993-introduction}. For \(\mathsf {C}\) to be lextensive means that for any objects \(a, b \in  \mathsf {C}\), the coproduct functor \(+ : (\mathsf {C} \downarrow  a) \times  (\mathsf {C} \downarrow  b) \to  (\mathsf {C} \downarrow  (a + b))\) is an equivalence. Since \(\mathsf {C} \downarrow  x \simeq  \mathbb {S}\mathsf {pan}(\mathsf {C})(x, \bullet )\), a natural expression of universal property of lextensive coproducts is as \emph{loose} coproducts in the double category \(\mathbb {S}\mathsf {pan}(\mathsf {C})\) of spans.
 	\begin{equation}\mathbb {S}\mathsf {pan}(\mathsf {C})(a + b, c) \simeq  \mathbb {S}\mathsf {pan}(\mathsf {C})(a, c) \times  \mathbb {S}\mathsf {pan}(\mathsf {C})(b, c).\end{equation}
 	More generally, Heindel and Sobociński \cite{heindel-2009-van} showed that van Kampen colimits are those colimits which remain colimits in the bicategory of spans; that is to say, van Kampen colimits have a \emph{loose} universal property in \(\mathbb {S}\mathsf {pan}(\mathsf {C})\).
 \par{}
 	The basic theory of loose bimodules has not yet been laid out. In this paper, we aim to settle some basic definitions and theory of loose bimodules, sufficent for expressing the above ideas and for applications in studying the compositionality of coupled dynamical systems \cite{libkind-2025-towards}. In particular, in this paper we will:
 \subsection{Define the 2-category of loose bimodules.}\label{djm-00JP-1}\par{}
 		 In fact, we will provide \emph{two} definitions for loose bimodules and prove them to be equivalent. The first adapts Joyal's \emph{barrel} presentation for bimodules. It is known that barrels — functors from a small category into the walking arrow, \(\Delta  [1]\) — are equivalent to bimodules internal to \(\mathsf {Set}\) via an equivalence of categories \(\mathsf {Cat} \downarrow  \Delta  [1] \simeq  \mathsf {Set}\text {-}\mathsf {Bimod}\). This suggests we define loose bimodules as double functors \(\mathbb {M} \to  \mathbb {L}\mathsf {oose}\) into the \hyperref[ssl-001G]{walking loose arrow}, which immediately yields a 2-category of loose bimodules given by the slice \(\mathcal {D}\mathsf {bl} \downarrow  \mathbb {L}\mathsf {oose}\).
We call loose bimodules presented in this way .
 \par{}
We justify this definition for loose bimodules by showing that it is equivalent to a more explicit formalisation of the notion of a \emph{pseudo-bimodule} between pseudo-categories sketched above. Our chosen formalisation for pseudo-bimodules will be expressed in the language of models for \emph{limit sketches}. We first recall in \cref{jrb-001S} some of the basic theory of limit sketches, including the limit sketch for ordinary bimodules,  before developing the theory of \hyperref[djm-00I5]{\emph{pseudo-models}} for \emph{\(\mathscr {F}\)-sketches} in \cref{djm-00I9}.
We then show that pseudo-bimodules are precisely pseudo-models in \(\mathsf {Cat}\) for the sketch whose \emph{strict} models in \(\mathsf {Set}\) are ordinary bimodules. We thus exhibit all the coherence data and conditions inherent to pseudo-bimodules as arising canonically from the strict notion of bimodules.
\par{}
To prove that our double-barrel definition is equivalent to our sketch-based notion of pseudo-bimodule internal to \(\mathsf {Cat}\), we begin by showing in \cref{djm-00IE} that Leinster's \emph{unbiased} double categories (Definition 5.2.1 of \cite{leinster-2004-higher}) are equivalent to the \hyperref[djm-00I5]{pseudo-models} for an \(\mathscr {F}\)-sketch structure on \(\Delta ^{\mathsf {op}}\) which encodes the \emph{Segal conditions}. We then give a “slice theorem” for pseudo-models of locally-discrete \(\mathscr {F}\)-sketches in \cref{djm-00IG}, which describes conditions under which the slice of a 2-category of pseudo-models for an \(\mathscr {F}\)-sketch is again a 2-category of pseudo-models for an \(\mathscr {F}\)-sketch. This generalises the well-known result that slices of presheaf categories are again presheaf categories, and will be proven by introducing a notion of “model fibration” whose relation to pseudo-models mirrors that of discrete fibrations to presheaves. It will follow as a special case of this slice theorem that the 2-category \(\mathcal {D}\mathsf {bl} \downarrow  \mathbb {L}\mathsf {oose}\) of double barrels, being a slice of a 2-category of pseudo-models, is the 2-category of pseudo-models for the \(\mathscr {F}\)-sketch of pseudo-bimodules.
\subsection{Investigate loose universal properties.}\par{}
After laying out our definitions of loose bimodules, we turn to constructing them. The  is trivial to construct both as a double barrel and as a pseudo-bimodule. In order to express loose universal properties, however, we need to form restrictions of loose bimodules along double functors, which we address in \cref{jrb-003H}.
\par{}
With this notion of restriction we then provide examples of \emph{loose universal properties} — those expressible in terms of equivalences of loose bimodules. In particular, we give a definition of loose adjunction in \cref{kdc-000M} and use this definition to express \hyperref[kdc-000S]{loose terminal objects}, \hyperref[kdc-000Q]{loose products}, and \hyperref[kdc-000N]{loose closures}.
 \par{}
Finally, in \cref{kdc-000T}, we give a general definition of \hyperref[kdc-000T]{van Kampen colimit} which generalizes  Heindel and Sobociński's characterization of van Kampen colimits in \(\mathbb {S}\mathsf {pan}(\mathsf {C})\) \cite{heindel-2009-van}. We show that though they consider the universal property bicategorically, it matches our loose colimit definition because \(\mathbb {S}\mathsf {pan}(\mathsf {C})\) is an equipment.
 \section{Sketches, barrels and bimodules}\label{jrb-002Z}\par{}
  In this section we recall the notion of a \emph{limit sketch} and describe a “slice theorem” for models of sketches which shows that slices of sketchable categories are themselves sketchable. We then use this result to present the equivalence between barrels and internal bimodules in \(\mathsf {Set}\) in a way that mirrors our later proof of the equivalence between double barrels and loose bimodules. This will give us a bit of time to get comfortable with the sketches for categories and bimodules before we see them again in the more complicated 2-categorical setting.
\subsection{The limit sketch for bimodules}\label{jrb-001S}
\begin{definition}[{Limit sketch}]\label{jrb-001T}
\par{}
A \textbf{limit sketch} is a category \(\mathsf {L}\) along with a collection of functors \(\mathsf {J} \to  x \downarrow  \mathsf {L} \) for small \(\mathsf {J}\) and \(x \in  \mathsf {L}\) called the \textbf{marked cones} of \(\mathsf {L}\) (with \textbf{vertex} \(x\)).
\par{}
A \textbf{model} \(M : \mathsf {L} \to  \mathsf {K}\) is a functor on the underlying categories such that for each marked cone \(S : \mathsf {J} \to  x \downarrow  \mathsf {L}\) of \(\mathsf {L}\), the cone:
\begin{equation}\mathsf {J} \xrightarrow {S} x \downarrow  \mathsf {L} \xrightarrow {x \downarrow  M} Mx \downarrow  \mathsf {K}\end{equation}
is marked in \(\mathsf {K}\). The 2-category of limit sketches, sketch morphisms and (arbitrary) natural transformations is denoted by \(\mathsf {Sketch}\).
\end{definition}

\begin{example}[{Limit sketches from categories}]\label{jrb-001Y}
\par{}
We can endow an arbitrary category \(\mathsf {C}\) with a limit sketch structure by declaring its marked cones to be all limit cones in \(\mathsf {C}\). When we refer to a model of a limit sketch \(\mathsf {L}\) in a category \(\mathsf {C}\) it is with respect to this induced limit sketch structure on \(\mathsf {C}\). That is, a model of \(\mathsf {L}\) in \(\mathsf {C}\) is a functor \(M : \mathsf {L} \to  \mathsf {C}\) that sends marked cones to limit cones.
\end{example}

\begin{example}[{The limit sketch for categories}]\label{jrb-0031}
\par{}
The 1-category \(\mathsf {Cat}_0\) of (small) categories can be identified with the full subcategory of simplicial sets — i.e. functors \(\Delta ^{\mathsf {op}} \to  \mathsf {Set}\) — which satisfy the \emph{Segal condition}. This condition can equivalently be expressed as being a model for a certain sketch structure on \(\Delta ^{\mathsf {op}}\), which means \(\mathsf {Cat}_0\) is the category of models in \(\mathsf {Set}\) for this sketch. We will describe this sketch structure on \(\Delta ^{\mathsf {op}}\) in detail, as it is closely related to the sketch for bimodules.
\par{}
  The sketch structure is defined in terms of the active–inert orthogonal factorisation system on \(\Delta \), so we first briefly recall how these two classes of maps in \(\Delta \) are defined.
\par{}
The \emph{inert maps} are those which “preserve distance”, or equivalently are generated by the maps \(\mathsf {d}_{0}, \mathsf {d}_{n} \colon  [n-1] \to  [n]\) in \(\Delta \). For example, there are two inert maps \([1] \to  [2]\) in \(\Delta \). These are depicted below.
  
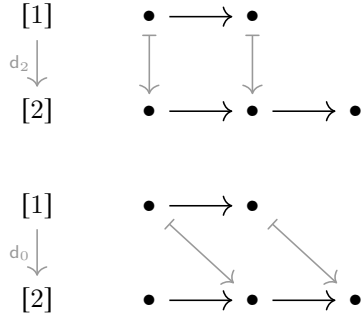
\begin{center}
    \begin {tikzcd}
      {[1]} & \bullet  & \bullet  \\
      {[2]} & \bullet  & \bullet  & \bullet  \\
      {[1]} & \bullet  & \bullet  \\
      {[2]} & \bullet  & \bullet  & \bullet 
      \arrow ["{\mathsf {d}_{2}}"', color={rgb,255:red,153;green,153;blue,153}, from=1-1, to=2-1]
      \arrow [from=1-2, to=1-3]
      \arrow [color={rgb,255:red,153;green,153;blue,153}, maps to, from=1-2, to=2-2]
      \arrow [color={rgb,255:red,153;green,153;blue,153}, maps to, from=1-3, to=2-3]
      \arrow [from=2-2, to=2-3]
      \arrow [from=2-3, to=2-4]
      \arrow ["{\mathsf {d}_{0}}"', color={rgb,255:red,153;green,153;blue,153}, from=3-1, to=4-1]
      \arrow [from=3-2, to=3-3]
      \arrow [color={rgb,255:red,153;green,153;blue,153}, maps to, from=3-2, to=4-3]
      \arrow [color={rgb,255:red,153;green,153;blue,153}, maps to, from=3-3, to=4-4]
      \arrow [from=4-2, to=4-3]
      \arrow [from=4-3, to=4-4]
    \end {tikzcd}
  \end{center}\par{}
A map is \emph{active} when it preserves endpoints, i.e. the top and bottom objects of \([n]\). Consequently there is a unique active map from \([1]\) to \([n]\) for any \([n] \in  \Delta \); the active map \([1] \to  [2]\) is depicted below
  
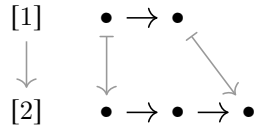
\begin{center}
    \begin {tikzcd}[column sep=small]
      {[1]} & \bullet  & \bullet  \\
      {[2]} & \bullet  & \bullet  & \bullet 
      \arrow [color={rgb,255:red,153;green,153;blue,153}, from=1-1, to=2-1]
      \arrow [from=1-2, to=1-3]
      \arrow [color={rgb,255:red,153;green,153;blue,153}, maps to, from=1-2, to=2-2]
      \arrow [color={rgb,255:red,153;green,153;blue,153}, maps to, from=1-3, to=2-4]
      \arrow [from=2-2, to=2-3]
      \arrow [from=2-3, to=2-4]
    \end {tikzcd}
\end{center}

This map fails to be inert because the distance between \(0,1 \in  [1]\) is \(1\), whereas the distance between their images in \([2]\) is \(2\).
\par{}
Now let \(\Delta ^{\mathsf {op}}_{\mathsf {inert}}\) denote the wide subcategory of inert maps in \(\Delta ^{\mathsf {op}}\), and let \(\mathsf {el} \colon  \Delta _{\mathsf {el}}^{\mathsf {op}} \hookrightarrow  \Delta ^{\mathsf {op}}\) denote the full subcategory of \(\Delta ^{\mathsf {op}}_{\mathsf {inert}}\) with objects \([0]\) and \([1]\), which we call \emph{elementary}.
\par{}
The required limit sketch structure on \(\Delta ^{\mathsf {op}}\) is then given by declaring the marked cones to be the canonical inclusions:
\begin{equation}[n] \downarrow _{\mathsf {inert}} \Delta ^{\mathsf {op}}_{\mathsf {inert}} \hookrightarrow  [n] \downarrow  \Delta ^{\mathsf {op}}\end{equation}
for each \([n] \in  \Delta \).
\par{}
What this means is that for each \([n] \in  \Delta ^{\mathsf {op}}\) there is a unique marked cone with apex \([n]\) consisting of all inert maps to the objects \([0]\) and \([1]\). For example, the unique marked cone with apex \([2]\) is the following: 
  
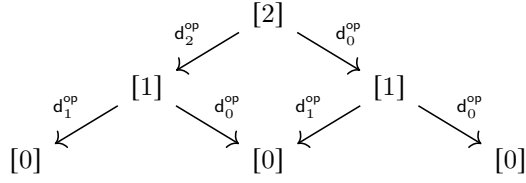
\begin{center}
    \begin {tikzcd}[row sep=1em]
      && {[2]} \\
      & {[1]} && {[1]} \\
      {[0]} && {[0]} && {[0]}
      \arrow ["{\mathsf {d}_{2}^{\mathsf {op}}}"', from=1-3, to=2-2]
      \arrow ["{\mathsf {d}_{0}^{\mathsf {op}}}", from=1-3, to=2-4]
      \arrow ["{\mathsf {d}_{1}^{\mathsf {op}}}"', from=2-2, to=3-1]
      \arrow ["{\mathsf {d}_{0}^{\mathsf {op}}}", from=2-2, to=3-3]
      \arrow ["{\mathsf {d}_{1}^{\mathsf {op}}}"', from=2-4, to=3-3]
      \arrow ["{\mathsf {d}_{0}^{\mathsf {op}}}", from=2-4, to=3-5]
    \end {tikzcd}
  \end{center}

  which is better visualised as a cocone in \(\Delta \): 
  
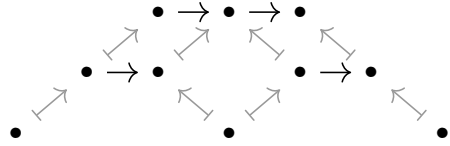
\begin{center}
    \begin {tikzcd}[column sep = small, row sep = 1em]
      && \bullet  & \bullet  & \bullet  \\
      & \bullet  & \bullet  && \bullet  & \bullet  \\
      \bullet  &&& \bullet  &&& \bullet 
      \arrow [from=1-3, to=1-4]
      \arrow [from=1-4, to=1-5]
      \arrow [draw={rgb,255:red,153;green,153;blue,153}, maps to, from=2-2, to=1-3]
      \arrow [from=2-2, to=2-3]
      \arrow [draw={rgb,255:red,153;green,153;blue,153}, maps to, from=2-3, to=1-4]
      \arrow [draw={rgb,255:red,153;green,153;blue,153}, maps to, from=2-5, to=1-4]
      \arrow [from=2-5, to=2-6]
      \arrow [draw={rgb,255:red,153;green,153;blue,153}, maps to, from=2-6, to=1-5]
      \arrow [draw={rgb,255:red,153;green,153;blue,153}, maps to, from=3-1, to=2-2]
      \arrow [draw={rgb,255:red,153;green,153;blue,153}, maps to, from=3-4, to=2-3]
      \arrow [draw={rgb,255:red,153;green,153;blue,153}, maps to, from=3-4, to=2-5]
      \arrow [draw={rgb,255:red,153;green,153;blue,153}, maps to, from=3-7, to=2-6]
    \end {tikzcd}
  \end{center}\par{}
For a simplicial set \(M \colon  \Delta  ^{\mathsf {op}} \to  \mathsf {Set}\) to be a \emph{model} for this sketch structure, it must satisfy the property that the image of these marked cones are limit cones. In particular, the maps \(M_{\mathsf {d}_{0}}, M_{\mathsf {d}_{2}} : M([2]) \to  M([1])\) must exhibit \(M([2])\) as a pullback of the two maps \(M_{\mathsf {d}_{1}}, M_{\mathsf {d}_{2}} : M([1]) \to  M([0])\):
  
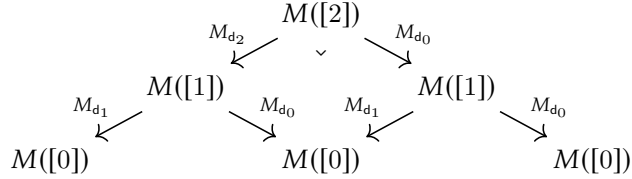
\begin{center}
    \begin {tikzcd}[row sep=1em, column sep=small]
      && {M([2])} \\
      & {M([1])} && {M([1])} \\
{M([0])} && {M([0])} && {M([0])}
      \arrow ["{M_{\mathsf {d}_{2}}}"', from=1-3, to=2-2]
      \arrow ["{M_{\mathsf {d}_{0}}}", from=1-3, to=2-4]
      \arrow ["{M_{\mathsf {d}_{1}}}"', from=2-2, to=3-1]
      \arrow ["{M_{\mathsf {d}_{0}}}", from=2-2, to=3-3]
      \arrow ["{M_{\mathsf {d}_{1}}}"', from=2-4, to=3-3]
      \arrow ["{M_{\mathsf {d}_{0}}}", from=2-4, to=3-5]
      \arrow ["\lrcorner "{anchor=center, pos=0.125, rotate=-45}, draw=none, from=1-3, to=3-3]
    \end {tikzcd}
\end{center}

and similarly \(M([n])\) must be given as the “\(n\)-fold” pullback of \(M([1])\) over \(M([0])\). This is precisely the Segal condition for a simplicial set to be the nerve of a category. Thus, a model \(M \colon  \Delta  ^{\mathsf {op}} \to  \mathsf {Set}\) of this limit sketch may be identified with the small category with object set \(M([0])\) morphism set \(M([1])\), and the source and target maps given by the images of the inert maps \(\mathsf {d}_{1}, \mathsf {d}_{0} \colon  [0] \to  [1]\). More generally, set \(M([n])\) is the set of \(n\) composable morphisms and the images of the various inert maps \([n] \hookrightarrow  [m]\) give the projections \(M([m]) \to  M([n])\) to sub-sequences of composable morphisms. The images of the \emph{active} maps give the composition of morphisms; in particular, the image of the unique active map \([1] \to  [n]\) defines the \(n\)-ary composition.
\end{example}

\begin{remark}[{Limit sketches from algebraic patterns}]\label{jrb-001V}
\par{}
The limit sketch structure on \(\Delta ^{\mathsf {op}}\) is an instance of a more general construction of limit sketches from \emph{algebraic patterns}. Algebraic patterns as introduced in \cite{chu-2019-homotopycoherent} are structures on \(\left ( \infty ,1 \right )\)-categories, which for our purposes can be “decategorified” to structure on a 1-category given by an orthogonal factorisation system, \(\mathsf {inert} \dashv  \mathsf {active}\), and a distinguished class of objects termed \textbf{elementary}. A category \(\mathsf {C}\) with this structure can be made into a limit sketch by declaring the marked cones to be:
\begin{equation}[x] \downarrow _{\mathsf {inert}} \mathsf {el}  \hookrightarrow  [x] \downarrow  \mathsf {C}\end{equation}
where \(\mathsf {el}\) is the inclusion in \(\mathsf {C}\) of the full subcategory of elementary objects.  is then the limit sketch structure induced by the algebraic pattern on \(\Delta ^{\mathsf {op}}\) with elementaries, inerts and actives as described there.
\end{remark}

\begin{example}[{The biased limit sketch for categories}]\label{jrb-0032}
\par{}
The condition of being a model for the above sketch structure on \(\Delta ^{\mathsf {op}}\) is rather strict, to the point that such a model \(M : \Delta ^{\mathsf {op}} \to  \mathsf {Set}\) can be recovered from its action on a finite subcategory of \(\Delta ^{\mathsf {op}}\). Indeed, the fact that nerves of categories are 2-coskeletal means that \(M\) can be recovered up to isomorphism from its restriction to the truncation \(\Delta _{\leq  2}^{\mathsf {op}}\). However, being a model for the restriction of the sketch structure on \(\Delta ^{\mathsf {op}}\) to \(\Delta _{\leq  2}^{\mathsf {op}}\) is not equivalent to being the 2-skeleton of the nerve of a category; there's no guarantee that such a truncated simplicial set satisfies associativity, which comes from an equality between morphisms \(\mathsf {d}_{2}\,\mathsf {d}_{1} =  \mathsf {d}_{1}\,\mathsf {d}_{1} : [1] \to  [3]\).
\par{}
Being a model for the restriction of the sketch structure to \(\Delta _{\leq  3}^{\mathsf {op}}\) \emph{does} suffice, though, as we now observe.
\par{}The images of the inerts \(\mathsf {d}_{0},\mathsf {d}_{1}: [0] \to  [1]\) under a map \(M : \Delta _{\leq  3}^{\mathsf {op}} \to  \mathsf {Set}\) determines a span:
\begin{equation}M_0 \xleftarrow {s} M_1 \xrightarrow {t} M_0\end{equation}
There is a unique possible way to then extend this action of \(M\) to any inerts in \(\Delta ^{\mathsf {op}}\) in a way that satisfies the Segal conditions, at least up to a choice of \(n\)-fold composites \(M_1^{(n)}\) of the span \(M_1\). Clearly we must map \([n]\) to the chosen \(M_1^{(n)}\), and the inert \(\mathsf {i}_{j} : [n] \to  [m]\) must be sent to the unique function \(f : M_1^{(m)} \to  M_1^{(n)}\) satisfying \(\pi _k\,f = \pi _{k+j}\), where \(\pi _k : M_1^{(n)} \to  M_1\) is the projection onto the \(k^{\text {th}}\) component. This follows from the equality of morphisms in \(\Delta \):
\begin{equation}[1] \xrightarrow {\mathsf {i}_{k}} [n] \xrightarrow {\mathsf {i}_{j}} [m] \quad  = \quad  [1] \xrightarrow {\mathsf {i}_{k+j}} [m]\end{equation}
and the fact that \(\pi _k = M_{\mathsf {i}_{k}}\).
\par{}
For extending the definition of the active maps in \(\Delta _{\leq  3}\) to all active maps in \(\Delta \), it is useful to observe that there is an isomorphism \(\theta \) between the wide subcategory \(\Delta _{\mathsf {active}}^{\mathsf {op}}\) of active maps and the augmented simplex category \(\Delta _{+}\) which can be visualised by replacing the “edges” of each \([n] \in  \Delta _{\mathsf {active}}^{\mathsf {op}}\) with vertices, and mapping morphisms as so:

  
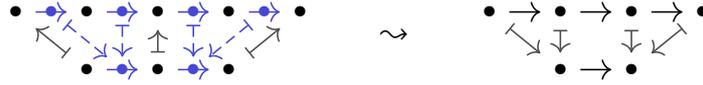
\begin{center}
\begin {tikzcd}[sep=small]
	\bullet  & \bullet  & \bullet  & \bullet  & \bullet  \\
	& \bullet  & \bullet  & \bullet 
	\arrow [""{name=0, anchor=center, inner sep=0}, "\bullet "{marking, text={rgb,255:red,71;green,71;blue,209}}, draw={rgb,255:red,71;green,71;blue,209}, from=1-1, to=1-2]
	\arrow [""{name=1, anchor=center, inner sep=0}, "\bullet "{marking, text={rgb,255:red,71;green,71;blue,209}}, draw={rgb,255:red,71;green,71;blue,209}, from=1-2, to=1-3]
	\arrow [""{name=2, anchor=center, inner sep=0}, "\bullet "{marking, text={rgb,255:red,71;green,71;blue,209}}, draw={rgb,255:red,71;green,71;blue,209}, from=1-3, to=1-4]
	\arrow [""{name=3, anchor=center, inner sep=0}, "\bullet "{marking, text={rgb,255:red,71;green,71;blue,209}}, draw={rgb,255:red,71;green,71;blue,209}, from=1-4, to=1-5]
	\arrow [color={rgb,255:red,77;green,77;blue,77}, maps to, from=2-2, to=1-1]
	\arrow [""{name=4, anchor=center, inner sep=0}, "\bullet "{marking, text={rgb,255:red,71;green,71;blue,209}}, draw={rgb,255:red,71;green,71;blue,209}, from=2-2, to=2-3]
	\arrow [color={rgb,255:red,77;green,77;blue,77}, maps to, from=2-3, to=1-3]
	\arrow [""{name=5, anchor=center, inner sep=0}, "\bullet "{marking, text={rgb,255:red,71;green,71;blue,209}}, draw={rgb,255:red,71;green,71;blue,209}, from=2-3, to=2-4]
	\arrow [color={rgb,255:red,77;green,77;blue,77}, maps to, from=2-4, to=1-5]
	\arrow [color={rgb,255:red,71;green,71;blue,209}, between={0.2}{0.8}, dashed, maps to, from=0, to=4]
	\arrow [color={rgb,255:red,71;green,71;blue,209}, between={0.2}{0.8}, dashed, maps to, from=1, to=4]
	\arrow [color={rgb,255:red,71;green,71;blue,209}, between={0.2}{0.8}, dashed, maps to, from=2, to=5]
	\arrow [color={rgb,255:red,71;green,71;blue,209}, between={0.2}{0.8}, dashed, maps to, from=3, to=5]
\end {tikzcd}
\qquad  $\leadsto $ \qquad 
\begin {tikzcd}[sep=small]
	\bullet  & \bullet  & \bullet  & \bullet  \\
	& \bullet  & \bullet 
	\arrow [from=1-1, to=1-2]
	\arrow [color={rgb,255:red,77;green,77;blue,77}, maps to, from=1-1, to=2-2]
	\arrow [from=1-2, to=1-3]
	\arrow [color={rgb,255:red,77;green,77;blue,77}, maps to, from=1-2, to=2-2]
	\arrow [from=1-3, to=1-4]
	\arrow [color={rgb,255:red,77;green,77;blue,77}, maps to, from=1-3, to=2-3]
	\arrow [color={rgb,255:red,77;green,77;blue,77}, maps to, from=1-4, to=2-3]
	\arrow [from=2-2, to=2-3]
\end {tikzcd}
\end{center}

The ordinal sum endows \(\Delta _{+}\) with a strict monoidal structure, and the Segal conditions on a model \(M : \Delta ^{\mathsf {op}} \to  \mathsf {Set}\) ensure that the composite of the induced map

For \(f : [n] \to  [m]\) active, we have that \(\mathsf {i}_{0} = f\,\mathsf {i}_{0} : [0] \to  [m]\) and \(\mathsf {i}_{m} = f\,\mathsf {i}_{n} : [0] \to  [m]\). Letting \(\mu _k : M_1^{(n)} \to  M_0\) denote the projection onto the \(k^\text {th}\) \(M_0\) component in the cone defining \(M_1^{(n)}\), we have \(\mu _0 \, M_f = M_0\) and \(\mu _n \, M_f = \mu _m\), so that \(M_f : M_1^{(m)} \to  M_1^{(n)}\) is moreover a map in \(\mathsf {Span} \left ( M_0,M_0 \right )\) from \(\mu _0,\mu _m : M_1^{(m)}\) to \(\mu _0,\mu _n : M_1^{(n)}\). Let \(M' : \Delta _\mathsf {active}^{\mathsf {op}}  \to  \mathsf {Span} \left ( M_0, M_0 \right )\) denote this induced map. The Segal conditions on \(M\) moreover ensure that the composite:
\begin{equation}\Delta _{+} \xrightarrow {\theta } \Delta _{\mathsf {active}}^{\mathsf {op}} \xrightarrow {M'} \mathsf {Span} \left ( M_0,M_0 \right )\end{equation}
is strictly-monoidal with respect to the ordinal sum in \(\Delta _{+}\) and the chosen compositions of the relevant spans in \(\mathsf {Span} \left ( M_0,M_0 \right )\). Such a monoidal map is equivalent to a monoid in \(\mathsf {Span} \left ( M_0,M_0 \right )\), which is equivalently a map \(N : \left ( \Delta _{+} \right )_{\leq  2} \to  \mathsf {Span} \left ( M_0,M_0 \right )\) which is monoidal to the extent that the monoidal product in \(\left ( \Delta _{+} \right )_{\leq  2}\) is defined. I.e. such that \(N_2 \cong  N_0 \otimes  N_0 \otimes  N_0\), etc. This is because all the axioms for a monoid can expressed as equations between at most ternary operations.
\par{}
It follows that any model \(M : \Delta _{\leq  3}^{\mathsf {op}} \to  \mathsf {Set}\) for the restriction of the sketch structure on \(\Delta ^{\mathsf {op}}\) to \(\Delta _{\leq  3}^{\mathsf {op}}\) uniquely extends to a model of \(\Delta ^{\mathsf {op}}\), and thus that the sketch morphism \(\Delta _{\leq  3}^{\mathsf {op}} \hookrightarrow  \Delta ^{\mathsf {op}}\) induces an equivalence of categories of models \(\mathsf {Sketch} \left ( \Delta ^{\mathsf {op}}, \mathsf {Set} \right ) \simeq  \mathsf {Sketch} \left ( \Delta _{\leq  3}^{\mathsf {op}}, \mathsf {Set} \right )\), so we can equally view \(\Delta _{\leq  3}^{\mathsf {op}}\) as a sketch for the category \(\mathsf {Cat}_0\).
\par{}
The difference between the two sketches for categories is that the data of a model for \(\Delta _{\leq  3}^{\mathsf {op}}\) encodes at most 3-ary composition, whereas models for \(\Delta ^{\mathsf {op}}\) express unbiased \(n\)-ary composition. For this reason, we refer to \(\Delta _{\leq  3}\) with its inherited sketch structure as the \emph{biased} sketch for categories.
\end{example}

\begin{remark}[{The sketch for internal categories}]\label{jrb-0033}
\par{}
We have observed an equivalence between categories models for \(\Delta ^{\mathsf {op}}\) in \(\mathsf {Set}\) and \(\mathsf {Cat}_0\), the 1-category of small categories. More generally, models of \(\Delta ^{\mathsf {op}}\) in any other category \(\mathcal {E}\) with pullbacks are equivalent to categories internal to \(\mathcal {E}\). This is perhaps made most clear by comparing models for the biased sketch for categories to the usual definition of internal category.
\end{remark}
\par{}
From our sketch(es) for categories we can obtain a sketch for barrels (or indeed any slice of \(\mathsf {Cat}\)) by the following general property of models for sketches in \(\mathsf {Set}\):

\begin{theorem}[{Slice theorem for sketches}]\label{jrb-001X}
\par{}
Given a limit sketch \(\mathsf {L}\) and a model \(M : \mathsf {L} \to  \mathsf {Set}\), there is an isomorphism:
\begin{equation}\mathsf {Sketch} \left ( \mathsf {L}, \mathsf {Set} \right ) \downarrow  M \cong  \mathsf {Sketch} \left ( \mathsf {El}\left (M\right ), \mathsf {Set} \right )\end{equation}
where \(\mathsf {El}\left (M\right )\) is the limit sketch on the category of elements of \(M\) whose marked cones are all lifts of marked cones in \(\mathsf {L}\) along the projection \(\pi _M : \mathsf {El}\left (M\right ) \to  \mathsf {L}\).
\end{theorem}

\begin{proof}[{Sketch of proof for \cref{jrb-001X}.}]\label{jrb-001S-proof}
\par{}\cref{jrb-001X} can be viewed as an extension of a well-known result about (co)presheaf categories that for any (accessible) copresheaf \(P : \mathsf {C} \to  \mathsf {Set}\), there is an equivalence:
\begin{equation}\left [ \mathsf {C}, \mathsf {Set} \right ] \downarrow  P \simeq  \left [ \mathsf {El}\left (P\right ), \mathsf {Set} \right ]\end{equation}
This can be proven by translating from copresheaves to discrete opfibrations via the equivalence \(\left [ \mathsf {C}, \mathsf {Set} \right ] \simeq  \mathsf {OpFib} \left ( \mathsf {C} \right )\) sending \(P\) to \(\pi _P : \mathsf {El}\left (P\right ) \to  \mathsf {C}\). We then have:
\begin{equation}\left [ \mathsf {C}, \mathsf {Set} \right ] \downarrow  P \simeq  
\mathsf {OpFib}\left ( \mathsf {C} \right ) \downarrow  \left [ \mathsf {El}\left (P\right ) \xrightarrow {\pi _P} \mathsf {C} \right ]
\cong  \mathsf {OpFib} \left ( \mathsf {El}\left (P\right ) \right )
\simeq  \left [ \mathsf {El}\left (P\right ), \mathsf {Set} \right ]
\end{equation}
with the middle isomorphism coming from the fact that discrete opfibrations are closed under composition and left-cancellative. To extend this to \emph{models}, it suffices to show that the class of discrete opfibrations which are equivalent to the category of elements for a model are also closed under composition and left-cancellative. We defer a characterisation of such "model opfibrations" and their properties until  in the more general setting of \(\mathscr {F}\)-sketches, at which point we also prove a slice theorem for \(\mathscr {F}\)-sketches, \cref{djm-00IG}, from which \cref{jrb-001X} follows as a special case.
\end{proof}

\begin{corollary}[{The limit sketches for barrels}]\label{jrb-001Z}
\par{}
There is an equivalence between the category \(\mathsf {Cat} \downarrow  \Delta  [1]\) of barrels and models for \(\left ( \Delta  \downarrow  [1] \right )^{\mathsf {op}}\) whose sketch structure is induced by the discrete opfibration \(\mathsf {cod} : \left ( \Delta  \downarrow  [1] \right )^{\mathsf {op}} \to  \Delta ^{\mathsf {op}}\). Similarly, there is an equivalence:
\begin{equation}\mathsf {Sketch} \left ( \left ( \Delta _{\leq  3} \downarrow  [1] \right )^{\mathsf {op}}, \mathsf {Set} \right ) \simeq  \mathsf {Cat} \downarrow  \Delta  [1]\end{equation}
for the induced sketch structure on \(\left ( \Delta _{\leq  3} \downarrow  [1] \right )^{\mathsf {op}}\).

\begin{proof}[{proof of \cref{jrb-001Z}.}]\label{jrb-001Z-proof}
\par{}
The codomain projection \(\mathsf {cod} : \left ( \Delta  \downarrow  [1] \right )^{\mathsf {op}} \to  \Delta ^{\mathsf {op}}\) is the category of elements for the corepresentable \(\Delta  [1]\). By \cref{jrb-001X}, we therefore have:
  \begin{equation}\mathsf {Sketch} \left ( \left ( \Delta  \downarrow  [1] \right )^{\mathsf {op}}, \mathsf {Set} \right ) \simeq  \mathsf {Sketch} \left ( \Delta ^{\mathsf {op}}, \mathsf {Set} \right )\downarrow  \Delta  [1]\end{equation}
  The result then follows from the equivalence \(\mathsf {Sketch} \left ( \Delta ^{\mathsf {op}}, \mathsf {Set} \right ) \simeq  \mathsf {Cat}\) described in , which identifies \(\Delta  [1]\) with the walking arrow. A similar argument applied using the equivalence \(\mathsf {Sketch} \left ( \Delta _{\leq  3}^{\mathsf {op}}, \mathsf {Set} \right ) \simeq  \mathsf {Cat}\) from \cref{jrb-0032} yields:
\begin{equation}\mathsf {Sketch} \left ( \left ( \Delta _{\leq  3} \downarrow  [1] \right )^{\mathsf {op}}, \mathsf {Set} \right ) \simeq  \mathsf {Cat} \downarrow  \Delta  [1]\end{equation}\end{proof}
\end{corollary}

\begin{explication}[{Models of \(\left ( \Delta  \downarrow  [1] \right )^{\mathsf {op}}\) as bimodules}]\label{jrb-0034}
\par{}
Knowing that models for \(\left ( \Delta  \downarrow  [1] \right )^{\mathsf {op}}\) in \(\mathsf {Set}\) are equivalent to barrels, we might refer to this sketch as the sketch for barrels. However, we'll now observe that models of this sketch in general categories with limits (or at least \emph{pullbacks}) correspond more directly to \emph{internal bimodules} (also known as \emph{internal profunctors}). It is only when considering models in \(\mathsf {Set}\) that the slice theorem can be applied, obtaining the equivalence between bimodules internal to \(\mathsf {Set}\) and barrels.
\par{}
The data of a \(\Delta ^{\mathsf {op}}\)-model in a category \(\mathcal {E}\) with pullbacks involves assigning to each map \(f : [n] \to  [1]\) in \(\Delta  \downarrow  [1]\) an object \(M_f \in  \mathcal {E}\). The model condition, however, ensures that each \(M_f\) is a limit of a diagram in \(\mathcal {E}\) of the \(M_g\)'s for \(g : [e] \to  [1]\) with \(e\) \emph{elementary} in \(\Delta \). There are only five such objects \(g\), and the maps between them are as follows:

  
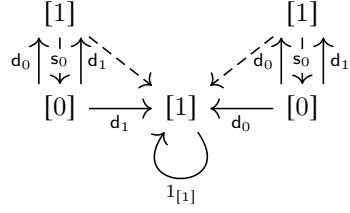
\begin{center}
\begin {tikzcd}
	{[1]} && {[1]} \\
	{[0]} & {[1]} & {[0]}
	\arrow ["{\mathsf {s}_{0}}"{description}, from=1-1, to=2-1]
	\arrow [dashed, from=1-1, to=2-2]
	\arrow [dashed, from=1-3, to=2-2]
	\arrow ["{\mathsf {s}_{0}}"{description}, from=1-3, to=2-3]
	\arrow ["{\mathsf {d}_{0}}", shift left=3, from=2-1, to=1-1]
	\arrow ["{\mathsf {d}_{1}}"', shift right=3, from=2-1, to=1-1]
	\arrow ["{\mathsf {d}_{1}}"', from=2-1, to=2-2]
	\arrow ["{1_{[1]}}", from=2-2, to=2-2, loop, in=235, out=305, distance=10mm]
	\arrow ["{\mathsf {d}_{0}}", shift left=3, from=2-3, to=1-3]
	\arrow ["{\mathsf {d}_{1}}"', shift right=3, from=2-3, to=1-3]
	\arrow ["{\mathsf {d}_{0}}", from=2-3, to=2-2]
\end {tikzcd}
\end{center}

We represent the image of this full-subcategory of the elementaries under a model \(M : \Delta ^{\mathsf {op}} \to  \mathcal {E}\) with the following suggestive notation:

  
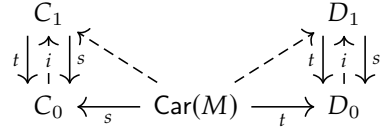
\begin{center}
\begin {tikzcd}
	{C_1} && {D_1} \\
	{C_0} & {\mathsf {Car}(M)} & {D_0}
	\arrow ["t"', shift right=3, from=1-1, to=2-1]
	\arrow ["s", shift left=3, from=1-1, to=2-1]
	\arrow ["t"', shift right=3, from=1-3, to=2-3]
	\arrow ["s", shift left=3, from=1-3, to=2-3]
	\arrow ["i"{description}, from=2-1, to=1-1]
	\arrow [dashed, from=2-2, to=1-1]
	\arrow [dashed, from=2-2, to=1-3]
	\arrow ["s", from=2-2, to=2-1]
	\arrow ["t"', from=2-2, to=2-3]
	\arrow ["i"{description}, from=2-3, to=1-3]
\end {tikzcd}
\end{center}

Now given a general map \(f : [n] \to  [1]\), which is equivalent to a non-decreasing sequence of 0's and 1's of length \(n+1\), the image \(M_f\) must be given by a certain composite of the spans \(\left \langle  s,t \right \rangle  : C_1 \to  C_0 \times  C_0\), \(\left \langle  s,t \right \rangle  : \mathsf {Car}(M) \to  C_0 \times  D_0\) and \(\left \langle  s,t \right \rangle  : D_1 \to  D_0 \times  D_0\). Any \(f : [n] \to  [1]\) can be pre-composed with the unique active \(\mathsf {a}_{n} : [1] \to  [n]\) to give one of the three possible maps \([1] \to  [1]\). This means every \(M_f\) admits some canonical morphism to either \(C_1\), \(D_1\) or \(\mathsf {Car}(M)\) depending on the image of \(f\) in \([1]\). These maps respectively give unbiased descriptions of:
\begin{enumerate}\item{}composition for the \(\mathcal {E}\)-graph  \(C_1 \rightrightarrows  C_0\)
  \item{}composition for the \(\mathcal {E}\)-graph \(D_1 \rightrightarrows  C_0\)
  \item{}actions of \(C\) and \(D\) on the right and left on \(\mathsf {Car} (M)\)\end{enumerate}\par{}
In particular, for the map \(f : [3] \to  [1]\) represented by the sequence \(0,0,1,1\), the unique active \(\mathsf {a}_{3} : [1] \to  [3]\) induces a map into \(\mathsf {Car}(M)\) of the following form:

  
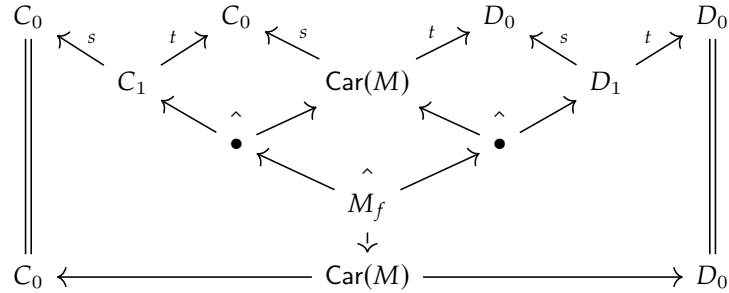
\begin{center}
 \begin {tikzcd}[row sep=0.8em, column sep=1.8em]
 	{ C_0} && { C_0} && { D_0} && { D_0} \\
 	& {C_1} && {\mathsf {Car}(M)} && {D_1} \\
 	&& \bullet  && \bullet  \\
 	&&& M_f \\
 	{C_0} &&& {\mathsf {Car}(M)} &&& {D_0}\\
 	\arrow [equals, from=1-1, to=5-1]
 	\arrow [equals, from=1-7, to=5-7]
 	\arrow ["s"',from=2-2, to=1-1]
 	\arrow ["t",from=2-2, to=1-3]
 	\arrow ["s"', from=2-4, to=1-3]
 	\arrow ["t", from=2-4, to=1-5]
 	\arrow ["s"', from=2-6, to=1-5]
 	\arrow ["t", from=2-6, to=1-7]
 	\arrow ["\lrcorner "{anchor=center, pos=0.125, rotate=135}, draw=none, from=3-3, to=1-3]
 	\arrow [from=3-3, to=2-2]
 	\arrow [from=3-3, to=2-4]
 	\arrow ["\lrcorner "{anchor=center, pos=0.125, rotate=135}, draw=none, from=3-5, to=1-5]
 	\arrow [from=3-5, to=2-4]
 	\arrow [from=3-5, to=2-6]
 	\arrow ["\lrcorner "{anchor=center, pos=0.125, rotate=135}, draw=none, from=4-4, to=2-4]
 	\arrow [from=4-4, to=3-3]
 	\arrow [from=4-4, to=3-5]
 	\arrow [dashed, from=4-4, to=5-4]
 	\arrow [from=5-4, to=5-1]
 	\arrow [from=5-4, to=5-7]
 \end {tikzcd}
  \end{center}

  exhibiting the actions of \(C\) and \(D\) on \(\mathsf {Car}(M)\).
\par{}
The functoriality of \(M\) ensures that these composition operations do indeed define \(\mathcal {E}\)-category structures on the \(\mathcal {E}\)-graphs \(C\) and \(D\), and ensures that the action of these categories on the “heteromorphism” object \(\mathsf {Car}(M)\) is functorial as required by the definition of internal bimodule.
\par{}
Similarly, models for \(\left ( \Delta _{\leq  3} \downarrow  [1] \right )^{\mathsf {op}}\) in \(\mathcal {E}\) express the biased definition for \(\mathcal {E}\)-bimodule.
\end{explication}
\par{}
We shall henceforth refer to the sketches \(\left ( \Delta ^{\mathsf {op}} \downarrow  [1] \right )\) and \(\left ( \Delta _{\leq  3}^{\mathsf {op}} \downarrow  [1] \right )\) respectively as the \emph{sketch for bimodules} and the \emph{biased sketch for bimodules}.
\subsection{From barrels to bimodules}\label{jrb-0035}\par{}\cref{jrb-001Z} provides an abstract argument for the equivalence between \(\mathsf {Set}\)-bimodules and barrels. We now supplement this with a more direct elaboration of how all the data of a \(\mathsf {Set}\)-bimodule can be recovered from the comparatively simple structure of a barrel.

\begin{definition}[{The source and target of a  barrel}]\label{ssl-000V}
\par{}
  Let \(C \colon  \mathsf {C} \to  \Delta  [1]\) be a . The \textbf{source} \(\mathsf {C}_0\) and \textbf{target} \(\mathsf {C}_1\) of \(C\) are the fibres of \(C\) over the objects \(0\) and \(1\) of \(\Delta  [1]\) respectively. These are equivalently the (domains of the)
 pullbacks of \(C\) along the functors \(0,1 \colon  \bullet  \to  \Delta  [1]\).
\end{definition}
\par{}
  The source and target of a barrel act on a set called the \textbf{carrier} which is defined below. 

\begin{definition}[{Carrier of a  barrel}]\label{ssl-0038}
\par{}
  Let \(C \colon   \mathsf {C} \to  \Delta  [1]\) be a  . Its \textbf{carrier} \(\mathsf {Car}(C)\) is the set of morphisms in \(\mathsf {C}\) which map to the walking arrow \(0 \to  1\). Elements of \(\mathsf {Car}(C)\) are called \textbf{heteromorphisms} of \(C\).
\end{definition}

\begin{explication}[{The labelling perspective for barrels}]\label{ssl-0062}
\par{}
   Given a  \(C \colon  \mathsf {C} \to  \Delta  [1]\), the objects and morphisms of \(\Delta  [1]\) \emph{label} the elements \(\mathsf {C}\) as being either part of the source, target, or carrier. We can represent this labelling visually with colours: blue for elements of the source, pink for elements of the target, and black for the heteromorphisms.
\begin{center}
	\begin {tikzcd}[row sep=1em]
		\textcolor {rgb,255:red,92;green,92;blue,214}{{c_0}} & \textcolor {rgb,255:red,214;green,92;blue,214}{{c_1}} \\
		\textcolor {rgb,255:red,179;green,179;blue,179}{0} & \textcolor {rgb,255:red,179;green,179;blue,179}{1}
		\arrow ["x", from=1-1, to=1-2]
		\arrow [draw={rgb,255:red,179;green,179;blue,179}, maps to, from=1-1, to=2-1]
		\arrow [draw={rgb,255:red,179;green,179;blue,179}, maps to, from=1-2, to=2-2]
		\arrow [draw={rgb,255:red,179;green,179;blue,179}, from=2-1, to=2-2]
	\end {tikzcd}
\end{center}\par{}
Composition in the category \(\Delta  [1]\) tells us that we can only post-compose a black arrow with a pink one, and the result will be coloured black. Similarly, we can only  \emph{pre-compose} black arrows with blue ones, the result again being a black arrow.
\begin{center}
	\begin {tikzcd}[row sep=1em]
		\textcolor {rgb,255:red,92;green,92;blue,214}{{c_0}} & \textcolor {rgb,255:red,214;green,92;blue,214}{{c_1}} & \textcolor {rgb,255:red,214;green,92;blue,214}{{c_1'}} \\
		\textcolor {rgb,255:red,179;green,179;blue,179}{0} & \textcolor {rgb,255:red,179;green,179;blue,179}{1} & \textcolor {rgb,255:red,179;green,179;blue,179}{1}
		\arrow ["x", from=1-1, to=1-2]
		\arrow ["{f\,x}", popup=1em, from=1-1, to=1-3]
		\arrow [draw={rgb,255:red,179;green,179;blue,179}, maps to, from=1-1, to=2-1]
		\arrow ["f", color={rgb,255:red,214;green,92;blue,214}, from=1-2, to=1-3]
		\arrow [draw={rgb,255:red,179;green,179;blue,179}, maps to, from=1-2, to=2-2]
		\arrow [draw={rgb,255:red,179;green,179;blue,179}, maps to, from=1-3, to=2-3]
		\arrow [draw={rgb,255:red,179;green,179;blue,179}, from=2-1, to=2-2]
		\arrow [draw={rgb,255:red,179;green,179;blue,179}, popdown=1em, from=2-1, to=2-3]
		\arrow [draw={rgb,255:red,179;green,179;blue,179}, equals, from=2-2, to=2-3]
	  \end {tikzcd}
    \qquad 
	\begin {tikzcd}[row sep=1em]
		\textcolor {rgb,255:red,92;green,92;blue,214}{{c_0}} & \textcolor {rgb,255:red,92;green,92;blue,214}{{c_0'}} & \textcolor {rgb,255:red,214;green,92;blue,214}{{c_1}} \\
		\textcolor {rgb,255:red,179;green,179;blue,179}{0} & \textcolor {rgb,255:red,179;green,179;blue,179}{0} & \textcolor {rgb,255:red,179;green,179;blue,179}{1}
		\arrow ["f", color={rgb,255:red,92;green,92;blue,214}, from=1-1, to=1-2]
		\arrow ["{x\,f}", popup=1em, from=1-1, to=1-3]
		\arrow [draw={rgb,255:red,179;green,179;blue,179}, maps to, from=1-1, to=2-1]
		\arrow ["x", from=1-2, to=1-3]
		\arrow [draw={rgb,255:red,179;green,179;blue,179}, maps to, from=1-2, to=2-2]
		\arrow [draw={rgb,255:red,179;green,179;blue,179}, maps to, from=1-3, to=2-3]
		\arrow [draw={rgb,255:red,179;green,179;blue,179}, equals,  from=2-1, to=2-2]
		\arrow [draw={rgb,255:red,179;green,179;blue,179}, popdown=1em, from=2-1, to=2-3]
		\arrow [draw={rgb,255:red,179;green,179;blue,179}, from=2-2, to=2-3]
	\end {tikzcd}
\end{center}\par{}
More generally, we can compose any finite sequence of arrows which obey the following rules:
\begin{enumerate}\item{}blue arrows are only followed by blue or black arrows
  \item{}black arrows are only followed by pink arrows
  \item{}pink arrows are only followed by pink arrows \end{enumerate}
The length and colouring of such a path is determined by the sequence of images of the objects in the path under the map \(C : \mathsf {C} \to  \Delta  [1]\). For example, a coloured path of the following sort:

  \begin{center}
	\begin {tikzcd}[row sep=1em]
	\textcolor {rgb,255:red,92;green,92;blue,214}{\bullet } &
  \textcolor {rgb,255:red,92;green,92;blue,214}{\bullet } &
  \textcolor {rgb,255:red,92;green,92;blue,214}{\bullet } &
  \textcolor {rgb,255:red,214;green,92;blue,214}{\bullet } &
  \textcolor {rgb,255:red,214;green,92;blue,214}{\bullet } &
  \textcolor {rgb,255:red,214;green,92;blue,214}{\bullet } &
  \textcolor {rgb,255:red,214;green,92;blue,214}{\bullet } &
  \\
	\textcolor {rgb,255:red,179;green,179;blue,179}{0} &
	\textcolor {rgb,255:red,179;green,179;blue,179}{0} &
	\textcolor {rgb,255:red,179;green,179;blue,179}{0} &
  \textcolor {rgb,255:red,179;green,179;blue,179}{1} &
  \textcolor {rgb,255:red,179;green,179;blue,179}{1} &
  \textcolor {rgb,255:red,179;green,179;blue,179}{1} &
  \textcolor {rgb,255:red,179;green,179;blue,179}{1} 
  \arrow [draw={rgb,255:red,92;green,92;blue,214},  from=1-1, to=1-2]
  \arrow [draw={rgb,255:red,92;green,92;blue,214},  from=1-2, to=1-3]
	\arrow [from=1-3, to=1-4]
    \arrow [draw={rgb,255:red,214;green,92;blue,214},  from=1-4, to=1-5]
    \arrow [draw={rgb,255:red,214;green,92;blue,214},  from=1-5, to=1-6]
    \arrow [draw={rgb,255:red,214;green,92;blue,214},  from=1-6, to=1-7]
		\arrow [draw={rgb,255:red,179;green,179;blue,179}, maps to, from=1-1, to=2-1]
		\arrow [draw={rgb,255:red,179;green,179;blue,179}, maps to, from=1-2, to=2-2]
		\arrow [draw={rgb,255:red,179;green,179;blue,179}, maps to, from=1-3, to=2-3]
		\arrow [draw={rgb,255:red,179;green,179;blue,179}, maps to, from=1-4, to=2-4]
		\arrow [draw={rgb,255:red,179;green,179;blue,179}, maps to, from=1-5, to=2-5]
		\arrow [draw={rgb,255:red,179;green,179;blue,179}, maps to, from=1-6, to=2-6]
		\arrow [draw={rgb,255:red,179;green,179;blue,179}, maps to, from=1-7, to=2-7]
		\arrow [draw={rgb,255:red,179;green,179;blue,179}, equals, from=2-1, to=2-2]
		\arrow [draw={rgb,255:red,179;green,179;blue,179}, equals, from=2-2, to=2-3]
		\arrow [draw={rgb,255:red,179;green,179;blue,179}, from=2-3, to=2-4]
		\arrow [draw={rgb,255:red,179;green,179;blue,179}, equals, from=2-4, to=2-5]
		\arrow [draw={rgb,255:red,179;green,179;blue,179}, equals, from=2-5, to=2-6]
		\arrow [draw={rgb,255:red,179;green,179;blue,179}, equals, from=2-6, to=2-7]
	\end {tikzcd}
\end{center}

corresponds to the sequence \(0,0,0,1,1,1,1\). Such a sequence of length \(n\) is equivalent to a map \([n] \to  [1]\) in \(\Delta \). For any map \(f : [n] \to  [1]\) in \(\Delta \), let \(C_f\) denote the set of paths of arrows in \(\mathsf {C}\) whose type is given by the corresponding sequence. A “subsequence” of \(f\) is determined by an inert map \(g : [m] \to  [n]\), which induces an inert map \(f\,g \to  f\) in \(\left ( \Delta  \downarrow  [1] \right )\). For such an inert map, there is a function \(C_f \to  C_{f\,g}\) given by restricting the sequence of arrows in \(\mathsf {C}\) to the subsequence lying over the image of \(g\).
\par{}
For any \emph{active} map \(g : [m] \to  [n]\) we get an active map \(g : f\,g \to  f\) in \(\left ( \Delta  \downarrow  [1] \right )\) and a function \(C_f \to  C_{f\,g}\) given by composing along the “image” of each edge in \(m\):

  \begin{center}
	\begin {tikzcd}[row sep=1em,column sep=1em]
	\textcolor {rgb,255:red,92;green,92;blue,214}{\bullet } &
  \textcolor {rgb,255:red,92;green,92;blue,214}{\bullet } &
  \textcolor {rgb,255:red,92;green,92;blue,214}{\bullet } &
  \textcolor {rgb,255:red,214;green,92;blue,214}{\bullet } &
  \textcolor {rgb,255:red,214;green,92;blue,214}{\bullet } &
  \textcolor {rgb,255:red,214;green,92;blue,214}{\bullet } &
  \textcolor {rgb,255:red,214;green,92;blue,214}{\bullet } &
  \\
	\textcolor {rgb,255:red,179;green,179;blue,179}{0} &
	\textcolor {rgb,255:red,179;green,179;blue,179}{0} &
	\textcolor {rgb,255:red,179;green,179;blue,179}{0} &
  \textcolor {rgb,255:red,179;green,179;blue,179}{1} &
  \textcolor {rgb,255:red,179;green,179;blue,179}{1} &
  \textcolor {rgb,255:red,179;green,179;blue,179}{1} &
\textcolor {rgb,255:red,179;green,179;blue,179}{1}
\\
	\textcolor {rgb,255:red,179;green,179;blue,179}{0} &
	\textcolor {rgb,255:red,179;green,179;blue,179}{0} &
 &
 &
  \textcolor {rgb,255:red,179;green,179;blue,179}{1} &
 &
  \textcolor {rgb,255:red,179;green,179;blue,179}{1} 
  \arrow ["a",draw={rgb,255:red,92;green,92;blue,214},  from=1-1, to=1-2]
  \arrow ["b",draw={rgb,255:red,92;green,92;blue,214},  from=1-2, to=1-3]
	\arrow ["c",from=1-3, to=1-4]
    \arrow ["d",draw={rgb,255:red,214;green,92;blue,214},  from=1-4, to=1-5]
    \arrow ["e",draw={rgb,255:red,214;green,92;blue,214},  from=1-5, to=1-6]
    \arrow ["f",draw={rgb,255:red,214;green,92;blue,214},  from=1-6, to=1-7]
		\arrow [draw={rgb,255:red,179;green,179;blue,179}, maps to, from=1-1, to=2-1]
		\arrow [draw={rgb,255:red,179;green,179;blue,179}, maps to, from=1-2, to=2-2]
		\arrow [draw={rgb,255:red,179;green,179;blue,179}, maps to, from=1-3, to=2-3]
		\arrow [draw={rgb,255:red,179;green,179;blue,179}, maps to, from=1-4, to=2-4]
		\arrow [draw={rgb,255:red,179;green,179;blue,179}, maps to, from=1-5, to=2-5]
		\arrow [draw={rgb,255:red,179;green,179;blue,179}, maps to, from=1-6, to=2-6]
		\arrow [draw={rgb,255:red,179;green,179;blue,179}, maps to, from=1-7, to=2-7]
		\arrow [draw={rgb,255:red,179;green,179;blue,179}, equals, from=2-1, to=2-2]
		\arrow [draw={rgb,255:red,179;green,179;blue,179}, equals, from=2-2, to=2-3]
		\arrow [draw={rgb,255:red,179;green,179;blue,179}, from=2-3, to=2-4]
		\arrow [draw={rgb,255:red,179;green,179;blue,179}, equals, from=2-4, to=2-5]
		\arrow [draw={rgb,255:red,179;green,179;blue,179}, equals, from=2-5, to=2-6]
\arrow [draw={rgb,255:red,179;green,179;blue,179}, equals, from=2-6, to=2-7]
		\arrow [draw={rgb,255:red,179;green,179;blue,179}, maps to, from=3-1, to=2-1]
		\arrow [draw={rgb,255:red,179;green,179;blue,179}, maps to, from=3-2, to=2-2]
		\arrow [draw={rgb,255:red,179;green,179;blue,179}, maps to, from=3-5, to=2-5]
		\arrow [draw={rgb,255:red,179;green,179;blue,179}, maps to, from=3-7, to=2-7]
		\arrow [draw={rgb,255:red,179;green,179;blue,179}, equals, from=3-1, to=3-2]
		\arrow [draw={rgb,255:red,179;green,179;blue,179}, from=3-2, to=3-5]
		\arrow [draw={rgb,255:red,179;green,179;blue,179}, equals, from=3-5, to=3-7]
\end {tikzcd}
\quad  $\leadsto $ \quad 
	\begin {tikzcd}[row sep=1.4em,column sep=1.4em]
	\textcolor {rgb,255:red,92;green,92;blue,214}{\bullet } &
  \textcolor {rgb,255:red,92;green,92;blue,214}{\bullet } &
  \textcolor {rgb,255:red,214;green,92;blue,214}{\bullet } &
  \textcolor {rgb,255:red,214;green,92;blue,214}{\bullet } &
  \\
	\textcolor {rgb,255:red,179;green,179;blue,179}{0} &
	\textcolor {rgb,255:red,179;green,179;blue,179}{0} &
  \textcolor {rgb,255:red,179;green,179;blue,179}{1} &
  \textcolor {rgb,255:red,179;green,179;blue,179}{1} &
  \arrow ["a",draw={rgb,255:red,92;green,92;blue,214},  from=1-1, to=1-2]
  \arrow ["dcb", from=1-2, to=1-3]
    \arrow ["fe",draw={rgb,255:red,214;green,92;blue,214},  from=1-3, to=1-4]
		\arrow [draw={rgb,255:red,179;green,179;blue,179}, maps to, from=1-1, to=2-1]
		\arrow [draw={rgb,255:red,179;green,179;blue,179}, maps to, from=1-2, to=2-2]
		\arrow [draw={rgb,255:red,179;green,179;blue,179}, maps to, from=1-3, to=2-3]
		\arrow [draw={rgb,255:red,179;green,179;blue,179}, maps to, from=1-4, to=2-4]
		\arrow [draw={rgb,255:red,179;green,179;blue,179}, equals, from=2-1, to=2-2]
		\arrow [draw={rgb,255:red,179;green,179;blue,179}, from=2-2, to=2-3]
		\arrow [draw={rgb,255:red,179;green,179;blue,179}, equals, from=2-3, to=2-4]
\end {tikzcd}
\end{center}

The map \(f \mapsto  C_f\) thus extends to a functor \(\left ( \Delta  \downarrow  [1] \right )^{\mathsf {op}} \to  \mathsf {Set}\) which is seen to satisfy the model condition upon observing that each \(C_f\) is obtained by “gluing together” copies of \(\mathsf {C}_0^\to \), \(\mathsf {Car}(C)\) and \(\mathsf {C}_1^\to \) in the appropriate way. We thus recover the \(\mathsf {Set}\)-bimodule corresponding to a given barrel.
\end{explication}

\begin{example}[{Hom barrels}]\label{ssl-000O}
\par{}
  Given a category \(\mathsf {C}\), its \textbf{hom barrel} is the barrel \(\mathsf {C} \times  \Delta  [1] \to  \Delta  [1]\) give by projection onto the second component. Note that the heteromorphisms of this barrel are precisely the homomorphisms of \(\mathsf {C}\), whence the name.
\end{example}

\begin{remark}[{Source, target and hom from the sketch perspective}]\label{jrb-003E}
\par{}
The \emph{barrel} perspective on bimodules makes it particularly easy to define the source, target and hom of a bimodule, but these concepts also admit natural descriptions from the sketch perspective. First, note that the maps \(\mathsf {d}_{0}, \mathsf {d}_{1} : [0] \to  [1]\) and \(\mathsf {s}_{0} : [1] \to  [0]\) in \(\Delta \) induce maps:

  \begin{center}
\begin {tikzcd}[column sep=5em]
{\left ( \Delta  \downarrow  [1] \right )^{\mathsf {op}}} & {\left ( \Delta  \downarrow  [0] \right )^{\mathsf {op}}}
\arrow ["{\Delta  \downarrow  \mathsf {s}_{0}}"{description}, from=1-1, to=1-2]
\arrow ["{\Delta  \downarrow  \mathsf {d}_{0}}"', popup=1em, from=1-2, to=1-1]
\arrow ["{\Delta  \downarrow  \mathsf {d}_{1}}", popdown=1em, from=1-2, to=1-1]
\end {tikzcd}
\end{center}

which are moreover maps of sketches by virtue of being maps in the slice over \(\Delta  ^{\mathsf {op}}\). Precomposition by these sketch maps therefore induces functors between the categories of models for these sketches. Recalling that models for \(\left ( \Delta  \downarrow  [0] \right )^{\mathsf {op}} \cong  \Delta  ^{\mathsf {op}}\) are categories, we observe that the precomposition maps \(\left ( \Delta  \downarrow  \mathsf {d}_{0} \right )^*\) and \(\left ( \Delta  \downarrow  \mathsf {d}_{0} \right )^*\) respectively correspond to the target and source of the bimodule, whereas precomposing by \({\Delta  \downarrow  \mathsf {s}_{0}}\) sends a category to its hom bimodule.
\end{remark}
\section{\(\mathscr {F}\)-sketches, double barrels and loose bimodules}\label{jrb-0036}\par{}
We now recapitulate the story from \cref{jrb-002Z} but with \emph{double} barrels and \emph{loose} bimodules. The role of limit sketches in the previous section will here be played by \(\mathscr {F}\)-sketches. Aside from the proliferation of coherence conditions in this higher-dimensional setting, we also contend with the fact that these higher-dimensional analogues are less established in the literature. In particular, the notions of \emph{pseudo-models} and \emph{model fibrations} for \(\mathscr {F}\)-sketches are, to our knowledge, new. We must therefore proceed more slowly as we establish the technical definitions and lemmas required to adapt the results of \cref{jrb-002Z}.
\subsection{\(\mathscr {F}\)-sketches and pseudo-models}\label{djm-00I9}\par{}
One approach to obtaining a sketch for loose bimodules would be to vastly augment the sketch for bimodules to capture all the higher-dimensional coherence data and composition. This approach obscures the fact that the weakness and coherence data for loose bimodules are canonical; they arise from weakening what it means to \emph{be} a model, rather than the sketch itself.
\par{}	
Our approach is instead to introduce the weaker notion of \emph{pseudo-models} for sketches, which faithfully captures the pseudo-structure and coherences expected of categorified models of ordinary sketches. Uniformly weakening all coherence conditions for a model often fails to produce the correct notion of pseudo-model, so we must instead work with the \emph{\(\mathscr {F}\)-sketches} of \cite{arkor-2024-enhanced}, which allow us to precisely specify the \emph{inert} operations — whose laws must hold strictly — and the \emph{active} operations whose laws may hold merely up to coherent isomorphisms.
\par{}
To justify this definition, we will prove in detail that the pseudo-models in \(\mathcal {C}\mathsf {at}\) of \(\Delta ^{\mathsf {op}}\), endowed with a certain \(\mathscr {F}\)-sketch structure, are unbiased pseudo-double categories in the sense of Definition 5.2.1 of Leinster's \emph{Higher Categories, Higher Operads} \cite{leinster-2004-higher}. This is also a crucial step in proving the equivalence:
	\begin{equation}
		\mathcal {D}\mathsf {bl} \downarrow  \mathbb {L}\mathsf {oose} \simeq  \mathcal {P}\mathsf {sMod}^{\mathsf {p}}_{\mathscr {F}} \left (\left ( \Delta  \downarrow  [1] \right )^{\mathsf {op}},\mathcal {C}\mathsf {at}\right )
	\end{equation}
	between the 2-categories of double barrels and of pseudo-models for (an \(\mathscr {F}\)-sketch structure on) \(\left ( \Delta  \downarrow  [1] \right )^{\mathsf {op}}\) in \(\mathcal {C}\mathsf {at}\). This provides a satisfying parallel with the fact that the \emph{strict} models of \(\left ( \Delta  \downarrow  [1] \right )^{\mathsf {op}}\) in \(\mathsf {Set}\) are ordinary bimodules.
\par{}
From this result we obtain an equivalence between double barrels and a natural notion of loose bimodule in the sense of pseudo-bimodule internal to \(\mathcal {C}\mathsf {at}\). Indeed, we present the 2-category of pseudo-models for \(\left ( \Delta  \downarrow  [1] \right )^{\mathsf {op}}\) in a 2-category \(\mathcal {E}\) with pullbacks as a natural definition of (unbiased) pseudo-bimodule internal to \(\mathcal {E}\).
\par{}
We begin by recalling the definitions of \(\mathscr {F}\)-category and \(\mathscr {F}\)-sketch — introduced respectively in \cite{lack-2012-enhanced} and Definition 5.8 of \cite{arkor-2024-enhanced} — though with a change of terminology: to avoid conflict between the use of the terms “tight” and “loose” in \cite{arkor-2024-enhanced} and our use of the terms for double categories we will substitute the term “inert” for “tight”. What they call “loose” morphisms we simply call “morphisms”.
 
\begin{definition}[{\(\mathscr {F}\)-categories and \(\mathscr {F}\)-sketches}]\label{djm-00IA}
\par{}
	An \(\mathscr {F}\)-category \(\mathcal {A}\) consists of an \textbf{underlying} \(2\)-category \(\mathcal {A}\) together with a class of marked \(1\)-cells called \textbf{inerts}, containing all identities and closed under composition. We denote by \({\mathcal {A}}_{\mathsf {inert}}\) the locally-full sub-2-category of \(\mathcal {A}\) spanned by the inerts. When all 1-cells are inert, we say the \(\mathscr {F}\)-category is \textbf{chordate}. When only the identity 1-cells are inert, we the \(\mathscr {F}\)-category is \textbf{inchordate}.
\par{}
A \textbf{cone} in an \(\mathscr {F}\)-category \(\mathcal {A}\) is a strict 2-functor \(\mathcal {J} \to  a \downarrow  {\mathcal {A}}_{\mathsf {inert}}\) for some \(a \in  \mathcal {A}\) called the \textbf{vertex} of the cone.
An \textbf{\(\mathscr {F}\)-sketch} \(\mathcal {A}\) is an \(\mathscr {F}\)-category together with a collection \(\mathsf {Cone}_{\mathcal {A}}\) of cones, referred to as the \textbf{marked cones} of \(\mathcal {A}\).
\end{definition}

\begin{remark}[{A more general notion of \(\mathscr {F}\)-sketch}]\label{djm-00IB}
\par{}
Our notion of cone for \hyperref[djm-00IA]{\(\mathscr {F}\)-sketches} \(\mathcal {A}\) requires not only that the projections be inert, but also the underlying diagram in \(\mathcal {A}\) be inert. 
This is less general than the definition for \(\mathscr {F}\)-sketch given in \cite{arkor-2024-enhanced} which allows for arbitrary marked \emph{cylinders} \(W \Rightarrow   \mathcal {A} (a ,S -)\) where \(S : \mathcal {J} \to  \mathcal {A}\) and \(W : \mathcal {J} \to  \mathscr {F}\) are \(\mathscr {F}\)-functors from a small \(\mathscr {F}\)-category. Our definition corresponds to requiring that \(\mathcal {J}\) be chordate, and \(W = \Delta  I\), where \(I\) is the (terminal) unit in \(\mathscr {F}\). However, for \(\mathscr {F}\)-sketches where the inerts form the left class of an orthogonal factorisation system, as is the case for those arising from algebraic patterns (see \hyperref[djm-00HQ]{below}), there is no loss of generality in assuming \(\mathcal {J}\) is chordate when \(W = \Delta  I\).
\end{remark}

\begin{definition}[{pseudo-\(\mathscr {F}\)-functor}]\label{djm-00I2}
\par{}
	Let \(\mathcal {A}\) and \(\mathcal {B}\) be \hyperref[djm-00IA]{\(\mathscr {F}\)-categories}. A \textbf{pseudo-\(\mathscr {F}\)-functor} \(F : \mathcal {A} \to  \mathcal {B}\) is a pseudofunctor \((F, F_0, F_2)\) on their underlying 2-categories which preserves inerts (\(Fi\) is inert when \(i\) is inert) and which are strict with respect to inerts in the following sense:
	\begin{enumerate}\item{}
			the unitors \(F_0 : 1_{Fx} \Rightarrow  F(1_x)\) are identities.
		
		\item{}
			the compositors \(F_2 : Fg \; Ff \Rightarrow  F(g \; f)\) are identities whenever either \(g\) or \(f\) is inert.
		\end{enumerate}\par{}
For \(w \in  \left \{ \mathsf {s(trict)},\mathsf {ps(eudo)}, \mathsf {l(a)x}, \mathsf {c(ola)x} \right \}\) we denote by \(\mathcal {P}\mathsf {s}^{w}_{\mathscr {F}}(\mathcal {A}, \mathcal {B})\) the 2-category of pseudo-\(\mathscr {F}\)-functors, \(w\)-transformations which are strict on inerts — henceforth \textbf{\(w\)-\(\mathscr {F}\)-transformations} — and modifications between them.
\end{definition}

\begin{remark}[{Strictness on inerts is insufficient}]\label{jrb-0037}
\par{}
Any pseudo-\(\mathscr {F}\)-functor \(F : \mathcal {A} \to  \mathcal {B}\) restricts to a \emph{strict} 2-functor \(F_{\mathsf {inert}} : \mathcal {A}_{\mathsf {inert}} \to  \mathcal {B}_{\mathsf {inert}}\), but that property is not sufficient for a pseudo-functor to be a pseudo-\(\mathscr {F}\)-functor. For example, any unital pseudofunctor from the \(\mathscr {F}\)-category freely generated by an inert \(A \to  B\) and a non-inert \(B \to  C\) has the strict-on-inerts property, but the only pseudo-\(\mathscr {F}\)-functors in this case are \emph{strict} 2-functors.
\end{remark}

\begin{definition}[{Pseudo-model of an \(\mathscr {F}\)-sketch}]\label{djm-00I5}
\par{}
	Let \(\mathcal {A}\) and \(\mathcal {K}\) be \hyperref[djm-00IA]{\(\mathscr {F}\)-sketches}. A \textbf{pseudo-model} of \(\mathcal {A}\) in \(\mathcal {K}\) is a \hyperref[djm-00I2]{pseudo-\(\mathscr {F}\)-functor} \(M : \mathcal {A} \to  \mathcal {K}\) such that for each marked cone \(S : \mathcal {J} \to  a \downarrow  {\mathcal {A}}_{\mathsf {inert}}\), the composite:
\begin{equation}\mathcal {J} \xrightarrow {S} a \downarrow  {\mathcal {A}}_{\mathsf {inert}} \xrightarrow {a \downarrow  {M}_{\mathsf {inert}}} Ma \downarrow  {\mathcal {K}}_{\mathsf {inert}}\end{equation} is marked in \(\mathcal {K}\).
\par{}
For a 2-category \(\mathcal {B}\), a model of \(\mathcal {A}\) in \(\mathcal {B}\) is defined to be a model of \(\mathcal {A}\) in the \(\mathscr {F}\)-sketch given by viewing \(\mathcal {B}\) as a chordate \(\mathscr {F}\)-category \(\mathcal {B}\) with all limit cones as marked cones.
\par{}
We denote the 2-category of pseudo-models, \(w\)-\(\mathscr {F}\)-transformations, and modifications by \(\mathcal {P}\mathsf {sMod}^{w}_{\mathscr {F}}(\mathcal {A}, \mathcal {K}) \hookrightarrow  \mathcal {P}\mathsf {s}^{w}_{\mathscr {F}}(\mathcal {A}, \mathcal {K})\).
\end{definition}

\begin{remark}[{(Co)lax models}]\label{jrb-0038}
\par{}
There is an obvious extension of these definitions to \emph{(co)lax-\(\mathscr {F}\)-functor} and \emph{(co)lax model} for an \(\mathscr {F}\)-sketch. We do not address these here, but simply note that the strategy we will use to establish results for the pseudo-\(\mathscr {F}\)-functors and pseudo-models \emph{does not} immediately extend along these generalisations.
\end{remark}

\begin{lemma}[{pre-composition by pseudo-\(\mathscr {F}\)-functors}]\label{djm-00IC}
\par{}
	Any \hyperref[djm-00I2]{pseudo-\(\mathscr {F}\)-functor} \(F : \mathcal {A} \to  \mathcal {B}\) induces a 2-functor
	\begin{equation}
		\mathcal {P}\mathsf {s}^{\mathsf {p}}_{\mathscr {F}}(\mathcal {B}, \mathcal {K}) \to  \mathcal {P}\mathsf {s}^{\mathsf {p}}_{\mathscr {F}}(\mathcal {A}, \mathcal {K})
	\end{equation}
	by pre-composition. If \(F\) is furthermore a \emph{\hyperref[djm-00I5]{pseudo-model}} between \(\mathscr {F}\)-sketches, then it induces a 2-functor
	\begin{equation}
		\mathcal {P}\mathsf {sMod}^{\mathsf {p}}_{\mathscr {F}}(\mathcal {B}, \mathcal {K}) \to  \mathcal {P}\mathsf {sMod}^{\mathsf {p}}_{\mathscr {F}}(\mathcal {A},\mathcal {K}).
	\end{equation}
\begin{proof}[{proof of \cref{djm-00IC}.}]\label{djm-00IC-proof}
\par{}
	This follows straightforwardly by the 2-out-of-3 property for identities.
\end{proof}
\end{lemma}

\begin{example}[{The limit sketch for double categories}]\label{jrb-0029}
\par{}The opposite of the simplex category is given the structure of a (locally-discrete) \(\mathscr {F}\)-category by declaring its \emph{inerts} to be those 1-cells which preserve distances. The marked cones are given, as in ,
by the functors:
\begin{equation}[n] \downarrow _{\mathsf {inert}} \mathsf {el}  \hookrightarrow  [n] \downarrow  \Delta ^{\mathsf {op}}\end{equation}
for each \([n] \in  \Delta ^{\mathsf {op}}\), viewing \([n] \downarrow _{\mathsf {inert}} \mathsf {el}\) as a chordate, locally-discrete \(\mathscr {F}\)-category.
\par{}
We call this the \emph{limit sketch for double categories} because, as we'll see in \cref{djm-00IE}, pseudo-models for this \(\mathscr {F}\)-sketch in \(\mathsf {Cat}\) are indeed (pseudo) double categories. More generally, pseudo-models for this \(\mathscr {F}\)-sketch in an arbitrary 2-category \(\mathcal {E}\) with pullbacks provides a satisfactory notion of internal \emph{pseudo-category}.
\end{example}

\begin{remark}[{\(\mathscr {F}\)-sketches from algebraic patterns}]\label{djm-00HQ}
\par{}
This \(\mathscr {F}\)-sketch structure on \(\Delta ^{\mathsf {op}}\) is \hyperref[jrb-001V]{again} a special case of a more general construction of an \(\mathscr {F}\)-sketch structure for an \emph{algebraic pattern} (\cite{chu-2019-homotopycoherent}). An algebraic pattern \(\left ( \mathsf {P},{\mathsf {P}}_{\mathsf {inert}}, \mathsf {P}_{\mathsf {el}} \right )\) can be equipped with the structure of an \(\mathscr {F}\)-sketch as follows:
	\begin{enumerate}\item{}
			The underlying 2-category is \(\mathsf {P}\), viewed as a locally-discrete 2-category.
		
		\item{}
			The inert morphisms are those which are inert in the algebraic pattern sense
		
    \item{}The marked cones are the inclusions \(p {\downarrow }_{\mathsf {inert}} \mathsf {P}_{\mathsf {el}} \hookrightarrow  p \downarrow  {\mathsf {P}}_{\mathsf {inert}}\) for each \(p \in  \mathsf {P}\). These are precisely the cones which encode the \emph{Segal condition} for functors out of the algebraic pattern.
	\end{enumerate}\end{remark}

\begin{theorem}[{pseudo-models of \(\Delta ^{\mathsf {op}}\) are pseudo-double categories}]\label{djm-00IE}
\par{}
The 2-category \(\mathcal {P}\mathsf {sMod}^{\mathsf {p}}_{\mathscr {F}}(\Delta ^{\mathsf {op}}, \mathcal {C}\mathsf {at})\) of \hyperref[djm-00I5]{pseudo-models} of the \(\mathscr {F}\)-sketch \(\Delta ^{\mathsf {op}}\) is 2-equivalent to the 2-category \(\mathcal {D}\mathsf {bl}\) of \emph{unbiased (pseudo) double categories}, in the sense of Definition 5.2.1 of \cite{leinster-2004-higher}, with pseudo-functors and tight transformations as the 1-cells and 2-cells.

\begin{proof}[{partial proof of \cref{djm-00IE}.}]\label{djm-00IE-proof}
\par{}
Our proof will proceed as follows:
\begin{enumerate}\item{}we define a map on objects \(A : \mathcal {P}\mathsf {sMod}^{\mathsf {p}}_{\mathscr {F}}(\Delta ^{\mathsf {op}}, \mathcal {C}\mathsf {at}) \to  \mathcal {D}\mathsf {bl}\)
  \item{}we define another map on objects \(B : \mathcal {D}\mathsf {bl} \to  \mathcal {P}\mathsf {sMod}^{\mathsf {p}}_{\mathscr {F}}(\Delta ^{\mathsf {op}}, \mathcal {C}\mathsf {at})\) and observe that \(\left ( AB \right ) M \cong  M\) for any pseudo double category \(M\) (and thus that \(A\) is essentially surjective)
  \item{}we extend the map on objects \(A\) into a 2-functor and observe that it's 2-fully-faithful.\end{enumerate}
Each stage of this proof involves some fairly laborious checking of coherence conditions, some details of which we omit where we feel that both the method and outcome of fully elaborating the proof are unenlightening. The central theme of this proof is that the action of \(B(\mathbb {D})\) on objects, 1-cells and 2-cells of \(\Delta ^{\mathsf {op}}\) can be specified with very little data, given the stringency of the pseudo-model condition on pseudo-\(\mathscr {F}\)-functors and the active–inert factorisation system on \(\Delta \). Here we give only the first stage of the proof — the map on objects \(A : \mathcal {P}\mathsf {sMod}^{\mathsf {p}}_{\mathscr {F}}(\Delta ^{\mathsf {op}}, \mathcal {C}\mathsf {at}) \to  \mathcal {D}\mathsf {bl}\) — which hopefully suffices to demonstrate the compatibility between the data involved in the objects of each 2-category. The proof is given in full in the appendix.
\par{}
We will define the mapping from pseudo-models of \(\Delta ^{\mathsf {op}}\) to Leinster's notion of pseudo double categories by explicitly exhibiting the structure and conditions specified in Definition 5.2.1 of \cite{leinster-2004-higher}.
\par{}
Let \(M : \Delta ^{\mathsf {op}} \to  \mathcal {C}\mathsf {at}\) be a pseudo-model of \(\Delta ^{\mathsf {op}}\), with \(M_n\) denoting the image under \(M\) of \([n] \in  \Delta \). We associate to \(M\) the pseudo double category defined as follows:
\begin{enumerate}\item{}The \emph{diagram} in \(\mathcal {C}\mathsf {at}\) is given by:

  \begin{center}
\begin {tikzcd}[row sep=tiny]
	& {M_1} \\
	{M_0} && {M_0}
	\arrow ["{M{\mathsf {d}_{1}}}"', from=1-2, to=2-1]
	\arrow ["{M{\mathsf {d}_{0}}}", from=1-2, to=2-3]
\end {tikzcd}
\end{center}

The model condition (corresponding to the Segal condition for the algebraic pattern \(\Delta ^{\mathsf {op}}\)) requires that \(M_n\) be isomorphic to the \(n\)-fold composite, denoted
for short by \(M_1^{(n)}\), of this span. We denote this isomorphism by \(\theta _n : M_1^{(n)} \to  M_n\).
  
  \item{}For each \(n \geq  0\), we define the external composition of \(n\) loose arrows, \(\mathsf {comp}_n : M_1^{(n)} \to  M_1\), to be:
  \begin{equation}M_1^{(n)} \xrightarrow {\theta _n} M_n \xrightarrow {M \mathsf {a}_n} M_1\end{equation}
  where \(\mathsf {a}_n : [1] \to  [n]\) is the unique such (active) map in \(\Delta ^{\mathsf {op}}\). The maps \(\mathsf {comp}_n\) commute with the projections down to \(M_0\) by the fact that each triangle in the following diagram commutes, where \(\pi _j : M_1^{(n)} \to  M_1\) is the canonical projection onto the \(j^{\text {th}}\) component:

  \begin{center}
\begin {tikzcd}
	& {M_1^{(n)}} \\
	{M_0} & {M_n} & {M_0} \\
	& {M_1}
	\arrow ["{M\mathsf {d}_{1}\;\pi _1}"', from=1-2, to=2-1]
	\arrow ["{\theta _n}", from=1-2, to=2-2]
	\arrow ["{M\mathsf {d}_{0}\;\pi _n}", from=1-2, to=2-3]
	\arrow ["{M\mathsf {d}_{n}}"', from=2-2, to=2-1]
	\arrow ["{M\mathsf {d}_{0}}", from=2-2, to=2-3]
	\arrow ["{M\mathsf {a}_n}"{pos=0.3}, from=2-2, to=3-2]
	\arrow ["{M\mathsf {d}_{1}}", from=3-2, to=2-1]
	\arrow ["{M\mathsf {d}_{0}}"', from=3-2, to=2-3]
\end {tikzcd}
\end{center}
\item{}The \emph{double-sequence maps} are invertible 2-cells of the following form (where \(\diamond \) indicates composition of spans):

  \begin{center}
\begin {tikzcd}
	{M^{(k_1 + \dots  + k_n)}_1} & {M_1^{(k_1)} \diamond \dots  \diamond  M_1^{(k_n)}} & [+4em] {M_1^{(n)}} & [+1em] {M_1}
	\arrow ["\cong ", from=1-1, to=1-2]
	\arrow [""{name=0, anchor=center, inner sep=0},"{\mathsf {comp}_{k_1 + \dots  + k_n}}"', popdown=2em, from=1-1, to=1-4]
	\arrow [""{name=1, anchor=center, inner sep=0},"{\mathsf {comp}_{k_1} \diamond  \dots  \diamond  \mathsf {comp}_{k_n}}", from=1-2, to=1-3]
\arrow ["{\mathsf {comp}_n}", from=1-3, to=1-4]
\arrow ["\gamma ", Rightarrow, shorten <=8pt, shorten >=8pt, from=0|-1, to=0]
\end {tikzcd}
\end{center}

To exhibit such 2-cells, we first observe that the following diagram commutes:

  \begin{center}
\begin {tikzcd}
	{M^{(k_1 + \dots  + k_n)}_1} & {M_1^{(k_1)} \diamond \dots  \diamond  M_1^{(k_n)}} & [+4em] {M_1^{(n)}} \\
	{M_{k_1 + \dots  + k_n}} && {M_n}
	\arrow ["\cong ", from=1-1, to=1-2]
	\arrow ["\theta "', from=1-1, to=2-1]
	\arrow ["{\mathsf {comp}_{k_1} \diamond  \dots  \mathsf {comp}_{k_n}}", from=1-2, to=1-3]
	\arrow ["\theta "', from=1-3, to=2-3]
	\arrow [""{name=0, anchor=center, inner sep=0}, "{M(\mathsf {a}_{k_1}+\dots  + \mathsf {a}_{k_n})}"', from=2-1, to=2-3]
	\arrow ["\circlearrowright "{marking, allow upside down}, draw=none, from=1-2, to=0]
\end {tikzcd}
\end{center}

because it commutes under post-composition by the various projections \(M \mathsf {i}_j : M_n \to  M_1\) of the universal cone under \(M_n\) where \(\mathsf {i}_j : [1] \to  [n]\) is the inert map with \(\mathsf {i}_j (0) = j-1\) (recall that more generally \(\mathsf {i}_j : [m] \to  [n]\) denotes the inert map with the same property). We observe this is true as follows:
\begin{equation}\begin {aligned}
M \mathsf {i}_j \; M(\mathsf {a}_{k_1}+\dots  + \mathsf {a}_{k_n}) \; \theta 
&= \; M \mathsf {a}_{k_j} \; M \mathsf {i}_{k_1 + \dots  + k_{j-1}} \; \theta  \\
&= \; M \mathsf {a}_{k_j} \; \theta  \; M^{(\mathsf {i}_{k_1 + \dots  + k_{j-1}})}_1 \\
&= \; \mathsf {comp}_{k_j}\; M^{(\mathsf {i}_{k_1 + \dots  + k_{j-1}})}_1 \\
&= \; \pi _j \; \left ( \mathsf {comp}_{k_1} \diamond  \dots  \diamond  \mathsf {comp}_{k_n} \right ) \\
&= \; M \mathsf {i}_j \; \theta  \; \left ( \mathsf {comp}_{k_1} \diamond  \dots  \diamond  \mathsf {comp}_{k_n} \right ) \\
\end {aligned}\end{equation}
Consequently, the following diagram commutes:

  \begin{center}
\begin {tikzcd}
	{M^{(k_1 + \dots  + k_n)}_1} & {M_1^{(k_1)} \diamond \dots  \diamond  M_1^{(k_n)}} & [+4em] {M_1^{(n)}} & [+1em]{M_1} \\
	{M_{k_1 + \dots  + k_n}} && {M_n}
	\arrow ["\cong ", from=1-1, to=1-2]
	\arrow ["\theta "', from=1-1, to=2-1]
	\arrow ["{\mathsf {comp}_{k_1} \diamond  \dots  \diamond  \mathsf {comp}_{k_n}}", from=1-2, to=1-3]
	\arrow ["{\mathsf {comp}_n}", from=1-3, to=1-4]
	\arrow ["\theta "', from=1-3, to=2-3]
	\arrow [""{name=0, anchor=center, inner sep=0}, "{M(\mathsf {a}_{k_1}+\dots  + \mathsf {a}_{k_n})}"', from=2-1, to=2-3]
	\arrow ["{M\mathsf {a}_n}"', rounded corners, to path={ -| (\tikztotarget )[pos=0.25]\tikztonodes }, from=2-3, to=1-4]
	\arrow ["\circlearrowright "{marking, allow upside down}, draw=none, from=1-2, to=0]
\end {tikzcd}
\end{center}

This means that a 2-cell \(\gamma \) of the form shown above is equivalent to a 2-cell of the form:
\begin{equation}M\mathsf {a}_n \;M(\mathsf {a}_{k_1}+\dots  + \mathsf {a}_{k_n}) \Rightarrow   M \left ( \mathsf {a}_{k_1 + \dots  + k_n} \right )\end{equation} which we obtain as the laxator for \(M\) witnessing the composition:
\begin{equation}(\mathsf {a}_{k_1}+\dots  + \mathsf {a}_{k_n})\; \mathsf {a}_n = \mathsf {a}_{k_1 + \dots  + k_n}\end{equation}
This 2-cell is invertible, as required. To see that these 2-cells are moreover \emph{globular}, i.e. vertical with respect to \(M\mathsf {d}_{0}, M\mathsf {d}_{1} : M_1 \to  M_0\), we appeal to the associative property of \(M\) which allows us to perform the following manipulation of string diagrams for the 2-cell given by post-whiskering the pseudo compositor \(\mu \) with \(M\mathsf {d}_{0}\):
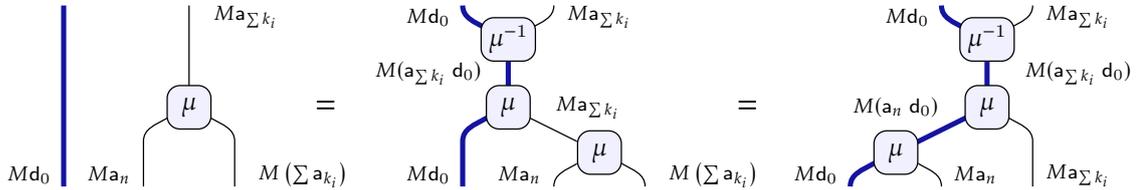
\begin{equation*}
    \tikzfig[1.2]{string-diagram}
\end{equation*}
Since any compositor with an inert input (indicated in bold and coloured blue) is an identity, the entire 2-cell must be an identity as required. The same argument applies for post-whiskering by \(M \mathsf {d}_{1}\).

\item{}The \(\iota \) maps are given by an invertible 2-cell of the form:

  \begin{center}
\begin {tikzcd}
	{M_1} & |[alias=middle]| {M_1^{(1)}} & {M_1}
	\arrow ["{\theta _1^{-1}}", from=1-1, to=1-2]
	\arrow [""{name=1, anchor=center, inner sep=0},equals, popdown=2em, from=1-1, to=1-3]
	\arrow ["{\mathsf {comp}_1}", from=1-2, to=1-3]
	\arrow ["{\iota }"{outer sep=3pt}, Rightarrow, shorten <=4pt, shorten >=4pt,from=middle, to=1]
  \end {tikzcd}
\end{center}

which we take to be the identity, noting that \(\mathsf {comp}_1\;\theta _1^{-1} = M1_{[1]}\;\theta _1 \;\theta _1^{-1} = 1_{M_1}\) by the fact that \hyperref[djm-00I2]{\(M\) is strictly unital}.
\end{enumerate}
The associativity and identity coherence axioms for \(\gamma \) and \(\iota \) correspond directly to those for the pseudo-\(\mathscr {F}\)-functor \(M\).
\end{proof}
\end{theorem}

\begin{remark}[{Lax model morphisms are lax double functors}]\label{jrb-0023}
\par{}
Tracing through the action on 1-cells in the proof of \cref{djm-00IE} reveals that the only dependence on the naturality cells being invertible comes from requiring that the corresponding \(F \mathsf {comp}_n\) cells be invertible as well. Weakening the invertibility conditions on both in each possible direction we instead obtain the 2-equivalences:
\begin{equation}
\mathcal {P}\mathsf {sMod}^{\mathsf {lx}}_{\mathscr {F}}(\Delta ^{\mathsf {op}}, \mathcal {C}\mathsf {at}) \simeq  \mathcal {D}\mathsf {bl}^{\mathsf {lx}}
\qquad 
\mathcal {P}\mathsf {sMod}^{\mathsf {cx}}_{\mathscr {F}}(\Delta ^{\mathsf {op}}, \mathcal {C}\mathsf {at}) \simeq  \mathcal {D}\mathsf {bl}^{\mathsf {cx}}
\end{equation}\end{remark}

\begin{remark}[{The biased sketch for pseudo-categories}]\label{jrb-003J}
\par{}
Given that pseudo-models of the sketch for categories in \(\mathcal {C}\mathsf {at}\) are pseudo double categories, one might suppose that pseudo-models of the \emph{biased} sketch for categories in \(\mathcal {C}\mathsf {at}\) are \emph{biased} pseudo double categories. This is, however, not true. The reason is essentially that the pentagon law for the associators of a pseudo double category involves an equality between 2-cells between morphisms which represent the various ways of composing \emph{four} composable arrows, whereas a pseudo-model for \(\Delta _{\leq  3}^{\mathsf {op}}\) expresses at most ternary composites. More concretely, given four composable loose morphisms:
\begin{equation}a  \mathrel {\overset {f}{\mathrel {\mkern 3mu\vcenter {\hbox {$\shortmid $}}\mkern -10mu{\to }}}}  b  \mathrel {\overset {g}{\mathrel {\mkern 3mu\vcenter {\hbox {$\shortmid $}}\mkern -10mu{\to }}}}  c  \mathrel {\overset {h}{\mathrel {\mkern 3mu\vcenter {\hbox {$\shortmid $}}\mkern -10mu{\to }}}}  d \mathrel {\overset {k}{\mathrel {\mkern 3mu\vcenter {\hbox {$\shortmid $}}\mkern -10mu{\to }}}} e\end{equation}
the pentagon law says that the following composites of associators are equal:

  \begin{center}
\begin {tikzcd}
	{((kh)g)f} & {(k(hg))f} & {k((hg)f)} \\
	{(kh)(gf)} && {k(h(gf))}
	\arrow ["{\alpha \,f}", from=1-1, to=1-2]
	\arrow ["\alpha "', from=1-1, to=2-1]
	\arrow ["\alpha ", from=1-2, to=1-3]
	\arrow ["{k \,\alpha }", from=1-3, to=2-3]
	\arrow [""{name=0, anchor=center, inner sep=0}, "\alpha "', from=2-1, to=2-3]
	\arrow ["\circlearrowright "{marking, allow upside down}, draw=none, from=1-2, to=0]
\end {tikzcd}
\end{center}

Viewed from the pseudo-model perspective, a path of four composable loose morphisms is an element of \(M_4\), and the start and end of the two paths in the above diagram are given by its image under the following two maps from \(M_4\) to \(M_1\):

  \begin{center}
\begin {tikzcd}[column sep=5em]
{M_4} & {M_1}
	\arrow ["{M_{\mathsf {d}_{1}}\,M_{\mathsf {d}_{2}}\,M_{\mathsf {d}_{3}}}"', shift right=1, from=1-1, to=1-2]
	\arrow ["{M_{\mathsf {d}_{1}}\,M_{\mathsf {d}_{1}}\,M_{\mathsf {d}_{1}}}", shift left=1, from=1-1, to=1-2]
\end {tikzcd}
\end{center}

The condition that composite along the two paths of associators are equal corresponds to the coherence property of pseudofunctors that all ways of defining unbiased compositors out of the binary compositor and unitor are equal.
\par{}
We note, however, that the data and coherence conditions for a pseudo-model of \(\Delta _{\leq  4}\) are moreover \emph{sufficient} to recover the usual biased notion of internal pseudo-category, as none of the coherence conditions for a pseudo-category involve equalities between morphisms for operations of higher arity than four.
That Leinster's unbiased notion of \emph{weak double category} is equivalent to the ordinary biased notion appears to be folklore, though we observe that the proof essentially follows by a variation of the proof for the equivalence between the biased and unbiased notions of monoidal category provided in Appendix B of \cite{leinster-2004-higher}. We shall therefore refer to the restriction of the \(\mathscr {F}\)-sketch structure on \(\Delta  ^{\mathsf {op}}\) to \(\Delta _{\leq  4}^{\mathsf {op}}\) as the \textbf{biased sketch for pseudo-categories}.
\end{remark}
\subsection{The slice theorem for pseudo-models}\label{djm-00KT}\par{}
The \emph{slice theorem} for \hyperref[djm-00I5]{pseudo-models} of \hyperref[djm-00IA]{\(\mathscr {F}\)-sketches}, generalising the one given for limit sketches in \cref{jrb-001X}, is as follows.

\begin{theorem}[{The slice theorem for \(\mathscr {F}\)-sketches}]\label{djm-00IG}
\par{}
	Let \(\mathsf {A}\) be a locally-discrete \hyperref[djm-00IA]{\(\mathscr {F}\)-sketch} whose marked cones are all inert, and let \(M : \mathsf {A} \to  \mathsf {Set}\) be a \emph{strict}, set-valued model of \(\mathsf {A}\). Let \(\mathcal {E}\mathsf {l}(M)\) be 2-category of \emph{elements} of \(M\) and equip this 2-category with the following structure of an \(\mathscr {F}\)-sketch induced by the projection \(\pi _M : \mathcal {E}\mathsf {l}(M) \to  \mathsf {A}\):  
	\begin{enumerate}\item{}
			inert 1-cells  \(\mathcal {E}\mathsf {l}(M)\) are those whose underlying 1-cell in \(\mathsf {A}\) is inert.
		
		\item{}
			marked cones in \(\mathcal {E}\mathsf {l}(M)\) are those mapped by \(\pi _M\) to marked cones in \(\mathsf {A}\).
		\end{enumerate}\par{}
	 then condsidering \(M\) as a (discrete) \hyperref[djm-00I5]{pseudo-model} in \(\mathcal {C}\mathsf {at}\), we have an equivalence: 
	\begin{equation}
		\mathcal {P}\mathsf {sMod}^{\mathsf {p}}_{\mathscr {F}} \left (\mathsf {A},\mathcal {C}\mathsf {at}\right ) \downarrow  M \simeq  \mathcal {P}\mathsf {sMod}^{\mathsf {p}}_{\mathscr {F}} \left (\mathcal {E}\mathsf {l}(M),\mathcal {C}\mathsf {at}\right ).
	\end{equation}\end{theorem}
\par{}
	Taking \(\mathcal {A} = \Delta ^{\mathsf {op}}\) and \(M = \Delta  [1]\) in the above theorem will yield the equivalence
	\begin{equation}
		\mathcal {P}\mathsf {sMod}^{\mathsf {p}}_{\mathscr {F}} \left (\Delta ^{\mathsf {op}},\mathcal {C}\mathsf {at}\right ) \downarrow  \Delta  [1] \simeq  \mathcal {P}\mathsf {sMod}^{\mathsf {p}}_{\mathscr {F}} \left (\left ( \Delta  \downarrow  [1] \right )^{\mathsf {op}},\mathcal {C}\mathsf {at}\right )
	\end{equation}
	Keeping in mind the equivalence \(\mathcal {P}\mathsf {sMod}^{\mathsf {p}}_{\mathscr {F}} \left (\Delta ^{\mathsf {op}},\mathcal {C}\mathsf {at}\right ) \simeq  \mathcal {D}\mathsf {bl}\) of \cref{djm-00IE} and noting that \(\Delta  [1] = \yo [1]\) corresponds to \(\mathbb {L}\mathsf {oose}\) under this equivalence, we observe that the 2-category on the left is equivalent to the 2-category of double barrels as described in \cref{djm-00JP-1}. Later, in \cref{jrb-003B}, we will observe that pseudo-models for \(\left ( \Delta  \downarrow  [1] \right )^{\mathsf {op}}\) provide an unbiased presentation for internal pseudo-bimodules, thus establishing the connection between double barrels and pseudo-bimodules in \(\mathcal {C}\mathsf {at}\).
  
\begin{remark}[{The local-discreteness condition}]\label{jrb-0039}
\par{}
A locally-discrete \(\mathscr {F}\)-category is one whose underlying 2-cells are all identities. The \(\mathscr {F}\)-sketches whose pseudo-models we are primarily interested in — the (biased and unbiased) sketches for pseudo-categories and pseudo-bimodules — are both of this form. Such an \(\mathscr {F}\)-sketch can equivalently be viewed as a 1-category with a given wide subcategory specifying the inerts. This is moreover equivalent to a category enriched in the cartesian-closed category \(\mathsf {Set}^\hookrightarrow \) of injective maps in \(\mathsf {Set}\). We work with the more general structure \(\mathscr {F}\)-sketches to more naturally express the property of being a pseudo-model and allow more flexible notions of sketch. We anticipate that this flexibility will allow us to exhibit the \emph{pseudo-model coclassifier} of an \(\mathscr {F}\)-sketch \(\mathcal {A}\) — i.e. another sketch \(\mathcal {A}'\) whose \emph{strict} models are equivalent to the pseudo-models of \(\mathcal {A}\). Even when \(\mathcal {A}\) is locally-discrete, its pseudo-model classifier will typically not be.
\end{remark}
\par{}
As noted in \cref{jrb-001S-proof}, \cref{djm-00IG} resembles the property that slices of (\(\mathsf {Set}\)-valued) presheaf categories are again presheaf categories, a convenient proof for which exploits the equivalence between presheaves and discrete fibrations. We can use a similar trick to prove \cref{djm-00IG} by establishing an equivalence between pseudo-\(\mathscr {F}\)-functors and “\(\mathscr {F}\)-opfibrations”. 
\subsubsection{\(\mathscr {F}\)-opfibrations}\label{jrb-000V}\par{}
A pseudo-\(\mathscr {F}\)-functor from a locally-discrete \(\mathscr {F}\)-category, \(\mathsf {A}\), into \(\mathcal {C}\mathsf {at}\) is \emph{a fortiori} a pseudo-functor from the underlying 1-category, \({\mathsf {A}}_{\lambda }\), into \(\mathcal {C}\mathsf {at}\). There is a well-known equivalence:
\begin{equation}
\mathcal {P}\mathsf {s}^{\mathsf {p}} \left ( {\mathsf {A}}_{\lambda },\mathcal {C}\mathsf {at} \right )
\simeq 
\mathcal {O}\mathsf {pFib} \left ( {\mathsf {A}}_{\lambda } \right )
\end{equation}
where \(\mathcal {O}\mathsf {pFib} \left ( {\mathsf {A}}_{\lambda } \right )\) is the 2-category with:
\begin{description}
\item[\textbf{0-cells:}] opfibrations over \(\mathsf {A}\)
  \item[\textbf{1-cells:}] opcartesian maps
  \item[\textbf{2-cells:}] vertical 2-cells between opcartesian maps\end{description} whose action on an object \(P : {\mathsf {A}}_{\lambda } \to  \mathcal {C}\mathsf {at}\) is given by taking its \emph{Grothendieck construction} \(\pi _P : \int {P} \to  \mathsf {A}\).
\par{}
Since \(\mathcal {P}\mathsf {s}^{\mathsf {p}}_{\mathscr {F}} \left (\mathsf {A},\mathcal {C}\mathsf {at}\right )\) is a locally-full sub-2-category of \(\mathcal {P}\mathsf {s}^{\mathsf {p}} \left ( {\mathsf {A}}_{\lambda }, \mathcal {C}\mathsf {at} \right )\), there exists some equivalent locally-full sub-2-category of “\(\mathscr {F}\)-opfibrations” in \(\mathcal {O}\mathsf {pFib} \left ( \mathsf {A} \right )\). This equivalence provides a convenient translation of \cref{djm-00IG} to a statement about \(\mathscr {F}\)-opfibrations. Before providing this statement, we give an explicit characterisation of those opfibrations which result from the Grothendieck construction of a pseudo-\(\mathscr {F}\)-functor, and those opcartesian morphisms between them which correspond to pseudo-\(\mathscr {F}\)-transformations.

\begin{definition}[{\(\mathscr {F}\)-opfibration}]\label{jrb-000W}
\par{}
An \textbf{\(\mathscr {F}\)-opfibration} a locally-discrete \(\mathscr {F}\)-category \(\mathsf {A}\) is a cloven opfibration \(p : \mathsf {E} \to  \mathsf {A}\) with two additional properties:
\begin{enumerate}\item{}for \(x \in  \mathsf {E}\), the opcartesian lift \(\underline {1_{px}}\) is equal to \(1_x\)
 
  \item{}for composable 1-cells in \(\mathsf {A}\) of the form:

  \begin{center}
\begin {tikzcd}
	px & b & a
	\arrow ["v", from=1-1, to=1-2]
	\arrow ["u", from=1-2, to=1-3]
\end {tikzcd}
\end{center}

if either \(u\) or \(v\) is inert then \(\underline {u} \left ( v_* x \right ) \;\underline {v}\left ( x \right ) = \underline {uv} \left ( x \right )\).\end{enumerate}
where \(\underline {u} \left ( x \right ) : u_*\left ( x \right ) \to  x\) denotes the chosen opcartesian lift of \(u\).
The domain of an \(\mathscr {F}\)-opfibration naturally inherits a (locally-discrete) \(\mathscr {F}\)-category structure by declaring the inerts to be \emph{chosen}-opcartesian lifts of inerts in \(\mathsf {A}\).
\end{definition}

\begin{definition}[{\(\mathscr {F}\)-opcartesian maps, \(\mathcal {O}\mathsf {pFib}_{\mathscr {F}}\left (\mathsf {A}\right )\)}]\label{jrb-000Y}
\par{}
Given two \(\mathscr {F}\)-opfibrations \(p : \mathsf {D} \to  \mathsf {A}\) and \(q : \mathsf {E} \to  \mathsf {A}\), we will say a opcartesian map \(f : \mathsf {D} \to  \mathsf {E}\) between them is additionally \textbf{\(\mathscr {F}\)-opcartesian} if it strictly preserves chosen lifts of inerts. I.e. if for any inert map \(u : px \to  a\) in \(\mathsf {A}\) we have:
\begin{equation}f \left ( x  \xrightarrow {\underline {u}x} u_*x \right ) = fx  \xrightarrow {\underline {u}(fx)} u_*(fx)\end{equation}
Note that this is equivalent to requiring that the opcartesian map \(f : \mathsf {D} \to  \mathsf {E}\) be an \(\mathscr {F}\)-functor with respect to the induced \(\mathscr {F}\)-category structures on \(\mathsf {D}\) and \(\mathsf {E}\) as described in \cref{jrb-000W}.
\par{}
We let \(\mathcal {O}\mathsf {pFib}_{\mathscr {F}}\left (\mathsf {A}\right )\) denote the 2-category of \(\mathscr {F}\)-opfibrations, \(\mathscr {F}\)-opcartesian maps and vertical transformations over a locally discrete \(\mathscr {F}\)-category \(\mathsf {A}\).
\end{definition}

\begin{lemma}[{\(\mathcal {O}\mathsf {pFib}_{\mathscr {F}}\left (\mathsf {A}\right )\) is equivalent to \(\mathcal {P}\mathsf {s}^{\mathsf {p}}_{\mathscr {F}} \left (\mathsf {A},\mathcal {C}\mathsf {at}\right )\)}]\label{jrb-000X}
\par{}
For \(\mathsf {A}\) a , the Grothendieck construction:
\begin{equation}
\mathcal {P}\mathsf {s}^{\mathsf {p}} \left ( {\mathsf {A}}_{\lambda },\mathcal {C}\mathsf {at} \right )
\xrightarrow {\int }
\mathcal {O}\mathsf {pFib} \left ( {\mathsf {A}}_{\lambda } \right )
\end{equation}
restricts to an equivalence:
\begin{equation}\mathcal {P}\mathsf {s}^{\mathsf {p}}_{\mathscr {F}} \left (\mathsf {A},\mathcal {C}\mathsf {at}\right ) \xrightarrow {\int }
\mathcal {O}\mathsf {pFib}_{\mathscr {F}}\left (\mathsf {A}\right )
\end{equation}
\begin{proof}[{proof of \cref{jrb-000X}.}]\label{jrb-000X-proof}
\par{}
It suffices to show that \(\mathcal {O}\mathsf {pFib}_{\mathscr {F}}\left (\mathsf {A}\right )\) corresponds to the essential image of \(\mathcal {P}\mathsf {s}^{\mathsf {p}}_{\mathscr {F}} \left (\mathsf {A},\mathcal {C}\mathsf {at}\right ) \hookrightarrow  \mathcal {P}\mathsf {s}^{\mathsf {p}} \left ( {\mathsf {A}}_{\lambda },\mathcal {C}\mathsf {at} \right )\) under the Grothendieck construction. It is straightforward to see that performing the Grothendieck construction on a pseudo-\(\mathscr {F}\)-functor produces an \(\mathscr {F}\)-opfibration, since the comparison between \(\left ( uv \right )_* \left ( a,x \right ) = \left ( a,F_{uv}x \right )\) and \(u_* \left ( v_* \left ( a,x \right ) \right ) = \left ( a,F_uF_v x \right )\) is given in terms of the compositor \(F_2 : F_uF_v \Rightarrow  F_{uv}\) by \(\left ( 1_a, F_2x \right )\). By \hyperref[djm-00I2]{the definition of pseudo-\(\mathscr {F}\)-functor} these compositors will be identities precisely when they need to be to satisfy condition 2 in \hyperref[jrb-000Y]{the definition of \(\mathscr {F}\)-opfibration}. Similarly, the unitarity condition of \cref{djm-00I2} ensures that the resulting opfibration satisfies condition 1 of \cref{jrb-000Y}. This correspondence between the conditions of being a pseudo-\(\mathscr {F}\)-functor and being an \(\mathscr {F}\)-opfibration can similarly be used to show that given any \(\mathscr {F}\)-opfibration, the corresponding \emph{indexing pseudofunctor} (i.e. its image under the pseudo-inverse of the Grothendieck construction) will be a pseudo-\(\mathscr {F}\)-functor.
\par{}
It remains to show that \(\mathscr {F}\)-opcartesian maps between the Grothendieck construction of pseudo-\(\mathscr {F}\)-functors are precisely those which arise from pseudo-\(\mathscr {F}\)-transformations. Given an \emph{arbitrary} pseudonatural transformation \(\phi  : F \Rightarrow   G : {\mathsf {A}}_{\lambda } \to  \mathcal {C}\mathsf {at}\) between pseudofunctors, the corresponding opcartesian map \(\int  \phi  : \int  F \to  \int  G\) between the Grothendieck constructions is given on objects and morphisms as follows:

  \begin{center}
\begin {tikzcd}
	{(a,x)} \\
	{(b,y)}
	\arrow ["{(u,f)}"', from=1-1, to=2-1]
\end {tikzcd}
$\mapsto $
\begin {tikzcd}
	{(a,\phi _a x)} \\
	{(b,\phi _b y)}
	\arrow ["{(u,\phi _b f \; \phi _u x)}", from=1-1, to=2-1]
\end {tikzcd}
\end{center}

and therefore acts in particular on chosen-opcartesian morphisms as:

  \begin{center}
\begin {tikzcd}
	{(a,x)} \\
	{(b,F_u \,x)}
	\arrow ["{(u,1_{F_ux})}"', from=1-1, to=2-1]
\end {tikzcd}
$\mapsto $
\begin {tikzcd}
	{(a,\phi _a \,x)} \\
	{(b,\phi _b \,F_u\,x)}
	\arrow ["{(u,\;\phi _u x)}", from=1-1, to=2-1]
\end {tikzcd}
\end{center}

The opcartesian map \(\int  \phi \) will thus be moreover \(\mathscr {F}\)-opcartesian precisely if for all inert \(u : a \to  b\), the pseudo-naturality transformation \(\phi _u :  G_u\; \phi _a\;x \Rightarrow   \phi _b \; F_u\;x\) is an identity; i.e. if and only if \(\phi \) is a pseudo-\(\mathscr {F}\)-transformation.
\end{proof}
\end{lemma}
\par{}
We're primarily interested in the 2-category of pseudo-\(\mathscr {F}\)-functors and \emph{pseudo}-\(\mathscr {F}\)-transformations, but we note that the proof above extends essentially without modification to exhibit an equivalence between the 2-category with pseudo-\(\mathscr {F}\)-functors and \emph{lax}-\(\mathscr {F}\)-transformations — those lax transformations which have strict naturality squares at inerts —
and the 2-category of \(\mathscr {F}\)-opfibrations and \emph{arbitrary} \(\mathscr {F}\)-functors in the slice. We will denote these 2-categories \(\mathcal {P}\mathsf {s}^{\mathsf {lx}}_{\mathscr {F}} \left ( {\mathsf {A}}_{\lambda },\mathcal {C}\mathsf {at} \right )\)
and \(\mathcal {O}\mathsf {pFib}^{\mathsf {lx}}_{\mathscr {F}}\left (\mathsf {A}\right )\) respectively.
\begin{lemma}\label{jrb-0025}
\par{}
For \(\mathsf {A}\) a locally-discrete \(\mathscr {F}\)-category, the Grothendieck construction:
\begin{equation}
\mathcal {P}\mathsf {s}^{\mathsf {lx}} \left ( {\mathsf {A}}_{\lambda },\mathcal {C}\mathsf {at} \right )
\xrightarrow {\int }
\mathcal {O}\mathsf {pFib}^{\mathsf {lx}} \left ( {\mathsf {A}}_{\lambda } \right )
\end{equation}
restricts to an equivalence:
\begin{equation}\mathcal {P}\mathsf {s}^{\mathsf {lx}}_{\mathscr {F}}\left ( \mathsf {A}, \mathcal {C}\mathsf {at} \right ) \xrightarrow {\int }
\mathcal {O}\mathsf {pFib}^{\mathsf {lx}}_{\mathscr {F}}\left (\mathsf {A}\right )
\end{equation}\end{lemma}
\par{}
The usefulness of the fibred perspective on pseudo-\(\mathscr {F}\)-functors is that we can exploit a result about factorisations of opfibrations through discrete opfibrations, suitably modified for the \(\mathscr {F}\)-categorical setting.

\begin{definition}[{Discrete \(\mathscr {F}\)-opfibration}]\label{jrb-0015}
\par{}
A \emph{discrete \(\mathscr {F}\)-opfibration} over a locally discrete \(\mathscr {F}\)-category \(\mathsf {A}\) is an \(\mathscr {F}\)-opfibration whose underlying opfibration of categories is discrete.
\end{definition}
\par{}
It is not generally true that a map between opfibrations is itself an opfibration, even if the map is opcartesian. However, if the codomain is a \emph{discrete} opfibration, then the result \emph{does} hold. To see that this is true, consider a map in \(\mathcal {C}\mathsf {at} \downarrow  \mathsf {B}\) between opfibrations:

  \begin{center}
\begin {tikzcd}
	{\mathsf {E}} && {\mathsf {D}} \\
	& {\mathsf {B}}
	\arrow ["f", from=1-1, to=1-3]
	\arrow ["p"', from=1-1, to=2-2]
	\arrow ["q", from=1-3, to=2-2]
\end {tikzcd}
\end{center}

with \(q\) discrete.

\begin{lemma}\label{jrb-0018}
\par{}
Any morphism in \(\mathsf {E}\) which is \(p\)-opcartesian is \(f\)-opcartesian.

\begin{proof}[{proof of \cref{jrb-0018}.}]\label{jrb-0018-proof}
\par{}
A morphsism \(u : x \to  y\) is \(p\)-preopcartesian if for any \(z \in  \mathsf {E}\), precomposition by \(u\) gives a bijection from \(p\)-vertical maps \(\mathsf {V}_p \left ( y,z \right )\) to \(p\)-lifts of \(pu\), \(\left [ pu \right ]_p \left ( x,z \right )\). There is clearly a commuting square:

  \begin{center}
\begin {tikzcd}
	{\mathsf {V}_f(y,z)} & {\left [ fu \right ]_f(x,z)} \\
	{\mathsf {V}_p (y,z)} & {\left [ pu \right ]_p (x,z)}
	\arrow ["{u^*}", from=1-1, to=1-2]
	\arrow [""{name=0, anchor=center, inner sep=0}, "{\mathsf {incl}}"',hook', from=1-1, to=2-1]
	\arrow [""{name=1, anchor=center, inner sep=0}, "{\mathsf {incl}}", hook', from=1-2, to=2-2]
	\arrow ["{u^*}"', from=2-1, to=2-2]
	\arrow ["\circlearrowright "{marking, allow upside down}, draw=none, from=0, to=1]
\end {tikzcd}
\end{center}

Since \(u\) is \(p\)-preopcartsian, the bottom arrow is an isomorphism, and since \(q\) is discrete, the left inclusion is an isomorphism. It follows that both the upper precomposition and the right inclusion are also isomorphisms, so \(u\) is additionally \(f\)-preopcartesian. Since \(f\)-preopcartesian morphisms are \(p\)-preopcartesian, and thus closed under composition, they must additionally be \(f\)-opcartesian.
\end{proof}
\end{lemma}
\par{}
Now, given a morphism \(u : x \to  fe\) in \(\mathsf {D}\), we can map down along \(q\) to \(qu : qx \to  pe\) and take the \(p\)-opcartesian (and thus \(f\)-opcartesian) lift,  \(\underline {qu}\left ( e \right ) : \left ( qu \right )_* e \to  e\). The image of this morphism under \(f\) is a lift of \(qu\) along \(q\), and so must be equal to \(u\). The functor \(f\) thus satisfies the opcartesian lifting property for opfibrations.

\begin{lemma}[{Lifting \(\mathscr {F}\)-opfibrations along discrete \(\mathscr {F}\)-opfibrations}]\label{jrb-0019}
\par{}
Given a morphism in \(\mathscr {F}\mathsf {Cat} \downarrow  \mathsf {B}\) between \(\mathscr {F}\)-opfibrations over a locally-discrete \(\mathscr {F}\)-category:

  \begin{center}
\begin {tikzcd}
	{\mathsf {E}} && {\mathsf {D}} \\
	& {\mathsf {B}}
	\arrow ["f", from=1-1, to=1-3]
	\arrow ["p"', from=1-1, to=2-2]
	\arrow ["q", from=1-3, to=2-2]
\end {tikzcd}
\end{center}

with \(q\) discrete, \(f\) is moreover an \(\mathscr {F}\)-opfibration with respect to the induced \(\mathscr {F}\)-category structure on \(\mathsf {D}\).

\begin{proof}[{proof of \cref{jrb-0019}.}]\label{jrb-0019-proof}
\par{}
From our observation above, such an \(f\) is automatically an opfibration so it suffices to demonstrate the two additional properties specified in \cref{jrb-000W}:
\begin{enumerate}\item{}for each \(d \in  \mathsf {D}\), the opcartesian lift of \(1_{fx}\) is equal to the opcartesian lift of \(1_{px}\), which is equal to \(1_x\) by the \(\mathscr {F}\)-opcartesianness of \(p\). 
  \item{}for composable 1-cells in \(\mathsf {D}\) of the form:

  \begin{center}
\begin {tikzcd}
	fx & a & b
	\arrow ["u", from=1-1, to=1-2]
	\arrow ["v", from=1-2, to=1-3]
\end {tikzcd}
  \end{center}
the composite of opcartesian lifts \(\underline {v} \left ( u_* x \right )\;\underline {u} \left ( x \right ) \) along \(f\) is by definition equal to the composite of opcartesian lifts \(\underline {qv} \left ( qu_*x \right ) \; \underline {qu} \left ( x \right ) \;\) along \(p\). Whenever either \(u\) (and thus \(qu\)) or \(v\) (and thus \(qv\)) are inert, this lift will be equal to \(\underline {(qv)(qu)}x = \underline {q(vu)}x\) by the \(\mathscr {F}\)-opcartesianness of \(p\), which is by definition the opcartesian lift \(\underline {vu}(x)\) along \(f\).
\end{enumerate}\end{proof}
\end{lemma}

\begin{lemma}[{\(\mathscr {F}\)-opfibrations are closed under composition}]\label{jrb-001C}
\par{}
Given a locally-discrete \(\mathscr {F}\)-category, \(\mathsf {A}\), an \(\mathscr {F}\)-opfibration \(p : \mathsf {B} \to  \mathsf {A}\) and another \(\mathscr {F}\)-opfibration \(q : \mathsf {C} \to  \mathsf {B}\) with respect to the induced \(\mathscr {F}\)-category structure on \(\mathsf {B}\), the composite \(pq : \mathsf {C} \to  \mathsf {A}\) is also an \(\mathscr {F}\)-opfibration.

\begin{proof}[{proof of \cref{jrb-001C}.}]\label{jrb-001C-proof}
\par{}
Fibrations in \(\mathcal {C}\mathsf {at}\) are known to be closed under composition, so it remains to observe that the strictness-on-inerts property holds for the composite of \(\mathscr {F}\)-opfibrations:
  \begin{enumerate}\item{}for each \(c \in  \mathsf {C}\), the opcartesian lift of \(1_{pq(c)}\) along \(pq\) is equal to the opcartesian lift of \(1_{qc}\) along \(p\), which is equal to \(1_c\) by the \(\mathscr {F}\)-opcartesianness of \(p\) and \(q\) respectively. 
  \item{}Consider composable 1-cells in \(\mathsf {A}\) of the form:

  \begin{center}
\begin {tikzcd}
	pq(c) & a & b
	\arrow ["u", from=1-1, to=1-2]
	\arrow ["v", from=1-2, to=1-3]
\end {tikzcd}
  \end{center}
  Lifting along \(pq\) is done by first lifting along \(p\), then along \(q\). To simplify the notation, let's denote the step-wise lifting of each morphism along \(p\) as follows:

  \begin{center}
\begin {tikzcd}
	q(c) & a^p & b^p
	\arrow ["u^p", from=1-1, to=1-2]
	\arrow ["v^p", from=1-2, to=1-3]
\end {tikzcd}
  \end{center}

  Then lifting each of those morphisms along \(q\) gives:

  \begin{center}
\begin {tikzcd}
  c & \left ( a^p \right )^q & \left ( b^p \right )^q
	\arrow ["\left ( u^p \right )^q", from=1-1, to=1-2]
	\arrow ["\left ( v^p \right )^q", from=1-2, to=1-3]
\end {tikzcd}
    \end{center}

If either \(u\) or \(v\) are inert in \(\mathsf {A}\), then so (respectively) is \(u^p\) or \(v^p\) in \(\mathsf {B}\), from which it follows that:
  \begin{equation}\left ( v^p \right )^q\; \left ( u^p \right )^q = \left ( v^p \; u^p \right )^q = \left ( v\;u \right )^{pq}\end{equation}
  as required.
\end{enumerate}\end{proof}
\end{lemma}

\begin{corollary}\label{jrb-001B}
\par{}
For \(\mathsf {A}\) a locally-discrete \(\mathscr {F}\)-category and \(q : \mathsf {E} \to  \mathsf {A}\) a discrete \(\mathscr {F}\)-opfibration, post-composition by \(q\) gives isomorphisms:
\begin{equation}\mathcal {O}\mathsf {pFib}_{\mathscr {F}}\left (\mathsf {E}\right ) \cong  \mathcal {O}\mathsf {pFib}_{\mathscr {F}}\left (\mathsf {A}\right ) \downarrow  q
 \qquad  \mathcal {O}\mathsf {pFib}^{\mathsf {lx}}_{\mathscr {F}}\left (\mathsf {E}\right ) \cong  \mathcal {O}\mathsf {pFib}^{\mathsf {lx}}_{\mathscr {F}}\left (\mathsf {A}\right ) \downarrow  q\end{equation}
\begin{proof}[{proof of \cref{jrb-001B}.}]\label{jrb-001B-proof}
\par{}
As a map on objects, post-composition by \(q\) does indeed send \(\mathscr {F}\)-opfibrations to \(\mathscr {F}\)-opfibrations with a map to \(q\) by \cref{jrb-001C}. These maps are bijective on objects by \cref{jrb-0019}. Morphisms between objects in the slice \(\mathcal {O}\mathsf {pFib}_{\mathscr {F}}\left (\mathsf {A}\right )\downarrow  q\)
are precisely the opcartesian maps over \(\mathsf {E}\) by \cref{jrb-0018}.\end{proof}
\end{corollary}

\begin{lemma}[{The slice theorem for pseudo-\(\mathscr {F}\)-functors}]\label{jrb-001N}
\par{}
Given a locally-discrete \(\mathscr {F}\)-category \(\mathsf {A}\) and a pseudo-\(\mathscr {F}\)-functor \(M : \mathsf {A} \to  \mathcal {C}\mathsf {at}\) whose image lands in discrete categories (i.e comes from an \(\mathscr {F}\)-functor \(\mathsf {A} \to  \mathsf {Set}\)), there are equivalences:
\begin{equation}\mathcal {P}\mathsf {s}^{\mathsf {p}}_{\mathscr {F}} \left (\mathsf {A},\mathcal {C}\mathsf {at}\right ) \downarrow   M \simeq  \mathcal {P}\mathsf {s}^{\mathsf {p}}_{\mathscr {F}} \left (\mathcal {E}\mathsf {l} \left ( M \right ),\mathcal {C}\mathsf {at}\right )\end{equation}
\begin{equation}\mathcal {P}\mathsf {s}^{\mathsf {lx}}_{\mathscr {F}} \left (\mathsf {A},\mathcal {C}\mathsf {at}\right ) \downarrow   M \simeq  \mathcal {P}\mathsf {s}^{\mathsf {lx}}_{\mathscr {F}} \left (\mathcal {E}\mathsf {l} \left ( M \right ),\mathcal {C}\mathsf {at}\right )\end{equation}
\begin{proof}[{proof of \cref{jrb-001N}.}]\label{jrb-001N-proof}
\par{}
By \cref{jrb-000X} there is an equivalence \(\int  : \mathcal {P}\mathsf {s}^{\mathsf {p}}_{\mathscr {F}} \left (\mathsf {A},\mathcal {C}\mathsf {at}\right ) \to  \mathcal {O}\mathsf {pFib}_{\mathscr {F}}\left (\mathsf {C}\right )\) which sends a pseudo-\(\mathscr {F}\)-functor to its Grothendieck construction. In particular, it will send an \(M\) as described above to the discrete \(\mathscr {F}\)-opfibration \(\pi _M : \mathcal {E}\mathsf {l} \left ( M \right ) \to  \mathsf {A}\) (note that \(\mathcal {E}\mathsf {l} \left ( M \right ) = \int {M}\)). The result then follows from this observation and \cref{jrb-001B}:
\begin{equation}
\begin {aligned}
\mathcal {P}\mathsf {s}^{\mathsf {p}}_{\mathscr {F}} \left (\mathsf {A},\mathcal {C}\mathsf {at}\right ) \downarrow   M
\simeq  \; & \mathcal {O}\mathsf {pFib}_{\mathscr {F}}\left (\mathsf {A}\right ) \downarrow   \left [ \mathcal {E}\mathsf {l} \left ( M \right ) \xrightarrow {\pi _M} \mathsf {A} \right ] \\
\cong  \; & \mathcal {O}\mathsf {pFib}_{\mathscr {F}}\left (\mathcal {E}\mathsf {l} \left ( M \right )\right ) \\
\simeq  \; & \mathcal {P}\mathsf {s}^{\mathsf {p}}_{\mathscr {F}} \left (\mathcal {E}\mathsf {l}\left ( M \right ),\mathcal {C}\mathsf {at}\right )
\end {aligned}
\end{equation}
A similar argument using the equivalence \(\mathcal {O}\mathsf {pFib}^{\mathsf {lx}}_{\mathscr {F}}\left (\mathsf {E}\right ) \cong  \mathcal {O}\mathsf {pFib}^{\mathsf {lx}}_{\mathscr {F}}\left (\mathsf {A}\right ) \downarrow  q\) proves the lax case.
\end{proof}
\end{lemma}
\subsubsection{Model opfibrations}\label{jrb-001H}\par{}
We'd like to now take the equivalence:
\begin{equation}\mathcal {P}\mathsf {s}^{\mathsf {p}}_{\mathscr {F}} \left (\mathsf {A},\mathcal {C}\mathsf {at}\right ) \downarrow  M \simeq  \mathcal {P}\mathsf {s}^{\mathsf {p}}_{\mathscr {F}} \left ( \mathcal {E}\mathsf {l} \left ( M \right ),\mathcal {C}\mathsf {at}\right )\end{equation}
for \(\mathsf {A}\) a locally-discrete \(F\)-category and \(M\) a “discrete” pseudo-\(\mathscr {F}\)-functor to \(\mathcal {C}\mathsf {at}\), and restrict it to an equivalence:
\begin{equation}\mathcal {P}\mathsf {sMod}^{\mathsf {p}}_{\mathscr {F}} \left (\mathsf {A},\mathcal {C}\mathsf {at}\right ) \downarrow  M \simeq  \mathcal {P}\mathsf {sMod}^{\mathsf {p}}_{\mathscr {F}} \left ( \mathcal {E}\mathsf {l} \left ( M \right ),\mathcal {C}\mathsf {at}\right )\end{equation}
in the case where \(\mathsf {A}\) is moreover an \(\mathscr {F}\)-sketch and \(M\) is moreover a model.

We can further exploit our fibred perspective on pseudo-\(\mathscr {F}\)-functors to analyse pseudo-\emph{models} once we give a characterisation of those \(\mathscr {F}\)-opfibrations which arise as Grothendieck constructions of pseudo-models.
\par{}
\begin{definition}[{Model opfibration}]\label{jrb-001P}
\par{}
A functor \(\mathsf {B} \to  \mathsf {A}\) into a locally-discrete \(\mathscr {F}\)-sketch is a \textbf{model opfibration} if it is isomorphic to the Grothendieck construction of a pseudo model of \(\mathsf {A}\). It is a \textbf{discrete model opfibration} if it is moreover a discrete \(\mathscr {F}\)-opfibration.
\end{definition}

To provide a more convenient characterisation of model opfibrations we first establish a lemma about lifting functors along general \(\mathscr {F}\)-opfibrations:

\begin{lemma}\label{jrb-001E}
\par{}
Given a locally-discrete \(\mathscr {F}\)-category \(\mathsf {A}\), a pseudo-\(\mathscr {F}\)-functor \(P : \mathsf {A} \to  \mathcal {C}\mathsf {at}\) and an \(\mathscr {F}\)-functor \(S : \mathsf {J} \to  \mathsf {A}\) from a \emph{chordate} \(\mathscr {F}\)-category (all morphisms are inerts), the category of lifts of \(S\) to an \(\mathscr {F}\)-functor \(S' : \mathsf {J} \to  \int  P\) is isomorphic to \(\lim  PS\).

\begin{proof}[{proof of \cref{jrb-001E}.}]\label{jrb-001E-proof}
\par{}
Consider such a lift:

  \begin{center}
\begin {tikzcd}
	& {\int  P} \\
	{\mathsf {J}} & {\mathsf {A}}
	\arrow [""{name=0, anchor=center, inner sep=0}, "{\pi _P}", from=1-2, to=2-2]
	\arrow ["{S'}", from=2-1, to=1-2]
	\arrow ["S"', from=2-1, to=2-2]
	\arrow ["\circlearrowright "{marking, allow upside down}, draw=none, from=2-1, to=0]
\end {tikzcd}
\end{center}

The data of \(S'\) are choices for each \(j \in  \mathsf {J}\) of an object in the fibre over \(S_j\) of \(\pi _P\), i.e. an object \(x_j\) in \(PS_j\). Because the only inerts in \(\int  P\) are chosen-opcartesian lifts of inerts in \(\mathsf {A}\), the image of a morphism \(u : j \to  k\) in \(\mathsf {J}\) must be \emph{the} chosen opcartesian lift \(\underline {S_u} : \left ( Sj, x_j \right ) \to  \left ( Sk, x_k \right )\), from which it follows that \(x_k = PS_u x_k\). Such data are equivalent to a cone from \(\bullet \) to \(PS\) in \(\mathcal {C}\mathsf {at}\).
\par{}
A (vertical) morphism between lifts from \(j \mapsto  x_j\) to \(j \mapsto  y_j\) is equivalent to specifying for each \(j \in  \mathsf {J}\) a morphism \(\alpha _j : x_j \to  y_j\) in \(PS_j\) in a natural way, i.e. such that for each \(u : j \to  k\) in \(\mathsf {J}\), \(\alpha _k = P_u \alpha _j\), which is equivalently a morphism between the corresponding cones \(\left \langle  x_j \right \rangle \) and \(\left \langle  y_j \right \rangle \) in \(\mathcal {C}\mathsf {at}\). The result then follows by the fact that \(\lim  PS \cong  \left [ \mathsf {J},\mathcal {C}\mathsf {at} \right ] \left ( \Delta  1, PS \right )\).
\end{proof}
\end{lemma}
\par{}
For the purpose of characterising model opfibrations, the lifts we are interested in are of functors from the domain and codomain of the marked cones of \(\mathsf {A}\) (recall that we define marked cones in \(\mathsf {A}\) with vertex \(a \in  \mathsf {A}\) as \(\mathscr {F}\)-functors from locally-discrete chordate \(\mathscr {F}\)-categories to \(a \downarrow  {\mathsf {A}}_{\mathsf {inert}}\)). The diagram over which this cone lies is simply the composite of this \(\mathscr {F}\)-functor with the projection \(\mathsf {cod} : a \downarrow  {\mathsf {A}}_{\mathsf {inert}} \to  \mathsf {A}\).

\begin{definition}[{\(\mathscr {F}\)-opfibrations which lift cones}]\label{jrb-001F}
\par{}
Given a locally-discrete \(\mathscr {F}\)-category \(\mathsf {A}\) and an \hyperref[jrb-000W]{\(\mathscr {F}\)-opfibration} \(p : \mathsf {B} \to  \mathsf {A}\), we say that \(p\) \textbf{lifts} a marked cone \(S : \mathsf {J} \to  a \downarrow  {\mathsf {A}}_{\mathsf {inert}}\) if:
\begin{enumerate}\item{}for all commuting squares of \(\mathscr {F}\)-functors as shown on the left, there is a unique diagonal filler as shown on the right:

  \begin{center}
\begin {tikzcd}
	{\mathsf {J}} & {\mathsf {B}} \\
	{a \downarrow  {\mathsf {A}}_{\mathsf {inert}}} & {\mathsf {A}}
	\arrow ["{L}", from=1-1, to=1-2]
	\arrow [""{name=0, anchor=center, inner sep=0}, "S"', from=1-1, to=2-1]
	\arrow [""{name=1, anchor=center, inner sep=0}, "{p}", from=1-2, to=2-2]
	\arrow ["{\mathsf {cod}}"', from=2-1, to=2-2]
	\arrow ["\circlearrowright "{marking, allow upside down}, draw=none, from=0, to=1]
  \end {tikzcd}
  $ \quad  \leadsto  \quad $
  \begin {tikzcd}
	{\mathsf {J}} & {\mathsf {B}} \\
	{a \downarrow  {\mathsf {A}}_{\mathsf {inert}}} & {\mathsf {A}}
	\arrow ["L", from=1-1, to=1-2]
	\arrow ["S"', from=1-1, to=2-1]
	\arrow ["{p}", from=1-2, to=2-2]
	\arrow [""{name=0, anchor=center, inner sep=0}, "{L'}"{description}, from=2-1, to=1-2]
	\arrow ["{\mathsf {cod}}"', from=2-1, to=2-2]
	\arrow ["\circlearrowright "{marking, allow upside down, pos=0.6}, draw=none, from=0, to=1-1]
	\arrow ["\circlearrowright "{marking, allow upside down, pos=0.6}, draw=none, from=0, to=2-2]
\end {tikzcd}
  \end{center}
  \item{}morphism of lifts \(\alpha  : L \Rightarrow   K\) extend uniquely to morphisms between the corresponding diagonal fillers:

  \begin{center}
\begin {tikzcd}
	{\mathsf {J}} & {\mathsf {B}} \\
	{a \downarrow  {\mathsf {A}}_{\mathsf {inert}}} & {\mathsf {A}}
	\arrow [""{name=0, anchor=center, inner sep=0}, "L"', from=1-1, to=1-2]
	\arrow [""{name=1, anchor=center, inner sep=0}, "K", popup=1em, from=1-1, to=1-2]
	\arrow ["S"', from=1-1, to=2-1]
	\arrow ["{p}", from=1-2, to=2-2]
	\arrow ["{\mathsf {cod}}"', from=2-1, to=2-2]
	\arrow ["\alpha "', shorten <=4pt, shorten >=4pt, Rightarrow, from=0, to=0|-1]
  \end {tikzcd}
  $ \quad  \leadsto  \quad  $
  \begin {tikzcd}
	{\mathsf {J}} & {\mathsf {B}} \\
	{a \downarrow  {\mathsf {A}}_{\mathsf {inert}}} & {\mathsf {A}}
	\arrow ["S"', from=1-1, to=2-1]
	\arrow ["{p}", from=1-2, to=2-2]
	\arrow [""{name=0, anchor=center, inner sep=0}, "{L'}"', curve={height=10pt}, from=2-1, to=1-2]
	\arrow [""{name=1, anchor=center, inner sep=0}, "{K'}"{pos=0.7}, curve={height=-10pt}, from=2-1, to=1-2]
	\arrow ["{\mathsf {cod}}"', from=2-1, to=2-2]
	\arrow ["\alpha  '"',shorten <=1pt, shorten >=1pt, Rightarrow, from=0, to=1]
\end {tikzcd}
  \end{center}\end{enumerate}
This can be more succinctly expressed as the condition that precomposition by \(S\) induces an isomorphism 
\(\left ( \mathscr {F}\mathsf {Cat} \downarrow  \mathsf {A} \right ) \left ( \mathsf {cod}, p \right ) \cong  \left ( \mathscr {F}\mathsf {Cat} \downarrow  \mathsf {A} \right ) \left ( \mathsf {cod}\;S, p \right )\)\end{definition}

\begin{lemma}\label{jrb-001G}
\par{}
An \(\mathscr {F}\)-opfibration over an \(\mathscr {F}\)-sketch \(\mathsf {A}\) is a model opfibration if and only if it
lifts all marked cones of the \(\mathscr {F}\)-sketch.

\begin{proof}[{proof of \cref{jrb-001G}.}]\label{jrb-001G-proof}
\par{}
By \cref{jrb-000X} we may assume the \(\mathscr {F}\)-opfibration is the Grothendieck construction of a pseudo-\(\mathscr {F}\)-functor, \(P : \mathsf {A} \to  \mathcal {C}\mathsf {at}\).
  Note that the category of lifts of \(\mathsf {cod} : a \downarrow  {\mathsf {A}}_{\mathsf {inert}} \to  \mathsf {A}\) along \(\pi _P : \int  P \to  \mathsf {A}\) is equivalent to \(P_a\), since every object in \(a \downarrow  {\mathsf {A}}_{\mathsf {inert}}\) admits an inert morphism to \(1_a\), and must therefore be mapped under any lift \(T : a \downarrow  {\mathsf {A}}_{\mathsf {inert}} \to  \int  P\) to the opcartesian "pushforward" of the image of \(1_a\) along the underlying inert morphism in \(\mathsf {A}\). I.e. the object \(u : a \to  x\) must be mapped to \(u_* \left ( T(1_a) \right )\).
 Thus any such lift is determined by the image of \(1_a\), which is equivalent to an element of \(P_a\). A vertical 2-cell between lifts is equivalent to morphisms between the corresponding objects in \(P_a\).
\par{}
So we have:
\begin{enumerate}\item{}\(\mathscr {F}\mathsf {Cat} \downarrow  \mathsf {A} \left ( a \downarrow  {\mathsf {A}}_{\mathsf {inert}} \xrightarrow {\mathsf {cod}} \mathsf {A}, \pi _P \right ) \cong  P_a\)
\item{}\(\mathscr {F}\mathsf {Cat} \downarrow  \mathsf {A} \left ( \mathsf {cod} \; S, \pi _P \right ) \cong  \lim  PS\)\end{enumerate}\par{}
The map \(P_a \to  \lim  PS\) induced by pre-composition along \(S : \mathsf {J} \to  a \downarrow  {\mathsf {A}}_{\mathsf {inert}}\) is simply the map induced by the cone \(\Delta  P_a \Rightarrow   PS\) whose component at \(j \in  \mathsf {J}\) is the image of the arrow \(S_j : a \to  s_j\) under the pseudo-\(\mathscr {F}\)-functor \(P\). So the condition that this map be an isomorphism, and thus that \(P\) map this marked cone to a limit cone, is equivalent to the condition that \(\pi _P\) lift this cone.
\end{proof}
\end{lemma}
\par{}
Our proof of the equivalence:
\begin{equation}\mathcal {P}\mathsf {sMod}^{\mathsf {p}}_{\mathscr {F}} \left (\mathsf {A},\mathcal {C}\mathsf {at}\right ) \downarrow  M \simeq  \mathcal {P}\mathsf {sMod}^{\mathsf {p}}_{\mathscr {F}} \left ( \mathcal {E}\mathsf {l} \left ( M \right ),\mathcal {C}\mathsf {at}\right )\end{equation}
in \cref{djm-00IG} will have the same structure as our proof of \cref{jrb-001N}, with the full subcategory \(\mathcal {O}\mathsf {pFib}_{\mathsf {mod}}\left (\mathsf {A}\right ) \hookrightarrow  \mathcal {O}\mathsf {pFib}_{\mathscr {F}}\left (\mathsf {A}\right )\) of model opfibrations playing the role of \(\mathcal {O}\mathsf {pFib}_{\mathscr {F}}\left (\mathsf {A}\right )\).

\begin{proof}[{proof of \cref{djm-00IG}.}]\label{jrb-001H-proof}
\par{}\begin{equation}
\begin {aligned}
\mathcal {P}\mathsf {sMod}^{\mathsf {p}}_{\mathscr {F}} \left (\mathsf {A},\mathcal {C}\mathsf {at}\right ) \downarrow  M
\simeq  \; & \mathcal {O}\mathsf {pFib}_{\mathsf {mod}}\left (\mathsf {A}\right ) \downarrow  \left [ \mathcal {E}\mathsf {l} \left ( M \right ) \xrightarrow {\pi _M} \mathsf {A} \right ] \\
\cong  \; & \mathcal {O}\mathsf {pFib}_{\mathsf {mod}}\left (\mathcal {E}\mathsf {l} \left ( M \right )\right )  \\
\simeq  \; & \mathcal {P}\mathsf {sMod}^{\mathsf {p}}_{\mathscr {F}} \left (\mathcal {E}\mathsf {l}\left ( M \right ),\mathcal {C}\mathsf {at}\right )
\end {aligned}
\end{equation}
The first and last equivalences follow from \hyperref[jrb-001P]{the definition of model opfibrations}, but the middle isomorphism is yet to be proven. Given that the result holds with \(\mathscr {F}\)-opfibrations replacing model opfibrations, it suffices to show that post-composition of \(\mathscr {F}\)-opfibrations by discrete model opfibrations preserves and reflects the property of being a model opfibration. This is proven below as \cref{jrb-001I} after establishing a technical result about \(\mathscr {F}\)-opfibrations in \cref{jrb-001R}.
\end{proof}

\begin{lemma}[{Inert slice projections are isomorphisms}]\label{jrb-001R}
\par{}
Given an \(\mathscr {F}\)-opfibration \(p : \mathsf {B} \to  \mathsf {A}\) with \(\mathsf {A}\) locally-discrete, the projection \(b \downarrow  p : b \downarrow  {\mathsf {B}}_{\mathsf {inert}} \to  p_b \downarrow  {\mathsf {A}}_{\mathsf {inert}}\) is an isomorphism.

\begin{proof}[{proof of \cref{jrb-001R}.}]\label{jrb-001R-proof}
\par{}
Let \(p_{\mathsf {inert}} : {\mathsf {B}}_{\mathsf {inert}} \to  {\mathsf {A}}_{\mathsf {inert}}\) denote the restriction of \(p\) to the inert subcategories. The map \(p_{\mathsf {inert}}\) is a discrete opfibration since the inerts of \(\mathsf {B}\) are, by definition, the chosen lifts of inerts in \(\mathsf {A}\). The result then follows by the definition of discrete opfibration.
\end{proof}
\end{lemma}
\par{}
We are now ready to prove preservation and reflection of the model-opfibration property under post-composition by discrete model opfibrations:

\begin{lemma}[{Post-composition with discrete model opfibrations preserves and reflects model opfibrations}]\label{jrb-001I}
\par{}
Given a locally-discrete \(\mathscr {F}\)-sketch, \(\mathsf {A}\), a discrete model opfibration \(p : \mathsf {B} \to  \mathsf {A}\) and an \(\mathscr {F}\)-opfibration \(q : \mathsf {C} \to  \mathsf {B}\), \(q\) is a model opfibration with respect to the induced \(\mathscr {F}\)-sketch structure on \(\mathsf {B}\) if and only if the composite \(p\;q\) is a model opfibration.

\begin{proof}[{proof of \cref{jrb-001I}.}]\label{jrb-001I-proof}
\par{}
By \cref{jrb-001G} it suffices to exhibit certain unique diagonal fillers in the appropriate contexts. First, to show compositionality, assume we have a marked cone \(S : \mathsf {J} \to  x \downarrow  {\mathsf {A}}_{\mathsf {inert}}\) for the \(\mathscr {F}\)-sketch structure on \(\mathsf {A}\), and a functor \(L : \mathsf {J} \to  \mathsf {C}\):

  \begin{center}
\begin {tikzcd}[row sep=1em]
	{\mathsf {J}} & {\mathsf {C}} \\
	& {\mathsf {B}} \\
	{x \downarrow  {\mathsf {A}}_{\mathsf {inert}}} & {\mathsf {A}}
	\arrow [""{name=0, anchor=center, inner sep=0}, "L", from=1-1, to=1-2]
	\arrow ["S"', from=1-1, to=3-1]
	\arrow ["q", from=1-2, to=2-2]
	\arrow ["{p}", from=2-2, to=3-2]
	\arrow [""{name=1, anchor=center, inner sep=0}, "{\mathsf {cod}}"', from=3-1, to=3-2]
	\arrow ["\circlearrowright "{marking, allow upside down}, draw=none, from=0, to=1]
\end {tikzcd}
\end{center}

By the model-opfibration property of \(p\) there is a unique morphism \(u : x \downarrow  {\mathsf {A}}_{\mathsf {inert}} \to  \mathsf {B}\) making the following diagram commute:

  \begin{center}
\begin {tikzcd}[row sep=1em]
	{\mathsf {J}} & {\mathsf {C}} \\
	& {\mathsf {B}} \\
	{x \downarrow  {\mathsf {A}}_{\mathsf {inert}}} & {\mathsf {A}}
	\arrow ["L", from=1-1, to=1-2]
	\arrow ["S"', from=1-1, to=3-1]
	\arrow ["q", from=1-2, to=2-2]
	\arrow ["{p}", from=2-2, to=3-2]
	\arrow ["{u}"{description}, from=3-1, to=2-2]
	\arrow ["{\mathsf {cod}}"', from=3-1, to=3-2]
\end {tikzcd}
  \end{center}

  the image of \(1_x \in  x \downarrow  {\mathsf {A}}_{\mathsf {inert}}\) under \(u\) picks an element \(y\) in the fibre of \(p\) over \(x\) such that \(u\) factors as \(x \downarrow  {\mathsf {A}}_{\mathsf {inert}} \xrightarrow {\left ( y \downarrow  p \right )^{-1}} y \downarrow  {\mathsf {B}}_{\mathsf {inert}} \xrightarrow {\mathsf {cod}} \mathsf {B}\), where \(\left ( y \downarrow  p \right )^{-1}\) exists by \cref{jrb-001R}. The map \(S' := \left ( y \downarrow  p \right )^{-1}\;S\) is then a marked cone for the \(\mathscr {F}\)-sketch \(\mathsf {B}\), and so we have a further lift \(v\) as in the diagram below:

  \begin{center}
\begin {tikzcd}[row sep=1em]
	{\mathsf {J}} & {\mathsf {C}} \\
	{y \downarrow  {\mathsf {B}}_{\mathsf {inert}}} & {\mathsf {B}} \\
	{x \downarrow  {\mathsf {A}}_{\mathsf {inert}}} & {\mathsf {A}}
	\arrow ["L", from=1-1, to=1-2]
	\arrow ["{S'}"', from=1-1, to=2-1]
	\arrow ["q", from=1-2, to=2-2]
	\arrow ["{v}", from=2-1, to=1-2]
	\arrow ["{\mathsf {cod}}", from=2-1, to=2-2]
	\arrow ["y \downarrow  p"', from=2-1, to=3-1]
	\arrow ["{p}", from=2-2, to=3-2]
	\arrow ["{u}"{description}, from=3-1, to=2-2]
	\arrow ["{\mathsf {cod}}"', from=3-1, to=3-2]
\end {tikzcd}
  \end{center}

  Because the map \(y \downarrow  p : y \downarrow  {\mathsf {B}}_{\mathsf {inert}} \to  x \downarrow  {\mathsf {A}}_{\mathsf {inert}}\) is an isomorphism, such a diagonal filler \(v\) is equivalent to a diagonal filler \(w : x \downarrow  {\mathsf {A}}_{\mathsf {inert}} \to  \mathsf {C}\) satisfying the  condition that \(q\;w = u\) and \(w\; \left ( y \downarrow  p \right )\; S' = L\). The first condition is \emph{a priori} stronger than necessary for a diagonal filler of the outside square, which need only satisfy \(p\;q\;w = \mathsf {cod}\). However, the seemingly weaker condition implies that \(q\;w\) is a diagonal filler for \(S\) against \(p\), and thus is equivalent to the stronger condition by the model-opfibration property of \(p\).
\par{}
Now consider two lifts of \(S\) with a \(\left ( p\;q \right )\)-vertical 2-cell between them:

  \begin{center}
\begin {tikzcd}[row sep=1em]
	{\mathsf {J}} & {\mathsf {C}} \\
	& {\mathsf {B}} \\
	{x \downarrow  {\mathsf {A}}_{\mathsf {inert}}} & {\mathsf {A}}
	\arrow [""{name=0, anchor=center, inner sep=0}, "L"', from=1-1, to=1-2]
	\arrow [""{name=1, anchor=center, inner sep=0}, "K", popup=1em, from=1-1, to=1-2]
	\arrow ["S"', from=1-1, to=3-1]
	\arrow ["q", from=1-2, to=2-2]
	\arrow ["p", from=2-2, to=3-2]
	\arrow [""{name=2, anchor=center, inner sep=0}, "{\mathsf {cod}}"', from=3-1, to=3-2]
	\arrow ["{\alpha }"{outer sep=1pt},shorten <=5pt, shorten >=5pt, Rightarrow, from=0, to=0|-1]
	\arrow ["\circlearrowright "{marking, allow upside down}, draw=none, from=2, to=0]
\end {tikzcd}
\end{center}

By the discreteness of \(p\) such an \(\alpha \) must moreover be \(q\)-vertical. We can therefore repeat the previous trick of lifting the cone \(S\) to a marked cone \(S' : \mathsf {J} \to  y \downarrow  {\mathsf {B}}_{\mathsf {inert}}\) such that \(\mathsf {cod}\; S' = q\;L = q\;K\):

  \begin{center}
\begin {tikzcd}[row sep=1em]
	{\mathsf {J}} & {\mathsf {C}} \\
	{y \downarrow  {\mathsf {B}}_{\mathsf {inert}}}& {\mathsf {B}} \\
	{x \downarrow  {\mathsf {A}}_{\mathsf {inert}}} & {\mathsf {A}}
	\arrow [""{name=0, anchor=center, inner sep=0}, "L"', from=1-1, to=1-2]
	\arrow [""{name=1, anchor=center, inner sep=0}, "K", popup=1em, from=1-1, to=1-2]
	\arrow ["S'"', from=1-1, to=2-1]
	\arrow ["y \downarrow  p"', from=2-1, to=3-1]
	\arrow ["q", from=1-2, to=2-2]
	\arrow ["p", from=2-2, to=3-2]
	\arrow [""{name=2, anchor=center, inner sep=0}, "{\mathsf {cod}}"', from=3-1, to=3-2]
	\arrow ["{\alpha }"{outer sep=1pt},shorten <=5pt, shorten >=5pt, Rightarrow, from=0, to=0|-1]
	\arrow ["\circlearrowright "{marking, allow upside down}, draw=none, from=2, to=0]
\end {tikzcd}
\end{center}

By the model-opfibration property of \(q\) we can then factor \(\alpha \) through \(S'\) in a unique way, giving a 2-cell \(\alpha ' : K' \Rightarrow  L'\) between the diagonal fillers of \(K\) and \(L\). Precomposing by the isomorphism \(\left ( y \downarrow  p \right )^{-1} : x \downarrow  {\mathsf {A}}_{\mathsf {inert}} \to  y \downarrow  {\mathsf {B}}_{\mathsf {inert}}\) then produces the unique factorisation of \(\alpha \) through \(S\).
  
\begin{remark}[{Is discreteness necessary?}]\label{jrb-002F}
\par{}
Though we've just appealed to the discreteness of \(p\) for the lifting of 2-cells, it wasn't strictly necessary; \(p\,q\) will be a model opfibration for \emph{any} model opfibration \(p\). This fact won't have any further relevance in this paper so its proof is relegated to \cref{jrb-002D} in the appendix.
\end{remark}
\par{}
The proof of the converse proceeds similarly. Assume now that \(p\;q : \mathsf {C} \to  \mathsf {A}\) is known to be a model opfibration, and that we have a lift along \(q\) of a marked cone \(S : \mathsf {J} \to  y \downarrow  {\mathsf {B}}_{\mathsf {inert}}\). Such a marked cone is by definition a lift of a marked cone for \(\mathsf {A}\) along an isomorphism \(y \downarrow  p : y \downarrow  {\mathsf {B}}_{\mathsf {inert}} \to  p_y \downarrow  {\mathsf {A}}_{\mathsf {inert}}\):

  \begin{center}
\begin {tikzcd}[row sep=1em]
	{\mathsf {J}} & {\mathsf {C}} \\
	{y \downarrow  {\mathsf {B}}_{\mathsf {inert}}} & {\mathsf {B}} \\
	{p_y \downarrow  {\mathsf {A}}_{\mathsf {inert}}} & {\mathsf {A}}
	\arrow ["L", from=1-1, to=1-2]
	\arrow ["S"', from=1-1, to=2-1]
	\arrow ["q", from=1-2, to=2-2]
	\arrow ["{\mathsf {cod}}"', from=2-1, to=2-2]
	\arrow ["{y \downarrow  p}"', from=2-1, to=3-1]
	\arrow ["p", from=2-2, to=3-2]
	\arrow ["{\mathsf {cod}}"', from=3-1, to=3-2]
\end {tikzcd}
  \end{center}

By the model-opfibration property of \(p\;q\), there's a unique diagonal filler \(u : p_y \downarrow  {\mathsf {A}}_{\mathsf {inert}} \to  \mathsf {C}\) which when precomposed by the isomorphism \(y \downarrow  p\) gives a map \(v : y \downarrow  {\mathsf {B}}_{\mathsf {inert}} \to  \mathsf {C}\). Such a map is unique with the property that \(p\;q\;v\;\left ( y \downarrow  p \right )^{-1} = \mathsf {cod}\) and \(v\;S = L\). This first property seems weaker than the necessary condition that \(q\;v = \mathsf {cod}\) — or equivalently \(q\;v\;\left ( y \downarrow  p \right )^{-1} = u\) — but is actually equivalent by the model-opfibration property of \(p\) as noted above. 
\par{}
The unique factorisation of 2-cells between lifts through \(S\) follows by an analogous argument.
\end{proof}
\end{lemma}
\par{}
This completes our proof of 
\cref{djm-00IG}. The same argument yields the following lax version of the slice theorem for models:
\begin{theorem}[{The lax slice theorem for \(\mathscr {F}\)-sketches}]\label{jrb-0026}
\par{}
	For \(\mathsf {A}\) a locally-discrete \hyperref[djm-00IA]{\(\mathscr {F}\)-sketch} and \(M : \mathsf {A} \to  \mathsf {Cat}\) a discrete pseudo-model, there is an equivalence:
	\begin{equation}
		\mathcal {P}\mathsf {sMod}^{\mathsf {lx}}_{\mathscr {F}} \left (\mathsf {A},\mathcal {C}\mathsf {at}\right ) \downarrow  M \simeq  \mathcal {P}\mathsf {sMod}^{\mathsf {lx}}_{\mathscr {F}} \left (\mathcal {E}\mathsf {l}(M),\mathcal {C}\mathsf {at}\right ).
	\end{equation}\end{theorem}
\par{}
By the equivalence between \(\mathcal {C}\mathsf {at}\)-valued pseudo-models for \(\Delta ^{\mathsf {op}}\) and \(\mathcal {D}\mathsf {bl}\) given in \cref{djm-00IE} we can use \cref{djm-00IG} to show that slices of \(\mathcal {D}\mathsf {bl}\) and \(\mathcal {D}\mathsf {bl}^{\mathsf {lx}}\) over vertically-discrete pseudo double categories are also “pseudo-sketchable”. In particular, we obtain an \(\mathscr {F}\)-sketch whose pseudo-models in \(\mathcal {C}\mathsf {at}\) are double barrels:

\begin{corollary}\label{jrb-001L}
\par{}\begin{equation}\mathcal {P}\mathsf {sMod}^{\mathsf {p}}_{\mathscr {F}} \left ( \left ( \Delta   \downarrow  [1] \right )^{\mathsf {op}},\mathcal {C}\mathsf {at}\right ) \simeq  \mathcal {D}\mathsf {bl} \downarrow  \mathbb {L}\mathsf {oose}\end{equation}
\begin{equation}\mathcal {P}\mathsf {sMod}^{\mathsf {lx}}_{\mathscr {F}} \left ( \left ( \Delta   \downarrow  [1] \right )^{\mathsf {op}},\mathcal {C}\mathsf {at}\right ) \simeq  \mathcal {D}\mathsf {bl}^{\mathsf {lx}} \downarrow  \mathbb {L}\mathsf {oose}\end{equation}
\begin{proof}[{proof of \cref{jrb-001L}.}]\label{jrb-001L-proof}
\par{}
The \(\mathscr {F}\)-sketch structure on \(\Delta ^{\mathsf {op}}\) is \emph{realised} in the sense that all marked cones are in fact limit cones. All representable presheaves on \(\Delta \), including \(\Delta  [1] = \yo _{[1]}\), therefore map these marked cones to limits, and are thus discrete pseudo-models. The equivalences \(\mathcal {P}\mathsf {sMod}^{\mathsf {p}}_{\mathscr {F}} \left ( \Delta ^{\mathsf {op}},\mathcal {C}\mathsf {at}\right ) \simeq  \mathcal {D}\mathsf {bl}\) and \(\mathcal {P}\mathsf {sMod}^{\mathsf {lx}}_{\mathscr {F}} \left ( \Delta ^{\mathsf {op}},\mathcal {C}\mathsf {at}\right ) \simeq  \mathcal {D}\mathsf {bl}^{\mathsf {lx}}\) map \(\Delta  [1]\) to \(\mathbb {L}\mathsf {oose}\), so the result follows from 
\cref{djm-00IG} and \cref{jrb-0026}.
\end{proof}
\end{corollary}
\subsection{Pseudo-bimodules}\label{jrb-003A}
\begin{explication}[{Pseudo-Models of \(\left ( \Delta  \downarrow  [1] \right )^{\mathsf {op}}\) as pseudo-bimodules}]\label{jrb-003B}
\par{}
Just as strict models for the limit sketch \(\left ( \Delta  \downarrow  [1] \right )^{\mathsf {op}}\) in a 1-category \(E\) with pullbacks present bimodules internal to \(E\) (\cref{jrb-0034}), \emph{pseudo}-models for the \(\mathscr {F}\)-sketch \(\left ( \Delta  \downarrow  [1] \right )^{\mathsf {op}}\) in a 2-category \(\mathcal {E}\) with pullbacks present a natural notion of \emph{pseudo}-bimodules internal to \(\mathcal {E}\).
\par{}
To give a pseudo-model \(M : \left ( \Delta  \downarrow  [1] \right )^{\mathsf {op}} \to  \mathcal {E}\) involves in particular specifying a diagram of the form:

  \begin{center}
\begin {tikzcd}
	{C_1} && {D_1} \\
	{C_0} & {\mathsf {Car}(M)} & {D_0}
	\arrow ["t"', shift right=3, from=1-1, to=2-1]
	\arrow ["s", shift left=3, from=1-1, to=2-1]
	\arrow ["t"', shift right=3, from=1-3, to=2-3]
	\arrow ["s", shift left=3, from=1-3, to=2-3]
	\arrow ["i"{description}, from=2-1, to=1-1]
	\arrow [dashed, from=2-2, to=1-1]
	\arrow [dashed, from=2-2, to=1-3]
	\arrow ["s", from=2-2, to=2-1]
	\arrow ["t"', from=2-2, to=2-3]
	\arrow ["i"{description}, from=2-3, to=1-3]
\end {tikzcd}
\end{center}

as is also true for the strict models. And again, for each \(f : [n] \to  [1] \in  \left ( \Delta  \downarrow  [1] \right )^{\mathsf {op}}\) it must be the case that \(M_f\) is given by a composite of the spans \(C_1 : C_0 \mathrel {\mkern 3mu\vcenter {\hbox {$\shortmid $}}\mkern -10mu{\to }} C_0\), \(\mathsf {Car}(M) : C_0 \mathrel {\mkern 3mu\vcenter {\hbox {$\shortmid $}}\mkern -10mu{\to }} D_0\) and \(D_1 : D_0 \mathrel {\mkern 3mu\vcenter {\hbox {$\shortmid $}}\mkern -10mu{\to }} D_0\). For example, if \(f\) is represented by the sequence \(0,0,1,1,1\) then we must have \(M_f \cong  C_1 \diamond  \mathsf {Car}(M) \diamond  D_1 \diamond  D_1\). Each \(M_f\) admits a canonical map to either \(C_1\), \(\mathsf {Car}(M)\) or \(D_1\) determined by the unique active \([1] \to  [n]\), providing notions of composition for \(C_1\) and \(D_1\) as well as actions of \(C_1\) and \(D_1\) on \(\mathsf {Car}(M)\). The difference between these pseudo-models and the strict models in \cref{jrb-0034} is that these actions and compositions are no longer required to be strictly associative, but are instead \emph{pseudo}-associative. For example, for \(f : [3] \to  [1]\) given by the sequence \(0,0,1,1\), the composites of the images under \(M\) of the two composable pairs of maps in \(\Delta  \downarrow  [1]\) below need only be \emph{isomorphic}, rather than equal:

  \begin{center}
\begin {tikzcd}[sep=small]
	0 & 0 & 1 & 1 \\
	0 & 0 && 1 \\
	0 &&& 1
	\arrow [equals, from=1-1, to=1-2]
	\arrow [from=1-2, to=1-3]
	\arrow [equals, from=1-3, to=1-4]
	\arrow [maps to, draw={rgb,255:red,179;green,179;blue,179}, from=2-1, to=1-1]
	\arrow [equals, from=2-1, to=2-2]
	\arrow [maps to, draw={rgb,255:red,179;green,179;blue,179}, from=2-2, to=1-2]
	\arrow [from=2-2, to=2-4]
	\arrow [maps to, draw={rgb,255:red,179;green,179;blue,179}, from=2-4, to=1-4]
	\arrow [maps to, draw={rgb,255:red,179;green,179;blue,179}, from=3-1, to=2-1]
	\arrow [from=3-1, to=3-4]
	\arrow [maps to, draw={rgb,255:red,179;green,179;blue,179}, from=3-4, to=2-4]
\end {tikzcd}
\qquad 
\begin {tikzcd}[sep=small]
	0 & 0 & 1 & 1 \\
	0 && 1 & 1 \\
	0 &&& 1
	\arrow [equals, from=1-1, to=1-2]
	\arrow [from=1-2, to=1-3]
	\arrow [equals, from=1-3, to=1-4]
	\arrow [maps to, draw={rgb,255:red,179;green,179;blue,179}, from=2-1, to=1-1]
	\arrow [from=2-1, to=2-3]
	\arrow [maps to, draw={rgb,255:red,179;green,179;blue,179}, from=2-3, to=1-3]
	\arrow [equals, from=2-3, to=2-4]
	\arrow [maps to, draw={rgb,255:red,179;green,179;blue,179}, from=2-4, to=1-4]
	\arrow [maps to, draw={rgb,255:red,179;green,179;blue,179}, from=3-1, to=2-1]
	\arrow [from=3-1, to=3-4]
	\arrow [maps to, draw={rgb,255:red,179;green,179;blue,179}, from=3-4, to=2-4]
\end {tikzcd}
\end{center}

This says that the result of acting on \(\mathsf {Car}(M)\) from the right and then from the left is merely \emph{isomorphic} to the result of acting from the left and then from the right, which is what one should expect of a pseudo-bimodule. The same argument with all the numbers in the above diagrams replaced by either \(0\) or \(1\) says that the composition in \(C_1\) and \(D_1\) is only associative up to isomorphism (indeed, \(C_1\) and \(D_1\) specify \emph{pseudo}-categories internal to \(\mathcal {E}\)). However, not every operation is weakened; the strict preservation of composition with inerts means that the “restriction” operations (images of inerts under \(M\)) \emph{strictly} commute with all other operations. For example, if some heteromorphism \(h : {\color {blue} c} \to  {\color {magenta} d}\) is acted upon on the right by a \(C_1\) morphism \(f : {\color {blue} c' \to  c}\), the domain of the result must be \emph{equal} to \({\color {blue} c'}\), rather than merely isomorphic to it. As another example, given the following composable path:
\begin{equation}{\color {blue} c \xrightarrow {f} c'} \xrightarrow {g} {\color {magenta} d \xrightarrow {h} d'}\end{equation}
the result of composing the first two arrows to give \({\color {blue} c} \xrightarrow {gf} {\color {magenta} d \xrightarrow {h} d'}\) and then restricting to \(gf\) is equal to the result of restricting to the first two arrows and then composing.
\par{}
It is this need to distinguish between the “inert” operations — which must remain strict — and the “active” operations — which merely hold up to isomorphism — that motivates the use of \(\mathscr {F}\)-category structure on limit sketches.
\end{explication}

\begin{definition}[{The biased sketch for pseudo-bimodules}]\label{jrb-003K}
\par{}
We refer to the corresponding truncated \(\mathscr {F}\)-sketch \(\left ( \Delta _{\leq  4} \downarrow  [1] \right )^{\mathsf {op}}\) as the \textbf{biased sketch for pseudo-bimodules}, noting that it presents the same notions of “labelled” composition as above, but only up to 4-ary composites.
\end{definition}

\subsection{From double barrels to loose-bimodules}\label{jrb-003D}
\par{}
In this section we recall in detail the precise definition of \emph{double barrel} and observe how they present the data for a loose bimodule.

\begin{definition}[{walking loose arrow}]\label{djm-0047}
\par{}The \textbf{walking loose arrow}, \(\mathbb {L}\mathsf {oose}\), is the free strict double category generated by a single loose arrow \(\ell  \colon   0 \mathrel {\mkern 3mu\vcenter {\hbox {$\shortmid $}}\mkern -10mu{\to }} 1\).\end{definition}

\begin{explication}[{Walking loose arrow}]\label{ssl-001G}
\par{}
  The \hyperref[djm-0047]{walking loose arrow}, \(\mathbb {L}\mathsf {oose}\), is a double category  with two objects, three loose arrows (two identities along with the walking loose arrow \(\ell \)), two identity tight arrows, and three identity squares. These are all pictured below.

  \begin{center}\begin {tikzcd}
    0 & 0 & 1 & 1 \\
    0 & 0 & 1 & 1
    \arrow ["\shortmid "{marking}, equals, from=1-1, to=1-2]
    \arrow [equals, from=1-1, to=2-1]
    \arrow ["\shortmid "{marking}, from=1-2, to=1-3]
    \arrow [equals, from=1-2, to=2-2]
    \arrow ["\shortmid "{marking}, equals, from=1-3, to=1-4]
    \arrow [equals, from=1-3, to=2-3]
    \arrow [equals, from=1-4, to=2-4]
    \arrow ["\shortmid "{marking}, equals, from=2-1, to=2-2]
    \arrow ["\shortmid "{marking}, from=2-2, to=2-3]
    \arrow ["\shortmid "{marking}, equals, from=2-3, to=2-4]
  \end {tikzcd}\end{center}\end{explication}

\begin{definition}[{double barrel}]\label{ssl-0068}
\par{}A \textbf{double barrel} is a (necessarily strict) double functor \(M : \mathbb {M} \to  \mathbb {L}\mathsf {oose}\) into the \hyperref[djm-0047]{walking loose arrow}. The 2-category of double barrels is the (strict) slice 2-category \(\mathcal {D}\mathsf {bl} \downarrow  \mathbb {L}\mathsf {oose}\) over the \hyperref[djm-0047]{walking loose arrow}. Explicitly, it is the 2-category whose:
\begin{enumerate}\item{}
	objects are double barrels \(M : \mathbb {M} \to  \mathbb {L}\mathsf {oose}\).

\item{}
morphisms are strictly commuting triangles

\begin{center}
\begin {tikzcd}
	{\mathbb {M}} && {\mathbb {N}} \\
	& \mathbb {L}\mathsf {oose}
	\arrow ["{F}", from=1-1, to=1-3]
	\arrow ["{M}"', from=1-1, to=2-2]
	\arrow ["{N}", from=1-3, to=2-2]
\end {tikzcd}
\end{center}

where \(F\) is a pseudo-double functor

\item{}2-cells are tight transformations \(\alpha  \colon  F \Rightarrow  G\) so that \(N \alpha  = M\).\end{enumerate}
A \emph{lax} morphism of double barrels is a 1-cell as above but where \(F\) is a \emph{lax} double functor. These form the 1-cells of the 2-category \(\mathcal {D}\mathsf {bl}^{\mathsf {lx}} \downarrow  \mathbb {L}\mathsf {oose}\).
\end{definition}

\begin{definition}[{source and target of a double barrel}]\label{djm-004A}
\par{}
  Note that there are two double functors from the terminal double category into the walking loose arrow. These are \(0 \colon  \bullet  \to  \mathbb {L}\mathsf {oose}\) sending the unique object to \(0\) and \(1 \colon  \bullet  \to  \mathbb {L}\mathsf {oose}\) sending the unique object to \(1\).
\par{}Let \(M \colon  \mathbb {M} \to  \mathbb {L}\mathsf {oose}\) be a \hyperref[ssl-0068]{double barrel}. The \textbf{source} \(M_0\) of \(M\) is its pullback along (i.e. fibre above) \(0 : \bullet  \to  \mathbb {L}\mathsf {oose}\). Its \textbf{target} \(M_1\) is its pullback along \(1 : \bullet  \to  \mathbb {L}\mathsf {oose}\).\par{}Together, taking the source and target of a double barrel gives a cartesian 2-functor
\begin{equation}((-)_0, (-)_1) \colon  \mathcal {D}\mathsf {bl} \downarrow  \mathbb {L}\mathsf {oose} \to  \mathcal {D}\mathsf {bl} \times  \mathcal {D}\mathsf {bl}\end{equation}\end{definition}

\begin{notation}[{double barrel}]\label{ssl-0016}
\par{}
  We often denote a \hyperref[ssl-0068]{double barrel} \(M \colon  \mathbb {M} \to  \mathbb {L}\mathsf {oose}\) as \begin{equation}\mathbb {M}  \colon  \mathbb {D}_0 \mathrel {\mkern 3mu\vcenter {\hbox {$\shortmid $}}\mkern -10mu{\to }} \mathbb {D}_1\end{equation} where \(\mathbb {D}_0 \) and \(\mathbb {D}_1 \) are respectively the \hyperref[djm-004A]{source} and \hyperref[djm-004A]{target} of  \({M}\).
\par{}
  For double categories \(\mathbb {D}_0\) and \(\mathbb {D}_1\), a \textbf{\((\mathbb {D}_0, \mathbb {D}_1)\)-double barrel} is a \hyperref[ssl-0068]{double barrel} \(\mathbb {M} \colon  \mathbb {D}_0 \mathrel {\mkern 3mu\vcenter {\hbox {$\shortmid $}}\mkern -10mu{\to }} \mathbb {D}_1\).
\end{notation}

\begin{definition}[{Carrier of a double barrel}]\label{ssl-0036}
\par{}
  Let \(M : \mathbb {M} \to  \mathbb {L}\mathsf {oose}\) be a double barrel. Its \textbf{carrier} \(\mathsf {Car} (\mathbb {M})\) is the category whose: 
  \begin{itemize}\item{}objects are loose morphisms of \(\mathbb {M}\) living over the walking loose arrow. We refer to such loose morphisms as the \textbf{heteromorphisms} of \(\mathbb {M}\).
    \item{}morphisms are squares of \(\mathbb {M}\) living over the identity square of the walking loose arrow, referred to as the \textbf{heterocells} of \(\mathbb {M}\).\end{itemize}\end{definition}

\begin{explication}[{double barrels as bimodules between double categories}]\label{ssl-002C}
\par{}
  Let \(M \colon  \mathbb {M} \to  \mathbb {L}\mathsf {oose}\)  be a \hyperref[ssl-0068]{double barrel}. In the diagram below, \(\alpha  \colon  x \to  y\) is a morphism in  \(\mathsf {Car}(M)\), the \hyperref[ssl-0036]{carrier} of the double barrel.
\begin{center}
  \begin {tikzcd}
    \textcolor {rgb,255:red,51;green,102;blue,255}{{m_0}} & \textcolor {rgb,255:red,214;green,92;blue,214}{{m _1}} \\
    \textcolor {rgb,255:red,51;green,102;blue,255}{{n_0}} & \textcolor {rgb,255:red,214;green,92;blue,214}{{n_1}} & {\mathbb {M}} \\
    \textcolor {rgb,255:red,153;green,153;blue,153}{0} & \textcolor {rgb,255:red,153;green,153;blue,153}{1} \\
    \textcolor {rgb,255:red,153;green,153;blue,153}{0} & \textcolor {rgb,255:red,153;green,153;blue,153}{1} & {\mathbb {L}\mathsf {oose}}
    \arrow [""{name=0, anchor=center, inner sep=0}, "x", "\shortmid "{marking}, from=1-1, to=1-2]
    \arrow [color={rgb,255:red,51;green,102;blue,255}, from=1-1, to=2-1]
    \arrow [color={rgb,255:red,214;green,92;blue,214}, from=1-2, to=2-2]
    \arrow [""{name=1, anchor=center, inner sep=0}, "y"', "\shortmid "{marking}, from=2-1, to=2-2]
    \arrow [from=2-3, to=4-3]
    \arrow [""{name=2, anchor=center, inner sep=0}, "\shortmid "{marking, text={rgb,255:red,153;green,153;blue,153}}, color={rgb,255:red,153;green,153;blue,153}, from=3-1, to=3-2]
    \arrow [color={rgb,255:red,153;green,153;blue,153}, equals, from=3-1, to=4-1]
    \arrow [color={rgb,255:red,153;green,153;blue,153}, equals, from=3-2, to=4-2]
    \arrow [""{name=3, anchor=center, inner sep=0}, "\shortmid "{marking, text={rgb,255:red,153;green,153;blue,153}}, color={rgb,255:red,153;green,153;blue,153}, from=4-1, to=4-2]
    \arrow ["\alpha "{description}, draw=none, from=0, to=1]
    \arrow ["\mathrm {id}"{description}, color={rgb,255:red,153;green,153;blue,153}, draw=none, from=2, to=3]
  \end {tikzcd}
\end{center}\par{}
  We can act on the carrier on the left by the category of loose morphisms and squares in the source \(\mathbb {M}_0\). 
\begin{center}
  \begin {tikzcd}
    \textcolor {rgb,255:red,51;green,102;blue,255}{{m_0'}} & \textcolor {rgb,255:red,51;green,102;blue,255}{{m_0}} & \textcolor {rgb,255:red,214;green,92;blue,214}{{m_1}} & {} & \textcolor {rgb,255:red,51;green,102;blue,255}{{m_0'}} & \textcolor {rgb,255:red,214;green,92;blue,214}{{m_1}} \\
    \textcolor {rgb,255:red,51;green,102;blue,255}{{n_0'}} & \textcolor {rgb,255:red,51;green,102;blue,255}{{n_0}} & \textcolor {rgb,255:red,214;green,92;blue,214}{{n_1}} & {} & \textcolor {rgb,255:red,51;green,102;blue,255}{{n_0'}} & \textcolor {rgb,255:red,214;green,92;blue,214}{{n_1}} & {\mathbb {M}} \\
    \textcolor {rgb,255:red,153;green,153;blue,153}{0} & \textcolor {rgb,255:red,153;green,153;blue,153}{0} & \textcolor {rgb,255:red,153;green,153;blue,153}{1} & {} & \textcolor {rgb,255:red,153;green,153;blue,153}{0} & \textcolor {rgb,255:red,153;green,153;blue,153}{1} \\
    \textcolor {rgb,255:red,153;green,153;blue,153}{0} & \textcolor {rgb,255:red,153;green,153;blue,153}{0} & \textcolor {rgb,255:red,153;green,153;blue,153}{1} & {} & \textcolor {rgb,255:red,153;green,153;blue,153}{0} & \textcolor {rgb,255:red,153;green,153;blue,153}{1} & {\mathbb {L}\mathsf {oose}}
    \arrow [""{name=0, anchor=center, inner sep=0}, "{f}", "\shortmid "{marking, text={rgb,255:red,51;green,102;blue,255}}, color={rgb,255:red,51;green,102;blue,255}, from=1-1, to=1-2]
    \arrow [draw={rgb,255:red,51;green,102;blue,255}, from=1-1, to=2-1]
    \arrow [""{name=1, anchor=center, inner sep=0}, "x", "\shortmid "{marking}, from=1-2, to=1-3]
    \arrow [draw={rgb,255:red,51;green,102;blue,255}, from=1-2, to=2-2]
    \arrow [draw={rgb,255:red,214;green,92;blue,214}, from=1-3, to=2-3]
    \arrow ["{=}"{description}, draw=white, from=1-4, to=2-4]
    \arrow [""{name=2, anchor=center, inner sep=0}, "{ x\, f}", "\shortmid "{marking}, from=1-5, to=1-6]
    \arrow [draw={rgb,255:red,51;green,102;blue,255}, from=1-5, to=2-5]
    \arrow [draw={rgb,255:red,214;green,92;blue,214}, from=1-6, to=2-6]
    \arrow [""{name=3, anchor=center, inner sep=0}, "{g}"', "\shortmid "{marking, text={rgb,255:red,51;green,102;blue,255}}, color={rgb,255:red,51;green,102;blue,255}, from=2-1, to=2-2]
    \arrow [""{name=4, anchor=center, inner sep=0}, "y"', "\shortmid "{marking}, from=2-2, to=2-3]
    \arrow [""{name=5, anchor=center, inner sep=0}, "{y g}"', "\shortmid "{marking}, from=2-5, to=2-6]
    \arrow [from=2-7, to=4-7]
    \arrow ["\shortmid "{marking, text={rgb,255:red,153;green,153;blue,153}}, draw={rgb,255:red,153;green,153;blue,153}, equals, from=3-1, to=3-2]
    \arrow [draw={rgb,255:red,153;green,153;blue,153}, equals, from=3-1, to=4-1]
    \arrow ["\shortmid "{marking, text={rgb,255:red,153;green,153;blue,153}}, draw={rgb,255:red,153;green,153;blue,153}, from=3-2, to=3-3]
    \arrow [draw={rgb,255:red,153;green,153;blue,153}, equals, from=3-2, to=4-2]
    \arrow [draw={rgb,255:red,153;green,153;blue,153}, equals, from=3-3, to=4-3]
    \arrow ["{=}"{description}, draw=white, from=3-4, to=4-4]
    \arrow ["\shortmid "{marking, text={rgb,255:red,153;green,153;blue,153}}, draw={rgb,255:red,153;green,153;blue,153}, from=3-5, to=3-6]
    \arrow [draw={rgb,255:red,153;green,153;blue,153}, equals, from=3-5, to=4-5]
    \arrow [draw={rgb,255:red,153;green,153;blue,153}, equals, from=3-6, to=4-6]
    \arrow ["\shortmid "{marking, text={rgb,255:red,153;green,153;blue,153}}, draw={rgb,255:red,153;green,153;blue,153}, equals, from=4-1, to=4-2]
    \arrow ["\shortmid "{marking, text={rgb,255:red,153;green,153;blue,153}}, draw={rgb,255:red,153;green,153;blue,153}, from=4-2, to=4-3]
    \arrow ["\shortmid "{marking, text={rgb,255:red,153;green,153;blue,153}}, draw={rgb,255:red,153;green,153;blue,153}, from=4-5, to=4-6]
    \arrow ["{\beta }"{description}, color={rgb,255:red,51;green,102;blue,255}, shorten <=4pt, shorten >=4pt, draw=none, from=0, to=3]
    \arrow ["\alpha "{description}, shorten <=4pt, shorten >=4pt, draw=none, from=1, to=4]
    \arrow ["{x\,\beta }"{description}, shorten <=4pt, shorten >=4pt, draw=none, from=2, to=5]
  \end {tikzcd}
\end{center}\par{}
  We can also act on the carrier on the right by the category of loose morphisms and squares in the target \(\mathbb {M}_1\).
\begin{center}
  \begin {tikzcd}
    \textcolor {rgb,255:red,51;green,102;blue,255}{{m_0}} & \textcolor {rgb,255:red,214;green,92;blue,214}{{m _1}} & \textcolor {rgb,255:red,214;green,92;blue,214}{{m_1'}} & {} & \textcolor {rgb,255:red,51;green,102;blue,255}{{m_0}} & \textcolor {rgb,255:red,214;green,92;blue,214}{{m_1'}} \\
    \textcolor {rgb,255:red,51;green,102;blue,255}{{n_0}} & \textcolor {rgb,255:red,214;green,92;blue,214}{{n_1}} & \textcolor {rgb,255:red,214;green,92;blue,214}{{n_1'}} & {} & \textcolor {rgb,255:red,51;green,102;blue,255}{{n_0}} & \textcolor {rgb,255:red,214;green,92;blue,214}{{n_1'}} & {\mathbb {M}} \\
    \textcolor {rgb,255:red,153;green,153;blue,153}{0} & \textcolor {rgb,255:red,153;green,153;blue,153}{1} & \textcolor {rgb,255:red,153;green,153;blue,153}{1} & {} & \textcolor {rgb,255:red,153;green,153;blue,153}{0} & \textcolor {rgb,255:red,153;green,153;blue,153}{1} \\
    \textcolor {rgb,255:red,153;green,153;blue,153}{0} & \textcolor {rgb,255:red,153;green,153;blue,153}{1} & \textcolor {rgb,255:red,153;green,153;blue,153}{1} & {} & \textcolor {rgb,255:red,153;green,153;blue,153}{0} & \textcolor {rgb,255:red,153;green,153;blue,153}{1} & {\mathbb {L}\mathsf {oose}}
    \arrow [""{name=0, anchor=center, inner sep=0}, "x", "\shortmid "{marking}, from=1-1, to=1-2]
    \arrow [draw={rgb,255:red,51;green,102;blue,255}, from=1-1, to=2-1]
    \arrow [""{name=1, anchor=center, inner sep=0}, "{h}", "\shortmid "{marking, text={rgb,255:red,214;green,92;blue,214}}, color={rgb,255:red,214;green,92;blue,214}, from=1-2, to=1-3]
    \arrow [draw={rgb,255:red,214;green,92;blue,214}, from=1-2, to=2-2]
    \arrow [draw={rgb,255:red,214;green,92;blue,214}, from=1-3, to=2-3]
    \arrow ["{=}"{description}, draw=white, from=1-4, to=2-4]
    \arrow [""{name=2, anchor=center, inner sep=0}, "{h\,x}", "\shortmid "{marking}, from=1-5, to=1-6]
    \arrow [draw={rgb,255:red,51;green,102;blue,255}, from=1-5, to=2-5]
    \arrow [draw={rgb,255:red,214;green,92;blue,214}, from=1-6, to=2-6]
    \arrow [""{name=3, anchor=center, inner sep=0}, "y"', "\shortmid "{marking}, from=2-1, to=2-2]
    \arrow [""{name=4, anchor=center, inner sep=0}, "{k}"', "\shortmid "{marking, text={rgb,255:red,214;green,92;blue,214}}, color={rgb,255:red,214;green,92;blue,214}, from=2-2, to=2-3]
    \arrow [""{name=5, anchor=center, inner sep=0}, "{k\,y}"', "\shortmid "{marking}, from=2-5, to=2-6]
    \arrow [from=2-7, to=4-7]
    \arrow ["\shortmid "{marking, text={rgb,255:red,153;green,153;blue,153}}, draw={rgb,255:red,153;green,153;blue,153}, from=3-1, to=3-2]
    \arrow [draw={rgb,255:red,153;green,153;blue,153}, equals, from=3-1, to=4-1]
    \arrow ["\shortmid "{marking, text={rgb,255:red,153;green,153;blue,153}}, draw={rgb,255:red,153;green,153;blue,153}, equals, from=3-2, to=3-3]
    \arrow [draw={rgb,255:red,153;green,153;blue,153}, equals, from=3-2, to=4-2]
    \arrow [draw={rgb,255:red,153;green,153;blue,153}, equals, from=3-3, to=4-3]
    \arrow ["{=}"{description}, draw=white, from=3-4, to=4-4]
    \arrow ["\shortmid "{marking, text={rgb,255:red,153;green,153;blue,153}}, draw={rgb,255:red,153;green,153;blue,153}, from=3-5, to=3-6]
    \arrow [draw={rgb,255:red,153;green,153;blue,153}, equals, from=3-5, to=4-5]
    \arrow [draw={rgb,255:red,153;green,153;blue,153}, equals, from=3-6, to=4-6]
    \arrow ["\shortmid "{marking, text={rgb,255:red,153;green,153;blue,153}}, draw={rgb,255:red,153;green,153;blue,153}, from=4-1, to=4-2]
    \arrow ["\shortmid "{marking, text={rgb,255:red,153;green,153;blue,153}}, draw={rgb,255:red,153;green,153;blue,153}, equals, from=4-2, to=4-3]
    \arrow ["\shortmid "{marking, text={rgb,255:red,153;green,153;blue,153}}, draw={rgb,255:red,153;green,153;blue,153}, from=4-5, to=4-6]
    \arrow ["\alpha "{description}, shorten <=4pt, shorten >=4pt, draw=none, from=0, to=3]
    \arrow ["{\gamma }"{description}, color={rgb,255:red,214;green,92;blue,214}, shorten <=4pt, shorten >=4pt, draw=none, from=1, to=4]
    \arrow ["{\gamma  \, \alpha }"{description}, shorten <=4pt, shorten >=4pt, draw=none, from=2, to=5]
  \end {tikzcd}
\end{center}\par{}
  Since these actions are given by composition in a (pseudo-)double category \(\mathbb {M}\), they are unital and associative up to coherent isomorphism. The actual coherences themselves are just those from the double category \(\mathbb {M}\); we simply interpret them as an interaction between left and right action, or between left action and composition in the left double category, and so on, as determined by the labelling.
\end{explication}

\begin{example}[{Hom double barrels}]\label{jrb-003F}
\par{}
Given any pseudo double category \(\mathbb {C}\), there is a \textbf{hom double barrel} \(C(-,-) : \mathbb {C} \mathrel {\mkern 3mu\vcenter {\hbox {$\shortmid $}}\mkern -10mu{\to }} \mathbb {C}\) given by the projection \(\mathbb {C} \times  \mathbb {L}\mathsf {oose} \to  \mathbb {L}\mathsf {oose}\).
\end{example}

\begin{remark}[{Source, target and hom from the \(\mathscr {F}\)-sketch perspective}]\label{jrb-003G}
\par{}
Defining the source, target and hom maps for loose bimodules from the \(\mathscr {F}\)-category perspective proceeds in exactly the same way as for bimodules from the limit sketch perspective (\cref{jrb-003E}), since the maps:

  \begin{center}
\begin {tikzcd}[column sep=5em]
{\left ( \Delta  \downarrow  [1] \right )^{\mathsf {op}}} & {\left ( \Delta  \downarrow  [0] \right )^{\mathsf {op}}}
\arrow ["{\Delta  \downarrow  \mathsf {s}_{0}}"{description}, from=1-1, to=1-2]
\arrow ["{\Delta  \downarrow  \mathsf {d}_{0}}"', popup=1em, from=1-2, to=1-1]
\arrow ["{\Delta  \downarrow  \mathsf {d}_{1}}", popdown=1em, from=1-2, to=1-1]
\end {tikzcd}
\end{center}

are moreover morphisms of \(\mathscr {F}\)-sketches.
\end{remark}

\begin{explication}[{Biased pseudo-bimodules in \(\mathcal {C}\mathsf {at}\), explicitly}]\label{jrb-003C}
\par{}
We end this section with a detailed elaboration of the data and coherence conditions provided by the \emph{biased} presentation of pseudo-bimodules in \(\mathcal {C}\mathsf {at}\) to observe that it produces the “right” notion of loose bimodule, and to establish some terminology for loose bimodules which we shall use in later sections.
\par{}
Let \(\mathbb {C}\) and \(\mathbb {D}\) denote the source and target double categories of a biased pseudo-bimodule \(M : \left ( \Delta _{\leq  4} \downarrow  [1] \right )^{\mathsf {op}} \to  \mathcal {C}\mathsf {at}\). From this pseudo-model we obtain the following data:
  \begin{itemize}\item{}from the image of \((0) \to  (0,1) \leftarrow  (1)\) (viewing objects of \(\Delta _{\leq  4} \downarrow  [1]\) as sequences) we get a span of categories \(\mathsf {Tight}({C}) \leftarrow  \mathsf {Car} (\mathbb {M}) \rightarrow  \mathsf {Tight}({D})\), called the \emph{carrier}. An object \(x\in  \mathsf {Car} (\mathbb {M})\) mapping to \(c\) and \(d\) is denoted \(x:c\mathrel {\mkern 3mu\vcenter {\hbox {$\shortmid $}}\mkern -10mu{\to }} d\)
    and called a (loose) heteromorphism. A morphism \(\alpha :x\to  x'\) in \(\mathsf {Car} (\mathbb {M})\) mapping to \(f:c\to  c'\) and \(g:d\to  d'\) is called a heterocell and denoted
    as a cell like so:
    
\begin{center}\begin {tikzcd}
        c & d \\
        {c'} & {d'}
        \arrow [""{name=0, anchor=center, inner sep=0}, "x", "\shortmid "{marking}, from=1-1, to=1-2]
        \arrow ["f"{description}, from=1-1, to=2-1]
        \arrow ["g"{description}, from=1-2, to=2-2]
        \arrow [""{name=1, anchor=center, inner sep=0}, "{x'}"', from=2-1, to=2-2]
        \arrow ["\alpha "{description}, draw=none, from=0, to=1]
      \end {tikzcd}\end{center}

    We shall call a heterocell mapping to identities on both sides \emph{globular}.
    
    \item{}from the active maps \((0,1) \to  \left ( \underline {\mathbf {0}},0,\underline {\mathbf {1}} \right )\) and \((0,1) \to  \left ( \underline {\mathbf {0}},1,\underline {\mathbf {1}} \right )\)
we get respectively the “action” functors \(\mathsf {Loose}({C})\times _{\mathsf {Tight}({C})} \mathsf {Car}  (\mathbb {M})\to  \mathsf {Car} (\mathbb {M})\) and \(\mathsf {Car} (\mathbb {M})\times _{\mathsf {Tight}({D})} \mathsf {Loose}({D})\to  \mathsf {Car} (\mathbb {M})\), both denoted by \(\odot \).
    \item{}from the compositors for \((0,1) \to  (\underline {\mathbf {0}}, 0, \underline {\mathbf {1}}) \to  \left ( 0,1 \right )\) and
    \(\left ( 0,1 \right ) \to  \left ( \underline {\mathbf {0}}, 1 , \underline {\mathbf {1}} \right ) \to  \left ( 0,1 \right )\) we have for all heteromorphisms \(x:c\mathrel {\mkern 3mu\vcenter {\hbox {$\shortmid $}}\mkern -10mu{\to }} d\) globular unitors \(\mathrm {id}_c \odot  x\cong  x\) and \(x\odot  \mathrm {id}_d\cong  x\) which are invertible and natural in \(x\).
    \item{}from the compositors for each of the commuting squares:

  \begin{center}
  \begin {tikzcd}[ampersand replacement=\ampersand ,column sep=1em]
	\left ( 0,1 \right ) \ampersand  \left ( 0,0,1 \right ) \\
	\left ( 0,0,1 \right ) \ampersand  \left ( 0,0,0,1 \right )
	\arrow ["{\mathsf {d}_{1}}",from=1-1, to=1-2]
	\arrow ["{\mathsf {d}_{1}}"',""{name=0, anchor=center, inner sep=0}, from=1-1, to=2-1]
	\arrow ["{\mathsf {d}_{1}}",""{name=1, anchor=center, inner sep=0}, from=1-2, to=2-2]
	\arrow ["{\mathsf {d}_{2}}"',from=2-1, to=2-2]
	\arrow ["\circlearrowright "{marking, allow upside down}, draw=none, from=0, to=1]
  \end {tikzcd}
  \quad 
    \begin {tikzcd}[ampersand replacement=\ampersand ,column sep=1em]
	\left ( 0,1 \right ) \ampersand  \left ( 0,1,1 \right ) \\
	\left ( 0,1,1 \right ) \ampersand  \left ( 0,1,1,1 \right )
	\arrow ["{\mathsf {d}_{1}}",from=1-1, to=1-2]
	\arrow ["{\mathsf {d}_{1}}"',""{name=0, anchor=center, inner sep=0}, from=1-1, to=2-1]
	\arrow ["{\mathsf {d}_{1}}",""{name=1, anchor=center, inner sep=0}, from=1-2, to=2-2]
	\arrow ["{\mathsf {d}_{2}}"',from=2-1, to=2-2]
	\arrow ["\circlearrowright "{marking, allow upside down}, draw=none, from=0, to=1]
  \end {tikzcd}
    \quad 
    \begin {tikzcd}[ampersand replacement=\ampersand ,column sep=1em]
	\left ( 0,1 \right ) \ampersand  \left ( 0,1,1 \right ) \\
	\left ( 0,0,1 \right ) \ampersand  \left ( 0,0,1,1 \right )
	\arrow ["{\mathsf {d}_{1}}",from=1-1, to=1-2]
	\arrow ["{\mathsf {d}_{1}}"',""{name=0, anchor=center, inner sep=0}, from=1-1, to=2-1]
	\arrow ["{\mathsf {d}_{1}}",""{name=1, anchor=center, inner sep=0}, from=1-2, to=2-2]
	\arrow ["{\mathsf {d}_{2}}"',from=2-1, to=2-2]
	\arrow ["\circlearrowright "{marking, allow upside down}, draw=none, from=0, to=1]
\end {tikzcd}
    \end{center}

we get globular, invertible associators of the following three forms, natural in all their variables:
	\begin{equation}
		\begin {aligned}
(p'\odot  p)\odot  x &\cong  p'\odot  (p\odot  x) \\
x\odot  (q\odot  q') &\cong  (x\odot  q)\odot  q' \\
(p'\odot  x)\odot  q &\cong  p'\odot  (x\odot  q)
		\end {aligned}
	\end{equation}
    \item{}from the coherence equations from the two ways of defining the active map \([1] \to  [4]\) described in \cref{jrb-003J} we get four pentagon identities, corresponding to the four non-trivial non-decreasing sequences in \(\{0,1\}\) of length five. These relate the associators for strings of three pro-morphisms and one heteromorphism of the forms
      \(p''\odot  p'\odot  p\odot  x\), \(p'\odot  p\odot  x\odot  q\), \(p\odot  x\odot  q\odot  q'\), and \(x\odot  q\odot  q'\odot  q''\).
    
    \item{}from the coherence conditions for the two equal paths of morphisms in \(\Delta \):

  \begin{center}
\begin {tikzcd}[ampersand replacement=\ampersand ,row sep=0.1em]
	\ampersand  {[2]} \\
	{[1]} \ampersand \ampersand  {[3]} \ampersand  {[2]} \\
	\ampersand  {[2]}
	\arrow ["{\mathsf {d}_{2}}", from=1-2, to=2-3]
	\arrow ["{\mathsf {d}_{1}}", from=2-1, to=1-2]
	\arrow ["{\mathsf {d}_{1}}"', from=2-1, to=3-2]
	\arrow ["{\mathsf {s}_{1}}", from=2-3, to=2-4]
	\arrow ["{\mathsf {d}_{1}}"', from=3-2, to=2-3]
\end {tikzcd}
    \end{center}

      we get triangle identities of two kinds — corresponding to the two non-degenerate sequences of length 3 — which relate the left unitor of \(x:c\mathrel {\mkern 3mu\vcenter {\hbox {$\shortmid $}}\mkern -10mu{\to }} d\) to the right unitor of \(p:c\mathrel {\mkern 3mu\vcenter {\hbox {$\shortmid $}}\mkern -10mu{\to }} c'\)
      and the right unitor of \(x\) to the left unitor of \(q:d\mathrel {\mkern 3mu\vcenter {\hbox {$\shortmid $}}\mkern -10mu{\to }} d'\).
    \end{itemize}\par{}For pseudo-bimodules \(M :\mathbb {C}\mathrel {\mkern 3mu\vcenter {\hbox {$\shortmid $}}\mkern -10mu{\Rightarrow }}\mathbb {D}\) and \(M':\mathbb {C}'\mathrel {\mkern 3mu\vcenter {\hbox {$\shortmid $}}\mkern -10mu{\Rightarrow }}\mathbb {D}'\) a \textbf{pseudo-morphism} of pseudo-bimodules \( \phi  :M \Rightarrow   M' : \left ( \Delta _{\leq  4} \downarrow  [1] \right )^{\mathsf {op}} \to  \mathcal {C}\mathsf {at}\) then induces:
  \begin{itemize}\item{}pseudo-double functors \(F_0:\mathbb {C}\to  \mathbb {C}'\) and \(F_1:\mathbb {D}\to  \mathbb {D}'\)
    \item{}a functor \(F:\mathsf {Car} (\mathbb {M})\to  \mathsf {Car} (\mathbb {M}')\) on the carriers making a morphism of spans with \(F_0\) and \(F_1\); i.e., \(F\) respects the domain and codomain of heteromorphisms and heterocells.
    \item{}globular invertible laxators \(F(p\odot  x)\to  F_0(p)\odot  F(x)\) and \(F(x\odot  q)\to  F(x)\odot  F_1(q)\) for all \(p:c\mathrel {\mkern 3mu\vcenter {\hbox {$\shortmid $}}\mkern -10mu{\to }} c'\) and \(x:c\mathrel {\mkern 3mu\vcenter {\hbox {$\shortmid $}}\mkern -10mu{\to }} d,q:d\mathrel {\mkern 3mu\vcenter {\hbox {$\shortmid $}}\mkern -10mu{\to }} d'\) corresponding respectively to the pseudo-naturality cells for the morphisms \(\left ( 0,1 \right ) \to  \left ( \underline {\mathbf {0}},0,\underline {\mathbf {1}} \right )\) and \(\left ( 0,1 \right ) \to  \left ( \underline {\mathbf {0}},1,\underline {\mathbf {1}} \right )\) natural in \(p\), \(q\) and \(x\).
    \end{itemize}
  These data satisfy:
  \begin{itemize}\item{}three associativity hexagons for laxators, equating the two ways to get from \(F_0(p)\odot  F(x)\odot  F_1(q)\) to \(F(p\odot  x\odot  q)\),
    and similarly for \(F(p'\odot  p\odot  x)\) and \(F(x\odot  q\odot  q')\). These come from the composition law on the pseudo-naturality squares of \(\phi \) for the composites \(\mathsf {d}_{1}\,\mathsf {d}_{1} = \mathsf {d}_{2}\,\mathsf {d}_{1} : [1] \to  [3]\) applied to the three non-degenerate sequences of length four.
    \item{}unitality squares as in the definition of a lax double functor relating the laxator for \(\mathrm {id}_c\odot  x\) to the unitor \(F_{0,c}\), and
    similarly on the other side. These correspond to the composition law on the pseudo-naturality squares of \(\phi \) for the composites \(\left ( 0,1 \right ) \to  \left ( \underline {\mathbf {0}}, 0, \underline {\mathbf {1}} \right ) \to  \left ( 0,1 \right )\) and \(\left ( 0,1 \right ) \to  \left ( \underline {\mathbf {0}}, 1, \underline {\mathbf {1}} \right ) \to  \left ( 0,1 \right )\) as well as the unit law for \(\phi \) at \(\left ( 0,1 \right )\).\end{itemize}
  If we were considering instead \emph{lax}-morphisms of pseudo-bimodules then we would drop the invertibility condition on the laxators and have \(F_0\) and \(F_1\) instead by \emph{lax}-double functors.
\par{}
  For \(F,G: M\to  M'\) pseudo-bimodule morphisms, a \textbf{pseudo-bimodule transformation} \(\alpha :F\Rightarrow  G\) consists of:
  \begin{itemize}\item{}natural transformations \(\alpha _0:F_0\Rightarrow  G_0\) and \(\alpha _1:F_1\Rightarrow  G_1\) on the domain and codomain double functors
    \item{}for every heteromorphism \(x:c\mathrel {\mkern 3mu\vcenter {\hbox {$\shortmid $}}\mkern -10mu{\to }} d\) in \(X\), a cell
    
\begin{center}\begin {tikzcd}
        {F_0c} & {F_1d} \\
        {G_0c} & {G_1d}
        \arrow [""{name=0, anchor=center, inner sep=0}, "Fx", "\shortmid "{marking}, from=1-1, to=1-2]
        \arrow ["{\alpha _0c}"', from=1-1, to=2-1]
        \arrow ["{\alpha _1d}", from=1-2, to=2-2]
        \arrow [""{name=1, anchor=center, inner sep=0}, "Gx"', "\shortmid "{marking}, from=2-1, to=2-2]
        \arrow ["{\alpha  x}"{description}, draw=none, from=0, to=1]
      \end {tikzcd}\end{center}\end{itemize}
  satisfying
  \begin{itemize}\item{}naturality with respect to heterocells.
  \item{}two external functoriality conditions relating the laxators to the components.\end{itemize}
which come from the naturality condition for a modification between the pseudo-\(\mathscr {F}\)-transformations corresponding to \(F\) and \(G\).
\end{explication}
\section{Restriction for loose bimodules}\label{jrb-003H}\par{}
A notion of “restriction” of a bimodule along maps into its source and target is essential to the use of bimodules for expressing formal category theory. For example, given a 1-cell \(F : \mathsf {C} \to  \mathsf {D}\) in some 1-category \(\mathcal {E}\) with pullbacks, the internal notion of \(F\) being fully-faithful involves an isomorphism between the hom bimodule \(C(-,-)\) and the restriction of \(D(-,-)\) along its source and target via the map \(F\), often denoted \(D(F,F)\). Similarly, adjunctions are expressed as isomorphisms between bimodules of the form \(\mathsf {D} \left ( F,1 \right ) \cong  \mathsf {C} \left ( 1,G \right )\), each being the restriction along the source or target of a hom bimodule.
\par{}
In this section we explore the notion of restriction appropriate to \emph{loose} bimodules expressed from the biased perspective.
\subsection{Restriction of loose bimodules}\label{ssl-003B}\par{}
   Consider a  \(\mathbb {M} \colon  \mathbb {D}_0 \mathrel {\mkern 3mu\vcenter {\hbox {$\shortmid $}}\mkern -10mu{\to }} \mathbb {D}_1\) and pseudo double functors \(F_0 : \mathbb {E}_0 \to  \mathbb {D}_0 \) and \(F_0 : \mathbb {E}_1 \to  \mathbb {D}_1\).

   We aim to construct a "restricted" loose bimodule \(\mathbb {M}(F_0,F_1): \mathbb {E}_0 \mathrel {\mkern 3mu\vcenter {\hbox {$\shortmid $}}\mkern -10mu{\to }} \mathbb {E}_1\). We begin by defining the category that will be the \hyperref[ssl-0036]{carrier} for \(\mathbb {M}(F_0,F_1)\).

\begin{definition}[{Carrier of a restricted loose bimodule}]\label{ssl-003C}
\par{}
  Let \(\mathbb {M} \colon  \mathbb {D}_0 \mathrel {\mkern 3mu\vcenter {\hbox {$\shortmid $}}\mkern -10mu{\to }} \mathbb {D}_1\) be a  and let \(F_0 \colon  \mathbb {E}_0 \to  \mathbb {D}_0 \) and \(F_0 \colon  \mathbb {E}_1 \to  \mathbb {D}_1\) be pseudo double functors.
  
  We define a category, named \(\mathsf {Car}(\mathbb {M}(F_0, F_1))\) by anticipation, as follows:

  \begin{itemize}\item{}
      Objects are triples \((e_0 : {\mathsf {ob}}(\mathbb {E}_0), e_1: {\mathsf {ob}}(\mathbb {E}_1), m \colon  F_0 e_0 \mathrel {\mkern 3mu\vcenter {\hbox {$\shortmid $}}\mkern -10mu{\to }} F_1 e_1).\)
    \item{}
      A morphism \((e_0, e_1, m) \to  (e_0', e_1', m')\) is a triple \((f_0:e_0\to  e_0', f_1:e_1\to  e_1', \alpha )\) where \(\alpha \) is a square in \(\mathbb {M}\) of the following form:

  \begin{center}
        \begin {tikzcd}
          {F_0e_0} & {F_1e_1} \\
          {F_0e_0'} & {F_1e_1} \\
          {}
          \arrow [""{name=0, anchor=center, inner sep=0}, "m", "\shortmid "{marking}, from=1-1, to=1-2]
          \arrow ["{F_0f_0}"', from=1-1, to=2-1]
          \arrow ["{F_1f_1}", from=1-2, to=2-2]
          \arrow [""{name=1, anchor=center, inner sep=0}, "{m'}"', "\shortmid "{marking}, from=2-1, to=2-2]
          \arrow ["\alpha ", shorten <=4pt, shorten >=4pt, Rightarrow, from=0, to=1]
        \end {tikzcd}
      \end{center}\end{itemize}\par{}
  Equivalently, \(\mathsf {Car}(\mathbb {M}(F_0, F_1))\) is the limit of the diagram of  categories

  \begin{center}
    \begin {tikzcd}[cramped]
      {\mathsf {Tight}({\mathbb {E}_0})} && {\mathsf {Car}(\mathbb {M})} && {\mathsf {Tight}({\mathbb {E}_1})} \\
      & {\mathsf {Tight}({\mathbb {D}_0})} && {\mathsf {Tight}({\mathbb {D}_1})}
      \arrow ["{\mathsf {Tight}({F_0})}"', from=1-1, to=2-2]
      \arrow ["{\mathsf {dom}}", from=1-3, to=2-2]
      \arrow ["{\mathsf {codom}}"', from=1-3, to=2-4]
      \arrow ["{\mathsf {Tight}({F_1})}", from=1-5, to=2-4]
    \end {tikzcd}
  \end{center}\par{}
  The canonical limit projections give functors \({\mathsf {dom}} \colon  \mathsf {Car}(\mathbb {M}(F_0, F_1)) \to  \mathsf {Tight}({\mathbb {E}_0})\) and \({\mathsf {codom}} \colon  \mathsf {Car}(\mathbb {M}(F_0, F_1)) \to  \mathsf {Tight}({\mathbb {E}_1})\), as well as \(\pi :\mathsf {Car}(\mathbb {M}(F_0, F_1)) \to  \mathsf {Car}{\mathbb {M}}\).
\end{definition}
\par{}Now we describe the module structure on this carrier.
\begin{definition}[{Restriction of loose bimodules}]\label{ssl-003E}
\par{}
  
  Let \(\mathbb {M} \colon  \mathbb {D}_0 \mathrel {\mkern 3mu\vcenter {\hbox {$\shortmid $}}\mkern -10mu{\to }} \mathbb {D}_1\) be a  and let \(F_0 \colon  \mathbb {E}_0 \to  \mathbb {D}_0 \) and \(F_0 \colon   \mathbb {E}_1 \to  \mathbb {D}_1\) be pseudo double functors. The \textbf{restriction} \(\mathbb {M}(F_0, F_1) \colon  \mathbb {E}_0 \mathrel {\mkern 3mu\vcenter {\hbox {$\shortmid $}}\mkern -10mu{\to }} \mathbb {E}_1\) is defined as follows.

  \begin{itemize}\item{}The tight category \(\mathsf {Tight}({\mathbb {M}(F_0, F_1)})\) is \(\mathsf {Tight}({\mathbb {E}_0})+\mathsf {Tight}({\mathbb {E}_1})\).
    \item{}The loose category \(\mathsf {Loose}({\mathbb {M}(F_0,F_1)})\)
     is \(\mathsf {Loose}({\mathbb {E}_0})+\mathsf {Loose}({\mathbb {E}_1})+\mathsf {Car}(\mathbb {M}(F_0, F_1))\).
     \item{}External identities as well as the appropriate components of external domain, codomain, and composition are inherited from the \(\mathbb {E}_i\).
     \item{}The components of the external domain and codomain at \(\mathsf {Car}(\mathbb {M}(F_0, F_1))\) are given by the projections \({\mathsf {dom}} \colon  \mathsf {Car}(\mathbb {M}(F_0, F_1)) \to  \mathsf {Tight}({\mathbb {E}_0})\) and \({\mathsf {codom}} \colon  \mathsf {Car}(\mathbb {M}(F_0, F_1)) \to  \mathsf {Tight}({\mathbb {E}_1})\) highlighted above.
     \item{}The external composite of a string of proarrows like \(e_0\mathrel {\overset {n_0}{\mathrel {\mkern 3mu\vcenter {\hbox {$\shortmid $}}\mkern -10mu{\to }}}} e_0' \mathrel {\overset {(e_0',e_1,m)}{\mathrel {\mkern 3mu\vcenter {\hbox {$\shortmid $}}\mkern -10mu{\to }}}} e_1\mathrel {\overset {n_1}{\mathrel {\mkern 3mu\vcenter {\hbox {$\shortmid $}}\mkern -10mu{\to }}}} e_1'\) is simply \((e_0,e_1',F_0(n_0)m F_1(n_1))\), and similarly for external composites of cells that cross through \(\mathsf {Car}(\mathbb {M}(F_0, F_1))\).\end{itemize}\par{}
   It is immediate that the \hyperref[ssl-0036]{carrier} of the loose bimodule \(\mathbb {M}(F_0, F_1)\) is the carrier defined in \hyperref[ssl-003C]{Carrier of a restricted loose bimodule}. There is a canonical pseudo double functor \(\Pi :\mathbb {M}(F_0, F_1) \to  \mathbb {M}\) (which will be strict if the \(F_i\) are) given by \(F_i\) on the \(\mathbb {E}_i\) and by the canonical functor \(\pi :\mathsf {Car}(\mathbb {M}(F_0, F_1)) \to  \mathsf {Car}{\mathbb {M}}\) on the carriers. Composition with the structure map \(\mathbb {M}\to  \mathbb {L}\mathsf {oose}\) provides the structure of a loose bimodule on \(\mathbb {M}(F_0, F_1)\) and in particular a loose bimodule morphism \(\mathbb {M}(F_0, F_1) \to  \mathbb {M}\). Finally, we shall have
   considerable use for the observation that \(\Pi \) is faithful on cells of a fixed boundary.
\end{definition}
\subsubsection{Universal property of restriction of loose bimodules}\label{kdc-000P}\par{}We briefly give a universal property to the restriction just constructed. Our universal property is quite weak and does
not nearly subsume everything we need from restriction; thus the result below is by way of illustration only.
\begin{definition}[{The graph of categories of double categories and loose bimodules}]\label{kdc-000K}
\par{}
  We define the graph of categories of double categories and loose bimodules, denoted \(\mathbb {D}\mathsf {bl}_{\mathsf {ps},\mathsf {loose}}\), to be the virtual double category whose: 
  \begin{enumerate}\item{}objects are double categories
   \item{}tight 1-cells are pseudo double functors
   \item{}loose 1-cells are loose bimodules
   \item{}2-cells with boundary as displayed below 
   are simply functors \(\mathbb {M_0}\to \mathbb {M_1}\) over \(\mathbb {L}\mathsf {oose}\) 
   pulling back on the left to \(F\) and on the right to \(G\):
   
\begin{center}
   \begin {tikzcd}
      {\mathbb {D}_{0,0}} & {\mathbb {D}_{1,0}} \\
      {\mathbb {D}_{0,1}} & {\mathbb {D}_{1,1}}
      \arrow ["{M_0}", "\shortmid "{marking}, from=1-1, to=1-2]
      \arrow ["F"{description}, from=1-1, to=2-1]
      \arrow ["G"{description}, from=1-2, to=2-2]
      \arrow ["{M_1}"', "\shortmid "{marking}, from=2-1, to=2-2]
    \end {tikzcd}
    \end{center}\end{enumerate}\par{}Note that loose bimodules do not compose (), so that we cannot upgrade this graph of categories

to a pseudo double category. 
It should be possible to make a virtual double category, but we do not pursue that here.\end{definition}

\begin{proposition}[{The universal property of restriction}]\label{kdc-000G}
\par{}
  Consider a loose bimodule \(M:\mathbb {M}\to  \mathbb {L}\mathsf {oose}\) and pseudo double functors \(F_0:\mathbb {D}_0\to  M_0\) and \(F_1:\mathbb {D}_1\to  M_1\). Then there is a natural bijection 
  between cells with the two boundaries displayed below   in \(\mathbb {D}\mathsf {bl}_{\mathsf {ps},\mathsf {loose}}\).:
  
\begin{center}
  \begin {tikzcd}
      {N_0} & {N_1} & {N_0} & {N_1} \\
      {\mathbb {D}_0} & {\mathbb {D}_1} & {\mathbb {D}_0} & {\mathbb {D}_1} \\
      {M_0} & {M_1}
      \arrow ["N", "\shortmid "{marking}, from=1-1, to=1-2]
      \arrow ["{G_0}"', from=1-1, to=2-1]
      \arrow ["{G_1}", from=1-2, to=2-2]
      \arrow ["N", "\shortmid "{marking}, from=1-3, to=1-4]
      \arrow ["{G_0}"', from=1-3, to=2-3]
      \arrow ["{G_1}", from=1-4, to=2-4]
      \arrow ["{F_0}"', from=2-1, to=3-1]
      \arrow ["{F_1}", from=2-2, to=3-2]
      \arrow ["{M(F_0,F_1)}"', "\shortmid "{marking}, from=2-3, to=2-4]
      \arrow ["M"', "\shortmid "{marking}, from=3-1, to=3-2]
    \end {tikzcd}
  \end{center}
\begin{proof}[{proof of \cref{kdc-000G}}]\label{kdc-000G-proof}
\par{}
  Choosing a cell \(\alpha \) as on the right-hand side involves mapping a loose arrow \(n:e_0\mathrel {\mkern 3mu\vcenter {\hbox {$\shortmid $}}\mkern -10mu{\to }} e_1\) of
  the collage \(\mathbb {N}\) to a loose arrow \((G_0(e_0),G_1(e_1),\alpha _n:F_0G_0(e_0)\mathrel {\mkern 3mu\vcenter {\hbox {$\shortmid $}}\mkern -10mu{\to }} F_1G_1(e_1))\)
  of the collage of the restriction \(M(F_0,F_1)\), and similarly for the cells. Such data is visibly equivalent
  to the data of a cell as on the left-hand side, and the axioms to be checked are in either case identical,
  being entirely stated in \(\mathbb {M}\) itself. 
\end{proof}
\end{proposition}
\section{Loose adjunctions and loose universal properties}\label{kdc-000F}\par{}
As an application of the notion of restriction for loose bimodules introduced in \cref{jrb-003H}, we now present some corresponding loose universal constructions.

\begin{definition}[{loose adjunction}]\label{kdc-000M}
\par{}
  Given pseudo double functors \(F:\mathbb {D}\leftrightarrows  \mathbb {E}:G\), a \textbf{loose adjunction} between them
  consists of a pseudo-natural equivalence of loose bimodules \(\iota :\mathbb {E}(F,\mathrm {id}_{\mathbb {E}})\cong  \mathbb {D}(\mathrm {id}_{\mathbb {D}},G):\mathbb {D} \mathrel {\mkern 3mu\vcenter {\hbox {$\shortmid $}}\mkern -10mu{\Rightarrow }} \mathbb {E}\) 
  fixing \(\mathbb {D}\) and \(\mathbb {E}\). 
\end{definition}

\begin{remark}[{Lax loose adjunctions}]\label{jrb-003I}
\par{}
  There is a laxer notion of adjunction in which there is merely an \textbf{adjunction} between these bimodules. Such a notion is 
  needed to encompass the lax adjunctions used in \cite{patterson-2024-transposing} to construct a cartesian bicategory from 
  a double category with (iso-strong) finite products, but we leave further investigation for future work.
\end{remark}
\par{}
  Since a loose adjunction acts as the identity on the bounding double categories \(\mathbb {C}\) and \(\mathbb {D}\), its defining equivalence is determined by its action on carriers. The carrier categories of these restrictions
  are the categories of loose arrows \(F(d)\mathrel {\mkern 3mu\vcenter {\hbox {$\shortmid $}}\mkern -10mu{\to }} e\) and their cells, respectively, of 
  loose arrows \(d\mathrel {\mkern 3mu\vcenter {\hbox {$\shortmid $}}\mkern -10mu{\to }} G(e)\). Therefore a loose adjunction consists of an equivalence between 
  loose arrows and cells of the forms above, natural with respect to the \(\mathbb {D}\) and \(\mathbb {E}\) actions. 

\begin{example}[{Loose terminal object}]\label{kdc-000S}
\par{}
  Suppose the canonical double functor \(!:\mathbb {D}\to  1\) has a (pseudo) loose right adjoint \(x: 1\to  \mathbb {D}\). 
  This means that the loose hom \(\mathbb {D}(-,x)\) reduces to a point, as do cells into \(\mathrm {id}_x\). 
  More globally, the loose bimodule \(1(!,-):\mathbb {D}\mathrel {\mkern 3mu\vcenter {\hbox {$\shortmid $}}\mkern -10mu{\Rightarrow }} 1\) has a unique proarrow and cell of every 
  possible boundary over the walking loose arrow. This implies that an equivalence on the carriers
  \(1(!,-)\cong  \mathbb {D}(-,x)\) fixing the objects will automatically respect the actions, so that 
  the two conditions just given are both necessary and sufficient for a loose terminal object.
\par{}
  For instance, if \(\mathsf {C}\) has pullbacks and a strict initial object \(0\), then \(0\) will be
  loosely terminal in \(\mathbb {S}\mathsf {pan}(\mathsf {C})\), since \(0\) tips the unique span from any \(x\) 
  to \(0\), and \(\mathrm {id}_0\) gives the unique morphism of such spans with left side any \(f:x\to  x'\).
\end{example}

\begin{example}[{coproducts in lextensive categories are loose products in \(\mathbb {S}\mathsf {pan}\)}]\label{kdc-000Q}
\par{}
  Let \(\mathsf {C}\) be a lextensive category, that is, a category with finite limits and van Kampen finite coproducts. 
  Then \(\mathbb {D}:=\mathbb {S}\mathsf {pan}(\mathsf {C})\) has loose biproducts, in the sense that the diagonal double functor
  \(\Delta :\mathbb {D}\to  \mathbb {D}^2\) has a loose left and right adjoint given by the coproduct 
  \(+:\mathbb {D}^2\to  \mathbb {D}\). Indeed, a span \(a+b \mathrel {\mkern 3mu\vcenter {\hbox {$\shortmid $}}\mkern -10mu{\to }} c\) in \(\mathsf {C}\) is an object 
  of \(\mathsf {C}\downarrow  (a+b)\times  c\), which is equivalent to \(\mathsf {C}\downarrow  ((a \times  c) + (b \times  c))\) because lextensive 
  categories are distributive, and then to \(\mathsf {C}\downarrow  (a \times  c) \times  \mathsf {C}\downarrow  (b \times  c)\), i.e. to \(\mathbb {D}^2((a,b),(c,c))\), 
  by the definition of van Kampen coproducts.
\par{}
  Similarly, a cell as on left below is equivalent to a pair of cells as on the right:

  \begin{center}
\begin {tikzcd}[ampersand replacement=\amper ]
	{a+b} \amper  c \\
	{a'+b'} \amper  {c'}
	\arrow ["s"{inner sep=.8ex}, "\shortmid "{marking}, from=1-1, to=1-2]
	\arrow ["{f+g}"', from=1-1, to=2-1]
	\arrow ["h", from=1-2, to=2-2]
	\arrow ["{s'}"'{inner sep=.8ex}, "\shortmid "{marking}, from=2-1, to=2-2]
  \end {tikzcd}
  \qquad  $\leadsto $ \qquad 
\begin {tikzcd}[ampersand replacement=\amper ]
	a \amper  c \amper  b \amper  c \\
	{a'} \amper  {c'} \amper  {b'} \amper  {c'}
	\arrow ["{s_a}"{inner sep=.8ex}, "\shortmid "{marking}, from=1-1, to=1-2]
	\arrow ["f"', from=1-1, to=2-1]
	\arrow ["h", from=1-2, to=2-2]
	\arrow ["{s_b}"{inner sep=.8ex}, "\shortmid "{marking}, from=1-3, to=1-4]
	\arrow ["g"', from=1-3, to=2-3]
	\arrow ["h", from=1-4, to=2-4]
	\arrow ["{s'_{a'}}"'{inner sep=.8ex}, "\shortmid "{marking}, from=2-1, to=2-2]
	\arrow ["{s'_{b'}}"'{inner sep=.8ex}, "\shortmid "{marking}, from=2-3, to=2-4]
\end {tikzcd}
\end{center}

  Finally, if we are to compose a span \(a+b\mathrel {\overset {s}{\mathrel {\mkern 3mu\vcenter {\hbox {$\shortmid $}}\mkern -10mu{\to }}}} c\) with spans \(a'\mathrel {\overset {t}{\mathrel {\mkern 3mu\vcenter {\hbox {$\shortmid $}}\mkern -10mu{\to }}}} a,b'\mathrel {\overset {u}{\mathrel {\mkern 3mu\vcenter {\hbox {$\shortmid $}}\mkern -10mu{\to }}}} b,c\mathrel {\overset {v}{\mathrel {\mkern 3mu\vcenter {\hbox {$\shortmid $}}\mkern -10mu{\to }}}} c'\) and 
  reinterpret the result \(\mathsf {C}\downarrow  (a' \times  c') \times  \mathsf {C}\downarrow  (b' \times  c')\), we can just as well first pass into 
  \(\mathsf {C}\downarrow  (a \times  c) \times  \mathsf {C}\downarrow  (b \times  c)\) and then compose, essentially because both the mapping \(s\mapsto  s_a\)
  and the composition in \(\mathbb {D}\) are given by pullbacks. This concludes the proof that \(\mathbb {D}\) has 
  loose coproducts given by \(+\). The proof that \(+\) also gives loose products can be given along the same lines, 
  or by considering the external self-duality of span categories. 

\end{example}

\begin{example}[{Loosely monoidal closed double categories}]\label{kdc-000N}
\par{}
 Let \(\mathbb {D}\) be a symmetric monoidal double category. We say \(\mathbb {D}\) is \textbf{loosely monoidal closed} 
 if there is a double functor \([-,-]:\mathbb {D}^{\mathsf {co}}\times \mathbb {D}\to  \mathbb {D}\) such that the loose modules
 \(\mathbb {D}(-\otimes  -,-)\) and \(\mathbb {D}(-,[-,-])\) are equivalent when viewed with signature 
 \(\mathbb {D}^2\mathrel {\mkern 3mu\vcenter {\hbox {$\shortmid $}}\mkern -10mu{\Rightarrow }} \mathbb {D}\) (see .) In particular, for each object \(x\in  \mathbb {D}\),
 this implies the double functors \(-\otimes  x\) and \([x,-]\) are loosely adjoint. Our main examples 
 of these are in fact \textbf{compact} double categories as described by Patterson \cite{patterson-2024-toward} \cite{patterson-2024-toward-2}, 
 in that the loose hom will come from tensoring with a dual object.
\par{}
  For a very simple example, 
  consider the double category \(\mathbb {R}\mathsf {el}(\mathsf {E})\) of relations in a regular category, symmetric monoidal 
  under the cartesian product. There is an
  obvious bijection between relations \(a\times  b\mathrel {\mkern 3mu\vcenter {\hbox {$\shortmid $}}\mkern -10mu{\to }} c\) and relations \(a\mathrel {\mkern 3mu\vcenter {\hbox {$\shortmid $}}\mkern -10mu{\to }} b\times  c\), both being
  ternary relations on \(a\times  b\times  c\), and this bijection respects implication of relations as well 
  as composition with a pair of relations on the left or a single relation on the right.
  Thus \(\mathbb {R}\mathsf {el}(\mathsf {E})\) is loosely monoidal closed, with the closure coinciding with the cartesian product.
  Essentially the same argument works for categories of spans in any category with pullbacks. 
  In particular, considering the previous example, spans in a lextensive category \(\mathsf {C}\) are compact closed
  with respect to the cartesian product and also equipped with biproducts via the disjoint union.
\par{}
  If \(\mathsf {V}\) is cocomplete and closed symmetric monoidal, then the double category \(\mathsf {V}-\mathbb {P}\mathsf {rof}\) of
  enriched profunctors is loosely monoidal closed, with the closure given by \(J,K\mapsto  J^{\mathsf {op}}\otimes  K\). Indeed,
  insofar as a profunctor \(J\mathrel {\mkern 3mu\vcenter {\hbox {$\shortmid $}}\mkern -10mu{\to }} K\) is a functor \(J^{\mathsf {op}}\otimes  K\to  \mathsf {V}\), the correspondence on objects 
  is immediate; as for cells, given profunctors \(m:J\otimes  K\mathrel {\mkern 3mu\vcenter {\hbox {$\shortmid $}}\mkern -10mu{\to }} L\) and \(m':J'\otimes  K'\mathrel {\mkern 3mu\vcenter {\hbox {$\shortmid $}}\mkern -10mu{\to }} L'\), 
  together with functors \(f:J\to  J'\), \(g:K\to  K'\), and \(h:L\to  L'\), a cell in \(\mathsf {V}-\mathbb {P}\mathsf {rof}(-\otimes  -,-)\)
  with the given boundary is a natural transformation \(m\to  m'(f(-)\otimes  g(-),h(-))\),
  which is visibly equivalent to a natural transformation \(\bar  m\to  \bar  m'(f(-),g^{\mathsf {op}}(-)\otimes  h(-))\),
  where \(\bar {(-)} \) is the correspondence between profunctors \(J\otimes  K\mathrel {\mkern 3mu\vcenter {\hbox {$\shortmid $}}\mkern -10mu{\to }} L\) and \(J\mathrel {\mkern 3mu\vcenter {\hbox {$\shortmid $}}\mkern -10mu{\to }} K^{\mathsf {op}}\otimes  L\)
  mentioned above. 
\par{}
  Monoidal closed 2-categories produce both tightly and loosely monoidal closed double categories, depending 
  on the direction in which the 2-category is interpreted as a double category. Thus for instance we can interpret 
  \(\mathcal {C}\mathsf {at}\) as a double category in the loose direction and get a non-compact closed double category 
  via the cartesian product of categories. We have in fact no examples to hand of non-compact closed double categories 
  that are not degenerate in a similar sense, and indeed even the bicategorical analogue is surprisingly little studied. 
\end{example}

\begin{example}[{Van Kampen colimits}]\label{kdc-000T}
\par{}
  For our final example, we will need a notion of loose colimit of a single diagram in a double category.
\begin{definition}[{Loose colimit}]\label{kdc-000V}
\par{}
  Given double categories 
  \(\mathbb {I},\mathbb {D}\), let \(\mathbb {L}\mathsf {ax}(\mathbb {I},\mathbb {D})\) denote the double category of double functors, strict tight transformations, 
  pseudo loose transformations, and modifications from \(\mathbb {I}\) to \(\mathbb {D}\).  
  Let \(\Delta : \mathbb {D}\to  \mathbb {L}\mathsf {ax}(\mathbb {I},\mathbb {D})\) be the diagonal double functor, which is strict. Fixing a single diagram 
  \(D:\mathbb {I}\to  \mathbb {D}\), viewed as a pseudo double functor \(1\to  \mathbb {D}^{\mathbb {I}}\), 
  we can derive by restriction the loose bimodule 
  \(\mathbb {L}\mathsf {ax}(\mathbb {I},\mathbb {D})(D,\Delta  -):1\mathrel {\mkern 3mu\vcenter {\hbox {$\shortmid $}}\mkern -10mu{\Rightarrow }} \mathbb {D}\).
  A \textbf{loose colimit} of \(D\) is then a representation of this 
  loose bimodule: that is, an object \(\int _\ell ^J D\) equipped with an equivalence of bimodules
  \(\mathbb {L}\mathsf {ax}(\mathbb {I},\mathbb {D})(D,\Delta  -)\cong  \mathbb {D}(\int _\ell ^J D,-)\) fixing \(\mathbb {D}\).\end{definition}
\par{}Observe that we can induce 
  such a map of bimodules using a \textbf{loose cocone}, a protransformation \(\eta : D\mathrel {\mkern 3mu\vcenter {\hbox {$\shortmid $}}\mkern -10mu{\to }} \Delta \int _\ell ^J D\): 
  we can then map loose homs out of \(\int _\ell ^J D\) and cells out of \(\mathrm {id}_{\int _\ell ^J D}\) via composition and whiskering
  with \(\eta \), respectively. (Of course, there must be an equivalent definition of loose colimit directly in terms of initial
  cocones.)\par{}Furthermore, once we have a map \(t: \mathbb {D}(\int _\ell ^J D,-)\to  \mathbb {L}\mathsf {ax}(\mathbb {I},\mathbb {D})(D,\Delta  -)\) 
  of loose bimodules under \(\mathbb {D}\), such as the one induced by a such cocone, that \(t\) be an equivalence reduces 
  to the condition that it be an equivalence of carrier categories, thus precisely the following two conditions:
  \begin{enumerate}\item{}(One-dimensional universal property) Each protransformation \(D\mathrel {\mkern 3mu\vcenter {\hbox {$\shortmid $}}\mkern -10mu{\to }} \Delta  d\) is isomorphic to one in the image of \(t\) via 
    a globular modification.
    
    \item{}(Two-dimensional universal property) \(t\) induces a bijection between cells of the following forms:

  \begin{center}
      \begin {tikzcd}
          {\int ^J_\ell  D} & d && D & {\Delta  d} \\
          {\int ^J_\ell  D} & {d'} && D & {\Delta  d'}
          \arrow [""{name=0, anchor=center, inner sep=0}, "m", "\shortmid "{marking}, from=1-1, to=1-2]
          \arrow [equals, from=1-1, to=2-1]
          \arrow [""{name=1, anchor=center, inner sep=0}, "f", from=1-2, to=2-2]
          \arrow [""{name=2, anchor=center, inner sep=0}, "tm", "\shortmid "{marking}, from=1-4, to=1-5]
          \arrow [""{name=3, anchor=center, inner sep=0}, equals, from=1-4, to=2-4]
          \arrow ["{\Delta  f}", from=1-5, to=2-5]
          \arrow [""{name=4, anchor=center, inner sep=0}, "{m'}"', from=2-1, to=2-2]
          \arrow [""{name=5, anchor=center, inner sep=0}, "{tm'}"', "\shortmid "{marking}, from=2-4, to=2-5]
          \arrow ["\alpha "{description}, draw=none, from=0, to=4]
          \arrow [shorten <=19pt, shorten >=19pt, tail reversed, from=1, to=3]
          \arrow ["{t\alpha }"{description}, draw=none, from=2, to=5]
        \end {tikzcd}
      \end{center}

  \end{enumerate}\par{}Recall from \cite{patterson-2024-transposing} (Corollary 3.9) 
  that the companion of a tight arrow \(\alpha \) in \(\mathbb {L}\mathsf {ax}(\mathbb {I},\mathbb {D})\) exists if and only if each object component of \(\alpha \)
  has a companion and each cell component of \(\alpha \) is a commuter. 
  In particular, by Lemma 3.11 of loc. cit., if every pro-arrow 
  of \(\mathbb {I}\) is a companion, then the second condition becomes automatic; this is the case for instance if \(\mathbb {I}\) 
  is a 1-category interpreted as a tight double category. For purposes of the following definition, we 
  shall be most interested in the case that 
  \(\mathbb {I}\) is a 1-category (in the sense just mentioned) and \(\mathbb {D}\) is an equipment, 
  so that indeed \(\mathbb {D}^{\mathbb {I}}\) is also an equipment.

\begin{definition}[{Van Kampen colimit in a double category}]\label{kdc-000W}
\par{}
  Suppose \(D:\mathbb {I}\to  \mathbb {D}\) is a lax double functor and \(\eta :D\to  \Delta \int _t^J D\) 
  is a \textbf{tightly} colimiting cocone which has a pseudo protransformation companion
  \(\eta ^>:D^>\to  \Delta \int _t^J D\). Then we can ask whether \(\eta ^>\) also makes \(\int _t^J D\) into a 
  loose colimit of  \(D^>:\mathbb {I}_\ell \to  \mathbb {D}\), where \(\mathbb {I}_\ell \) is the bicategory 
  indexing \(D^>\) by covariantly transposing every tight arrow and cell of \(\mathbb {I}\). 
  If \(\eta ^>\) is indeed a loose colimit cocone, then we say that \(\eta \) is a 
  \textbf{van Kampen colimit} of \(D\).
\end{definition}
\par{}
  In \cref{kdc-000S}, we observed that a strict initial object in a category \(\mathsf {C}\) with pullbacks is 
  a van Kampen colimit
  in the double category \(\mathbb {S}\mathsf {pan}(\mathsf {C})\) of spans in \(\mathsf {C}\), and similarly in \cref{kdc-000Q} for 
  binary coproducts in a lextensive category. In fact, we can say more: 

\begin{proposition}[{Van Kampen colimits are van Kampen colimits}]\label{kdc-000X}
\par{}A colimit in a category \(\mathsf {C}\) 
  with pullbacks is van Kampen, in the standard sense, if and only if it is van Kampen in our sense in 
  \(\mathbb {S}\mathsf {pan}(\mathsf {C})\).
\begin{proof}[{proof of \cref{kdc-000X}}]\label{kdc-000X-proof}
\par{}
  It is known from \cite{sobocinski-2011-being}
 that a colimit in a category \(\mathsf {C}\) with pullbacks is van Kampen if and only if 
 it is preserved by the canonical covariant embedding of \(\mathsf {C}\) into its bicategory of spans. 
 Then the result follows from the lemma below and the fact that \(\mathbb {S}\mathsf {pan}(\mathsf {C})\) is
 an equipment.
\end{proof}
\end{proposition}

\begin{lemma}[{Bicategorical colimits in an equipment are loose double colimits}]\label{kdc-000Y}
\par{}
  Suppose \(\mathbb {D}\) is an equipment with underlying bicategory \(\mathcal {D}\). If \(D:\mathcal {I}\to  \mathcal {D}\) 
  is a lax double functor with a bicategorical colimit cocone \(\eta :D\to  \Delta \int _t^J D\), 
  then \(\eta \) is also a loose cocone for \(D\) viewed as a lax double functor \(\mathcal {I}\to  \mathbb {D}\). 

\begin{proof}[{proof of \cref{kdc-000Y}}]\label{kdc-000Y-proof}
\par{}
 Being a loose colimit in \(\mathbb {D}\) is, at a glance, a stronger condition than being 
 a colimit in \(\mathcal {D}\). The former requires that 
 cells of the shapes below must be in bijective correspondence, whereas the latter 
 only requires this for globular such cells:

  \begin{center}\begin {tikzcd}
	{\int ^J_\ell  D} & x & D & {\Delta  x} \\
	{\int ^J_\ell  D} & y & D & {\Delta  y}
	\arrow ["m"', "\shortmid "{marking}, from=1-1, to=1-2]
	\arrow [equals, from=1-1, to=2-1]
	\arrow [""{name=0, anchor=center, inner sep=0}, "f"', from=1-2, to=2-2]
	\arrow ["tm", "\shortmid "{marking}, from=1-3, to=1-4]
	\arrow [""{name=1, anchor=center, inner sep=0}, equals, from=1-3, to=2-3]
	\arrow ["{\Delta  f}", from=1-4, to=2-4]
	\arrow ["n", "\shortmid "{marking}, from=2-1, to=2-2]
	\arrow ["tn"', "\shortmid "{marking}, from=2-3, to=2-4]
	\arrow [shorten <=6pt, shorten >=6pt, tail reversed, from=0, to=0-|1]
\end {tikzcd}\end{center}

But cells as on the left above are in bijective correspondence with globular cells into 
\(n(-,f)\); if \(\eta \) is a bicategorical colimit cocone, then such cells are in bijection 
with globular cells into \(t(n(-,f))\), which is \(tn(-,\Delta  f)\) because \(t\) preserves restrictions;
finally, such globular cells are in bijection with cells of the shape on the right above.
\end{proof}
\end{lemma}
\end{example}
\section{Appendix}\label{jrb-002A}
\begin{proof}[{Proof of \cref{djm-00IE}}]\label{jrb-002B}
\par{}\par{}
The structure of the proof is as follows:
\begin{enumerate}\item{}we define a map on objects \(A : \mathcal {P}\mathsf {sMod}^{\mathsf {p}}_{\mathscr {F}}(\Delta ^{\mathsf {op}}, \mathcal {C}\mathsf {at}) \to  \mathcal {D}\mathsf {bl}\)
  \item{}we define another map on objects \(B : \mathcal {D}\mathsf {bl} \to  \mathcal {P}\mathsf {sMod}^{\mathsf {p}}_{\mathscr {F}}(\Delta ^{\mathsf {op}}, \mathcal {C}\mathsf {at})\) and observe that \(\left ( AB \right ) M \cong  M\) for any pseudo double category \(M\) (and thus that \(A\) is essentially surjective)
  \item{}we extend the map on objects \(A\) into a 2-functor and observe that it's 2-fully-faithful.\end{enumerate}
the first part of which was given in \cref{djm-00IE-proof} and is repeated here for completeness.

\par{}
We define the mapping from pseudo-models of \(\Delta ^{\mathsf {op}}\) to Leinster's notion of pseudo double categories by explicitly exhibiting the structure and conditions specified in Definition 5.2.1 of \cite{leinster-2004-higher}.

\par{}
Let \(M : \Delta ^{\mathsf {op}} \to  \mathcal {C}\mathsf {at}\) be a pseudo-model of \(\Delta ^{\mathsf {op}}\), with \(M_n\) denoting the image under \(M\) of \([n] \in  \Delta \). We associate to \(M\) the pseudo double category defined as follows:
\begin{enumerate}\item{}The \emph{diagram} in \(\mathcal {C}\mathsf {at}\) is given by:

  \begin{center}
\begin {tikzcd}[row sep=tiny]
	& {M_1} \\
	{M_0} && {M_0}
	\arrow ["{M{\mathsf {d}_{1}}}"', from=1-2, to=2-1]
	\arrow ["{M{\mathsf {d}_{0}}}", from=1-2, to=2-3]
\end {tikzcd}
\end{center}

The model condition (corresponding to the Segal condition for the algebraic pattern \(\Delta ^{\mathsf {op}}\)) requires that \(M_n\) be isomorphic to the \(n\)-fold composite, denoted
for short by \(M_1^{(n)}\), of this span. We denote this isomorphism by \(\theta _n : M_1^{(n)} \to  M_n\).
  
  \item{}For each \(n \geq  0\), we define the external composition of \(n\) loose arrows, \(\mathsf {comp}_n : M_1^{(n)} \to  M_1\), to be:
  \begin{equation}M_1^{(n)} \xrightarrow {\theta _n} M_n \xrightarrow {M \mathsf {a}_n} M_1\end{equation}
  where \(\mathsf {a}_n : [1] \to  [n]\) is the unique such (active) map in \(\Delta ^{\mathsf {op}}\). The maps \(\mathsf {comp}_n\) commute with the projections down to \(M_0\) by the fact that each triangle in the following diagram commutes, where \(\pi _j : M_1^{(n)} \to  M_1\) is the canonical projection onto the \(j^{\text {th}}\) component:

  \begin{center}
\begin {tikzcd}
	& {M_1^{(n)}} \\
	{M_0} & {M_n} & {M_0} \\
	& {M_1}
	\arrow ["{M\mathsf {d}_{1}\;\pi _1}"', from=1-2, to=2-1]
	\arrow ["{\theta _n}", from=1-2, to=2-2]
	\arrow ["{M\mathsf {d}_{0}\;\pi _n}", from=1-2, to=2-3]
	\arrow ["{M\mathsf {d}_{n}}"', from=2-2, to=2-1]
	\arrow ["{M\mathsf {d}_{0}}", from=2-2, to=2-3]
	\arrow ["{M\mathsf {a}_n}"{pos=0.3}, from=2-2, to=3-2]
	\arrow ["{M\mathsf {d}_{1}}", from=3-2, to=2-1]
	\arrow ["{M\mathsf {d}_{0}}"', from=3-2, to=2-3]
\end {tikzcd}
\end{center}
\item{}The \emph{double-sequence maps} are invertible 2-cells of the following form (where \(\diamond \) indicates composition of spans):

  \begin{center}
\begin {tikzcd}
	{M^{(k_1 + \dots  + k_n)}_1} & {M_1^{(k_1)} \diamond \dots  \diamond  M_1^{(k_n)}} & [+4em] {M_1^{(n)}} & [+1em] {M_1}
	\arrow ["\cong ", from=1-1, to=1-2]
	\arrow [""{name=0, anchor=center, inner sep=0},"{\mathsf {comp}_{k_1 + \dots  + k_n}}"', popdown=2em, from=1-1, to=1-4]
	\arrow [""{name=1, anchor=center, inner sep=0},"{\mathsf {comp}_{k_1} \diamond  \dots  \diamond  \mathsf {comp}_{k_n}}", from=1-2, to=1-3]
\arrow ["{\mathsf {comp}_n}", from=1-3, to=1-4]
\arrow ["\gamma ", Rightarrow, shorten <=8pt, shorten >=8pt, from=0|-1, to=0]
\end {tikzcd}
\end{center}

To exhibit such 2-cells, we first observe that the following diagram commutes:

  \begin{center}
\begin {tikzcd}
	{M^{(k_1 + \dots  + k_n)}_1} & {M_1^{(k_1)} \diamond \dots  \diamond  M_1^{(k_n)}} & [+4em] {M_1^{(n)}} \\
	{M_{k_1 + \dots  + k_n}} && {M_n}
	\arrow ["\cong ", from=1-1, to=1-2]
	\arrow ["\theta "', from=1-1, to=2-1]
	\arrow ["{\mathsf {comp}_{k_1} \diamond  \dots  \mathsf {comp}_{k_n}}", from=1-2, to=1-3]
	\arrow ["\theta "', from=1-3, to=2-3]
	\arrow [""{name=0, anchor=center, inner sep=0}, "{M(\mathsf {a}_{k_1}+\dots  + \mathsf {a}_{k_n})}"', from=2-1, to=2-3]
	\arrow ["\circlearrowright "{marking, allow upside down}, draw=none, from=1-2, to=0]
\end {tikzcd}
\end{center}

because it commutes under post-composition by the various projections \(M \mathsf {i}_j : M_n \to  M_1\) of the universal cone under \(M_n\) where \(\mathsf {i}_j : [1] \to  [n]\) is the inert map with \(\mathsf {i}_j (0) = j-1\) (recall that more generally \(\mathsf {i}_j : [m] \to  [n]\) denotes the inert map with the same property). We observe this is true as follows:
\begin{equation}\begin {aligned}
M \mathsf {i}_j \; M(\mathsf {a}_{k_1}+\dots  + \mathsf {a}_{k_n}) \; \theta 
&= \; M \mathsf {a}_{k_j} \; M \mathsf {i}_{k_1 + \dots  + k_{j-1}} \; \theta  \\
&= \; M \mathsf {a}_{k_j} \; \theta  \; M^{(\mathsf {i}_{k_1 + \dots  + k_{j-1}})}_1 \\
&= \; \mathsf {comp}_{k_j}\; M^{(\mathsf {i}_{k_1 + \dots  + k_{j-1}})}_1 \\
&= \; \pi _j \; \left ( \mathsf {comp}_{k_1} \diamond  \dots  \diamond  \mathsf {comp}_{k_n} \right ) \\
&= \; M \mathsf {i}_j \; \theta  \; \left ( \mathsf {comp}_{k_1} \diamond  \dots  \diamond  \mathsf {comp}_{k_n} \right ) \\
\end {aligned}\end{equation}
Consequently, the following diagram commutes:

  \begin{center}
\begin {tikzcd}
	{M^{(k_1 + \dots  + k_n)}_1} & {M_1^{(k_1)} \diamond \dots  \diamond  M_1^{(k_n)}} & [+4em] {M_1^{(n)}} & [+1em]{M_1} \\
	{M_{k_1 + \dots  + k_n}} && {M_n}
	\arrow ["\cong ", from=1-1, to=1-2]
	\arrow ["\theta "', from=1-1, to=2-1]
	\arrow ["{\mathsf {comp}_{k_1} \diamond  \dots  \diamond  \mathsf {comp}_{k_n}}", from=1-2, to=1-3]
	\arrow ["{\mathsf {comp}_n}", from=1-3, to=1-4]
	\arrow ["\theta "', from=1-3, to=2-3]
	\arrow [""{name=0, anchor=center, inner sep=0}, "{M(\mathsf {a}_{k_1}+\dots  + \mathsf {a}_{k_n})}"', from=2-1, to=2-3]
	\arrow ["{M\mathsf {a}_n}"', rounded corners, to path={ -| (\tikztotarget )[pos=0.25]\tikztonodes }, from=2-3, to=1-4]
	\arrow ["\circlearrowright "{marking, allow upside down}, draw=none, from=1-2, to=0]
\end {tikzcd}
\end{center}

This means that a 2-cell \(\gamma \) of the form shown above is equivalent to a 2-cell of the form:
\begin{equation}M\mathsf {a}_n \;M(\mathsf {a}_{k_1}+\dots  + \mathsf {a}_{k_n}) \Rightarrow   M \left ( \mathsf {a}_{k_1 + \dots  + k_n} \right )\end{equation} which we obtain as the laxator for \(M\) witnessing the composition:
\begin{equation}(\mathsf {a}_{k_1}+\dots  + \mathsf {a}_{k_n})\; \mathsf {a}_n = \mathsf {a}_{k_1 + \dots  + k_n}\end{equation}
This 2-cell is invertible, as required. To see that these 2-cells are moreover \emph{globular}, i.e. vertical with respect to \(M\mathsf {d}_{0}, M\mathsf {d}_{1} : M_1 \to  M_0\), we appeal to the associative property of \(M\) which allows us to perform the following manipulation of string diagrams for the 2-cell given by post-whiskering the pseudo compositor \(\mu \) with \(M\mathsf {d}_{0}\):
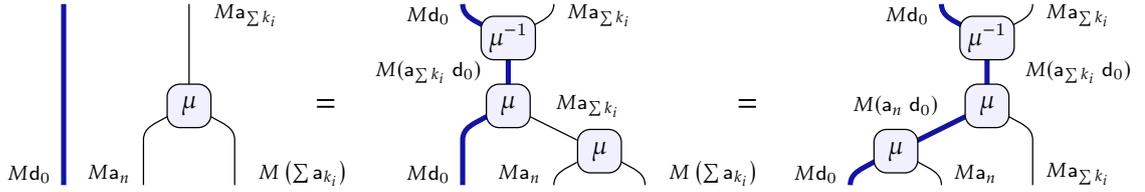
\begin{equation*}
    \tikzfig[1.2]{string-diagram}
\end{equation*}
  
Since any compositor with an inert input (indicated in bold and coloured blue) is an identity, the entire 2-cell must be an identity as required. The same argument applies for post-whiskering by \(M \mathsf {d}_{1}\).

\item{}The \(\iota \) maps are given by an invertible 2-cell of the form:

  \begin{center}
\begin {tikzcd}
	{M_1} & |[alias=middle]| {M_1^{(1)}} & {M_1}
	\arrow ["{\theta _1^{-1}}", from=1-1, to=1-2]
	\arrow [""{name=1, anchor=center, inner sep=0},equals, popdown=2em, from=1-1, to=1-3]
	\arrow ["{\mathsf {comp}_1}", from=1-2, to=1-3]
	\arrow ["{\iota }"{outer sep=3pt}, Rightarrow, shorten <=4pt, shorten >=4pt,from=middle, to=1]
  \end {tikzcd}
\end{center}

which we take to be the identity, noting that \(\mathsf {comp}_1\;\theta _1^{-1} = M1_{[1]}\;\theta _1 \;\theta _1^{-1} = 1_{M_1}\) by the fact that \hyperref[djm-00I2]{\(M\) is strictly unital}.
\end{enumerate}
The associativity and identity coherence axioms for \(\gamma \) and \(\iota \) correspond directly to those for the pseudo-\(\mathscr {F}\)-functor \(M\).

\par{}
We now have our map on objects \({\mathsf {ob}} \left ( \mathcal {P}\mathsf {sMod}^{\mathsf {p}}_{\mathscr {F}} \left (\Delta ^{\mathsf {op}},\mathcal {C}\mathsf {at}\right ) \right ) \to  {\mathsf {ob}} \left ( \mathcal {D}\mathsf {bl} \right )\). Conversely, given a pseudo double category, \(C\), we obtain a model \(M : \Delta ^{\mathsf {op}} \to  \mathcal {C}\mathsf {at}\) as follows:
\begin{enumerate}\item{}\(\mathsf {ps}\){On objects:} \(M_0 := C_0\), \(M_1 := C_1\) and \(M_n := C_1^{(n)} = M_1^{(n)}\) for \(n > 1\)
  \item{}\par{}\(\mathsf {ps}\){On morphisms:} we need only define the action on active and inert morphisms, as these form a factorisation system on \(\Delta \). Given a \emph{active} morphism \(f : [n] \to  [m]\) which can be expressed as a sum of active morphisms into \([1]\), \(\left ( \mathsf {a}_{k_1} + \dots  + \mathsf {a}_{k_n} \right ) : [n] \to  [k_1 + \dots  + k_n]\), the functor \(Mf\) is given by:
\begin{equation}M_1^{(m)} \xrightarrow {\cong } M_1^{(k_1)} \diamond  \dots  \diamond  M_1^{(k_n)} \xrightarrow {\mathsf {comp}_{k_1} \diamond  \dots  \diamond  \mathsf {comp}_{k_n}} M_1^{(n)}\end{equation}
modulo replacing \(M_1^{(1)}\) with \(M_1\) and \(\mathsf {comp}_1 : M_1^{(1)} \to  M_1\) with the identity. This replacement corresponds to substituting \(M\) with a “unital” (i.e. \(\iota \) is an identity) pseudo double category to which it is isomorphic. For an \emph{inert} morphism \(f : [n] \to  [m]\), we define \(Mf\) to simply be \(M_1^{(f)}\), i.e. the unique morphism \(M_1^{(m)} \to  M_1^{(n)}\) satisfying
\begin{equation}M_1^{(m)} \xrightarrow {M_1^{(f)}} M_1^{(n)} \xrightarrow {\pi _j} M_1 \quad  =
\quad  M_1^{(m)} \xrightarrow {\pi _{fj}} M_1
\end{equation} where as before we assume \(M_1 = M_1^{(1)}\).

\par{}
This action on morphisms is strictly functorial on inerts, as we have:
\begin{equation}\pi _{j} \; M_1^{(f)} \; M_1^{(g)} = \pi _{fj} \; M_1^{(g)} = \pi _{\left ( gf \right )j} = \pi _j \; M_1^{(gf)}\end{equation}
For a composite of an active \emph{after} an inert: 
\begin{equation}[n] \xrightarrow {f} [m] \xrightarrow {\mathsf {a}_{k_1} + \dots  + \mathsf {a}_{k_m}} [k_1 + \dots  + k_n]\end{equation}
which we note has an active–inert factorisation of the form:
\begin{equation}[n] \xrightarrow {\mathsf {a}_{k_{f1}} + \dots  + \mathsf {a}_{k_{fn}}} [k_{f1} + \dots  + k_{fn}] \xhookrightarrow {f'} [k_1 + \dots  + k_n]\end{equation}
we must verify that the following diagram commutes:

  \begin{center}
\begin {tikzcd}[column sep=8em]
	{M_1^{(k_1 + \dots  + k_n)}} & {M_1^{(m)}} \\
	{M_1^{(k_{f1} + \dots  +  k_{fn})}} & {M_1^{(n)}}
	\arrow ["{\mathsf {comp}_{k_1} \diamond  \dots  \diamond  \mathsf {comp}_{k_m}}", from=1-1, to=1-2]
	\arrow ["{M_1^{(f')}}"', from=1-1, to=2-1]
	\arrow ["{M^{(f)}_1}", from=1-2, to=2-2]
	\arrow ["{\mathsf {comp}_{k_{f1}} \diamond  \dots  \diamond  \mathsf {comp}_{k_{fn}}}"', from=2-1, to=2-2]
\end {tikzcd}
\end{center}

This follows from the fact it commutes after post-composition by each projection \(\pi _j : M_1^{(n)} \to  M_1\).

\par{}Now consider the composite of two active morphisms.
Given composable active morphisms, \([n] \xrightarrow {f} [m] \xrightarrow {g} [l]\), where \(f = \mathsf {a}_{x_1} + \dots  + \mathsf {a}_{x_n}\), \(g = \mathsf {a}_{y_1} + \dots  + \mathsf {a}_{y_m}\) and \(gf = \mathsf {a}_{z_1} + \dots  + \mathsf {a}_{z_m}\), we must provide a 2-cell of the following form:

  \begin{center}
\begin {tikzcd}[column sep=8em]
	{M_1^{(y_1 + \dots  + y_n)}} & |[alias=middle]| {M_1^{(x_1 + \dots  +  x_m)}} & {M_1^{(m)}}
	\arrow ["{\mathsf {comp}_{y_1} \diamond  \dots  \diamond  \mathsf {comp}_{y_n}}", from=1-1, to=1-2]
	\arrow [""{name=0, anchor=center, inner sep=0}, "{\mathsf {comp}_{z_1} \diamond  \dots  \diamond  \mathsf {comp}_{z_m}}"', popdown=1.5em, from=1-1, to=1-3]
	\arrow ["{\mathsf {comp}_{x_1} \diamond  \dots  \diamond  \mathsf {comp}_{x_m}}", from=1-2, to=1-3]
	\arrow ["\mu "', shorten >=4pt, Rightarrow, from=middle, to=middle|-0]
\end {tikzcd}
\end{center}

for which it suffices to define 2-cells:

  \begin{center}
\begin {tikzcd}[row sep=2em]
	{M_1^{(y_1 + \dots  + y_n)}} &  [+4em] |[alias=middle]| {M_1^{(x_1 + \dots  + x_m)}} & [+4em] {M_1^{(m)}} & {M_1} \\
	&& |[alias=foo]|{M_1^{(m)}}
	\arrow ["{\mathsf {comp}_{y_1} \diamond  \dots  \diamond  \mathsf {comp}_{y_n}}", from=1-1, to=1-2]
	\arrow [""{name=0, anchor=center, inner sep=0}, "{\mathsf {comp}_{z_1} \diamond  \dots  \diamond  \mathsf {comp}_{z_m}}"', rounded corners, to path={ |- (\tikztotarget )[pos=0.75]\tikztonodes }, from=1-1, to=2-3]
	\arrow ["{\mathsf {comp}_{x_1} \diamond  \dots  \diamond  \mathsf {comp}_{x_m}}", from=1-2, to=1-3]
	\arrow ["{\pi _j}", from=1-3, to=1-4]
	\arrow ["{\pi _j}"', rounded corners, to path={ -| (\tikztotarget )[pos=0.25]\tikztonodes },from=2-3, to=1-4]
	\arrow ["\mu _j"', shorten >=3pt, Rightarrow, from=middle, to=middle|-0]
\end {tikzcd}
\end{center}

which are compatible with the projections \(M \mathsf {d}_{0}, M \mathsf {d}_{1} : M_1 \to  M_0\) by the universal property of \(M_1^{(m)}\). Now, the composites along the top and bottom of this diagram simplify to:

  \begin{center}
\begin {tikzcd}[row sep=2em]
	{M_1^{(y_1 + \dots  + y_n)}} & [+2em] |[alias=middle]| {M_1^{(y_s + \dots  + y_t)}} & [+4em] {M_1^{(x_j)}} & [+1em] {M_1} \\
	& |[alias=foo]|{M_1^{(y_s + \dots  + y_t)}}
	\arrow ["{\mathsf {comp}_{y_s} \diamond  \dots  \diamond  \mathsf {comp}_{y_t}}", from=1-2, to=1-3]
	\arrow [""{name=0, anchor=center, inner sep=0}, "{\mathsf {comp}_{y_s + \dots  + y_t} = \mathsf {comp}_{z_j}}"', rounded corners, to path={ -| (\tikztotarget )[pos=0.25]\tikztonodes }, from=2-2, to=1-4]
	\arrow ["{\mathsf {comp}_{x_j}}", from=1-3, to=1-4]
	\arrow ["M \mathsf {i}_{y_1 + \dots  + y_{s-1}}", hook, from=1-1, to=1-2]
	\arrow ["M \mathsf {i}_{y_1 + \dots  + y_{s-1}}"', hook, rounded corners, to path={ |- (\tikztotarget )[pos=0.75]\tikztonodes },from=1-1, to=2-2]
	\arrow ["\mu _j"', shorten >=3pt, shorten <=6pt, Rightarrow, from=0|-middle, to=0]
\end {tikzcd}
\end{center}

where \(s = f(j-1)\) and \(t= fj\),
so we obtain such a 2-cell from the \(\gamma \) cells of the pseudo double category:

  \begin{center}
\begin {tikzcd}
{M_1^{(y_1 + \dots  + y_n)}}&  [+2em]{M^{(y_s + \dots  + y_t)}_1} & [+4em] {M_1^{(x_j)}} & [+1em] {M_1}
	\arrow ["M \mathsf {i}_{y_1 + \dots  + y_{s-1}}", hook,  from=1-1, to=1-2]
	\arrow [""{name=0, anchor=center, inner sep=0},"{\mathsf {comp}_{y_s + \dots  + y_t}}"', popdown=2em, from=1-2, to=1-4]
	\arrow [""{name=1, anchor=center, inner sep=0},"{\mathsf {comp}_{y_s} \diamond  \dots  \diamond  \mathsf {comp}_{y_t}}", from=1-2, to=1-3]
\arrow ["{\mathsf {comp}_{x_j}}", from=1-3, to=1-4]
\arrow ["\gamma ", Rightarrow, shorten <=8pt, shorten >=8pt, from=0|-1, to=0]
\end {tikzcd}
\end{center}

modulo some subtleties regarding our replacement of \(\mathsf {comp}_1\) with \(1_{M_1}\) which we now address. The \(\gamma \) cell we \emph{actually} want to use to fill the above diagram is the one corresponding to the \emph{unital} pseudo double category constructed from \(M\). This 2-cell, \(\gamma '\), will be the same as the \(\gamma \) cell for \(M\) when none of  \(y_s, \dots , y_t\) is equal to one. In cases where some \(y_j \in  \left \{ y_s, \dots , y_t \right \}\) \emph{is} equal to 1, we need to glue some \(\iota \) cells to the original \(\gamma \) cell to get the correct \(\gamma '\) cell for the isomorphic unital pseudo double category. For example, if \(y_s, \dots  ,y_t\) is \(2,1,3\), \(\gamma '\) is constructed as follows:

  \begin{center}
\begin {tikzcd}
	{M_1^{(2)} \diamond  M_1 \diamond  M_1^{(3)}} & |[alias=two]| {M_1^{(2)} \diamond  M_1^{(1)} \diamond  M_1^{(3)}} & [+4em] |[alias=three]| {M_1^{(3)}} & {M_1}
	\arrow ["{1 \diamond  \theta _1 \diamond  1}"', from=1-1, to=1-2]
	\arrow [""{name=0, anchor=center, inner sep=0}, "{\mathsf {comp}_2 \diamond  1 \diamond  \mathsf {comp}_3}", popup=1.8em, from=1-1, to=1-3]
	\arrow ["{\mathsf {comp}_2 \diamond  \mathsf {comp}_1 \diamond  \mathsf {comp}_3}", from=1-2, to=1-3]
	\arrow [""{name=1, anchor=center, inner sep=0}, "{\mathsf {comp}_5}"', popdown=1.5em, from=1-2, to=1-4]
	\arrow ["{\mathsf {comp}_3}", from=1-3, to=1-4]
	\arrow ["{1 \diamond \iota ^{-1}\diamond 1}", shorten <=4pt, shorten >=10pt, Rightarrow, from=0, to=0|-two]
	\arrow ["\gamma ", shorten >=5pt, shorten <=7pt, Rightarrow, from=1|-three, to=1]
\end {tikzcd}
\end{center}

This construction is analogous to the construction of a normal pseudofunctor isomorphic to a given pseudofunctor, as described in \cite{lack-2006-nerves}. We observe that these \(\mu _j\) 2-cells are coherent with respect to the projections
 \(M\mathsf {d}_{0}, M\mathsf {d}_{1} : M_1 \to  M_0\), i.e. we have:
\begin{equation}M\mathsf {d}_{0} \; \mu _j = M\mathsf {d}_{1}\; \mu _{j+1}\end{equation}
by the fact that the \(\gamma \) and \(\iota \) cells are globular and using the properties:
\begin{equation}M\mathsf {d}_{1}\; \mathsf {comp}_n = M\mathsf {d}_{1}\;\pi _1 \qquad 
M\mathsf {d}_{0}\; \mathsf {comp}_n = M\mathsf {d}_{0}\;\pi _n\end{equation}
So the \(\mu _j\) cells do indeed induce a unique \(\mu \) cell which we take to be our compositor.

\par{}
The compositors for active morphisms as well as the strict functoriality for composites of actives and inerts with inerts now suffice to define general compositors. Given an arbitrary pair of composable 1-cells in \(\Delta \), \([m] \xrightarrow {f} [n] \xrightarrow {g} [l]\), with active–inert factorisations as follows:
\begin{equation}[m] \xrightarrow {f_1} [\widehat {f}] \xrightarrow {f_2} [n] \xrightarrow {g_1} [\widehat {g}] \xrightarrow {g_2} [l]\end{equation}
we obtain the compositor \(Mf \; Mg \Rightarrow   M \left ( gf \right )\) as the following pasting, where: 
\begin{equation}[\widehat {f}] \xrightarrow {g_1'} [x] \xrightarrow {f_2'} [\widehat {g}]\end{equation}
is the active–inert factorisation of \(g_1\;f_2\):

  \begin{center}
\begin {tikzcd}
	& {M_1^{({\widehat {g}})}} & {M_1^{(n)}} & {M_1^{({\widehat {f}})}} \\
	{M_1^{(l)}} && {M_1^{(x)}} && {M_1^{(m)}}
	\arrow ["{Mg_1}", from=1-2, to=1-3]
	\arrow ["{Mf_2'}"{description, inner sep=1pt}, from=1-2, to=2-3]
	\arrow ["{Mf_2}", from=1-3, to=1-4]
	\arrow ["\circlearrowright "{marking, allow upside down, pos=0.4}, draw=none, from=1-3, to=2-3]
	\arrow ["{Mf_1}", from=1-4, to=2-5]
	\arrow ["{Mg_2}", from=2-1, to=1-2]
	\arrow [""{name=0, anchor=center, inner sep=0}, "{M(g_2\,f_2')}"', from=2-1, to=2-3]
	\arrow ["{Mg_1'}"{description, inner sep=1pt}, from=2-3, to=1-4]
	\arrow [""{name=1, anchor=center, inner sep=0}, "{M(g_1'\,f_1)}"', from=2-3, to=2-5]
	\arrow ["\mu "', shorten <=6pt, shorten >=6pt, Rightarrow, from=1-2, to=0]
	\arrow ["\circlearrowright "{marking, allow upside down, pos=0.6}, draw=none, from=1-4, to=1]
\end {tikzcd}
\end{center}

These compositors are clearly identities when one of \(f\) or \(g\) is inert, and they can be seen to satisfy associativity by the associativity axioms on the \(\gamma \) and \(\iota \) cells after some unwinding of definitions.

\par{}
Having now defined the maps on objects in both directions between pseudo-models of \(\Delta ^{\mathsf {op}}\) in \(\mathcal {C}\mathsf {at}\) and pseudo double categories, we observe that starting with a pseudo double category \(M\) and mapping across in both directions we obtain a new pseudo double category \(M'\) which is \emph{isomorphic} to the original one. Indeed, its the \emph{unital} pseudo double category isomorphic to \(M\): the isomorphism \(F : M \to  M'\) has components \(F_0 : M_0 \to  M'_0 = M_0\) and \(F_1 : M_1 \to  M'_1 = M_1\) given by identities, and composition maps of the form:

  \begin{center}
\begin {tikzcd}
	{M_1^{(n)}} & {M_1^{(n)}} \\
	{M_1} & {M_1}
	\arrow [equals, from=1-1, to=1-2]
	\arrow [""{name=0, anchor=center, inner sep=0}, "{\mathsf {comp}_n}"', from=1-1, to=2-1]
	\arrow [""{name=1, anchor=center, inner sep=0}, "{\mathsf {comp}'_n}", from=1-2, to=2-2]
	\arrow [equals, from=2-1, to=2-2]
	\arrow ["{F\mathsf {comp}_n}", shift right=2, shorten <=13pt, shorten >=13pt, Rightarrow, from=0, to=1]
\end {tikzcd}
\end{center}

given by identities when \(n \neq  0\), and by \(\iota \) when \(n = 1\). We conclude in particular that the map on objects \(\widehat {(-)} : {\mathsf {ob}} \left ( \mathcal {P}\mathsf {sMod}^{\mathsf {p}}_{\mathscr {F}}(\Delta ^{\mathsf {op}}, \mathcal {C}\mathsf {at}) \right ) \to  {\mathsf {ob}} \left ( \mathcal {D}\mathsf {bl} \right )\) is essentially surjective on objects. It remains to extend this map on objects to a 2-fully-faithful functor, at which point we can conclude that the 2-functor is in fact a 2-equivalence.
\end{enumerate}
\par{}
Given a pseudo-\(\mathscr {F}\)-transformation \(\phi  : M \Rightarrow   N\) we define a pseudo double functor \(\widehat {\phi } : \widehat {M} \to  \widehat {N}\) as follows:
\begin{enumerate}\item{}The maps \(\widehat {\phi }_0\) and \(\widehat {\phi }_1\) are just the components of \(\phi \) at \([0]\) and \([1]\). By the strictness of \(\phi \) on inerts, the following diagram strictly commutes, as required:

  \begin{center}
\begin {tikzcd}
	{M_1} & {N_1} \\
	{M_0 \times  M_0} & {N_0 \times  N_0}
	\arrow [""{name=0, anchor=center, inner sep=0}, "{\widehat {\phi }_1}", from=1-1, to=1-2]
	\arrow ["{\langle  M\mathsf {d}_{1},M\mathsf {d}_{0}\rangle }"', from=1-1, to=2-1]
	\arrow ["{\langle  N\mathsf {d}_{1},N\mathsf {d}_{0}\rangle }", from=1-2, to=2-2]
	\arrow [""{name=1, anchor=center, inner sep=0}, "{\widehat {\phi }_0 \times  \widehat {\phi }_0}"', from=2-1, to=2-2]
	\arrow ["\circlearrowright "{marking, allow upside down}, draw=none, from=0, to=1]
\end {tikzcd}
\end{center}
  \item{}The pseudo-functoriality cells \(F \mathsf {comp}_n\) are given in terms of the pseudo-naturality cells of \(\phi \) as follows:

  \begin{center}
\begin {tikzcd}
	{M_1^{(n)}} & {N_1^{(n)}} \\
	{M_n} & {N_n} \\
	{M_1} & {N_1}
	\arrow ["{\phi _1^{(n)}}", from=1-1, to=1-2]
	\arrow [""{name=0, anchor=center, inner sep=0}, "{\theta _n}"', from=1-1, to=2-1]
	\arrow [""{name=1, anchor=center, inner sep=0}, "{\mathsf {comp}_n}"', popleft=3em, from=1-1, to=3-1]
	\arrow [""{name=2, anchor=center, inner sep=0}, "{\theta _n}", from=1-2, to=2-2]
	\arrow [""{name=3, anchor=center, inner sep=0}, "{\mathsf {comp}_n}", popright=3em, from=1-2, to=3-2]
	\arrow ["{\phi _n}"{description, inner sep=1pt}, from=2-1, to=2-2]
	\arrow [""{name=4, anchor=center, inner sep=0}, "{M \mathsf {a}_n}"', from=2-1, to=3-1]
	\arrow [""{name=5, anchor=center, inner sep=0}, "{N\mathsf {a}_n}", from=2-2, to=3-2]
	\arrow ["{\phi _1}"', from=3-1, to=3-2]
	\arrow ["\circlearrowright "{marking, allow upside down}, draw=none, from=1, to=2-1]
	\arrow ["\circlearrowright "{marking, allow upside down}, draw=none, from=0, to=2]
	\arrow ["\circlearrowright "{marking, allow upside down}, draw=none, from=3, to=2-2]
	\arrow ["{\phi _{\mathsf {a}_n}}", shift right=2, shorten <=13pt, shorten >=13pt, Rightarrow, from=4, to=5]
\end {tikzcd}
  \end{center}

The upper square commutes by the fact that \(\phi \) is strict on inerts. We can spell this out in detail: let \(\pi _j : M_1^{(n)} \to  M_1\) denote the projection onto the \(j^{\text {th}}\) component, and let \(\mathsf {i}_j\) denote the inert morphism \([1] \to  [n]\) picking out \(j-1\) and \(j\). Then \(\phi _1^{(n)}\) is characterised by the property \(\pi _j\;\phi _1^{(n)} = \phi _1 \; \pi _j\), and \(M \mathsf {i}_j\;\theta _n = \pi _j\). So we have:
\begin{equation}\begin {aligned}
N \mathsf {i}_j \; \phi _n \; \theta _n = \phi _1 \; M \mathsf {i}_j \; \theta _n
= \phi _1 \; \pi _j = \pi _j \; \phi _1^{(n)} = N \mathsf {i}_j \; \theta _n \; \phi _1^{(n)}
\end {aligned}
\end{equation}
from which it follows that \(\phi _n \; \theta _n = \theta _n \; \phi _1^{(n)}\). The cell \(\phi _{\mathsf {a}_n}\), and thus the pasting of the diagram above, is vertical with respect to \(N \mathsf {d}_{0}\) and \(N \mathsf {d}_{1}\) as required by the fact that post-composing by \(N \mathsf {d}_{0}\) gives the pseudo-naturality cell for \(\mathsf {a}_n \;\mathsf {d}_{0}\) by the strict-on-inerts property of \(M\), \(N\) and \(\phi \), and this cell is an identity because \(\mathsf {a}_n \;\mathsf {d}_{0}\) is inert. That these pseudo-functoriality 2-cells for \(\widehat {\phi }\) are coherent with the \(\gamma \) and \(\iota \) 2-cells as defined above follows from the coherence of the pseudo-naturality 2-cells of \(\phi \) with respect to the pseudo-functoriality of \(M\) and \(N\).
\end{enumerate}
Thus, for any pseudo-\(\mathscr {F}\)-transformation \(\phi  : M \Rightarrow   N\) we have a pseudo double functor \(\widehat {\phi } : \widehat {M} \to  \widehat {N}\). This construction can moreover be performed in reverse: given any pseudo double functor \(F : \widehat {M} \to  \widehat {N}\) we recover a pseudo-\(\mathscr {F}\)-transformation \(\widetilde {F} : M \Rightarrow   N\) as follows:
\begin{enumerate}\item{}Components \(\widetilde {F}_0\) and \(\widetilde {F}_1\) are given by \(F_0\) and \(F_1\). For \(n > 1\), \(\widetilde {F}_n\) is given by the composite:
  \begin{equation}M_n \xrightarrow {\theta _n^{-1}} M_1^{(n)} \xrightarrow {F_1^{(n)}} N_1^{(n)} \xrightarrow {\theta _n ^{-1}} N_n\end{equation}
  These components do satisfy strict naturality at inert morphisms. To see this, observe that for a morphism \(f : [n] \to  [m]\) in \(\Delta \), the following square will commute whenever the two paths are equal under post-composition by all projections \(\pi _j : M_1^{(n)} \to  M_1\):

  \begin{center}
\begin {tikzcd}
	{M_m} & {M_1^{(m)}} \\
	{M_n} & {M_1^{(n)}}
	\arrow ["{\theta _m^{-1}}", from=1-1, to=1-2]
	\arrow ["Mf"', from=1-1, to=2-1]
	\arrow ["{M_1^{(f)}}", from=1-2, to=2-2]
	\arrow ["{\theta _n^{-1}}"', from=2-1, to=2-2]
\end {tikzcd}
\end{center}

Post-composing the lower path by \(\pi _j\) gives \(M(f \; \mathsf {i}_j)\), whereas post-composing the upper path by \(\pi _j\) gives \(M \left ( \mathsf {i}_{fj} \right )\). Those \(f\) such \(f \; \mathsf {i}_j = \mathsf {i}_{fj}\) for all \(0 < j \leq  n\) are precisely the inerts. Our naturality squares for \(f\) inert therefore commute by the following decomposition:

  \begin{center}
\begin {tikzcd}
{M_m} & {M_1^{(m)}} & {N_1^{(m)}} & {N_m} \\
{M_n} & {M_1^{(n)}} & {N_1^{(n)}} & {N_n}
\arrow ["\theta _m ^{-1}", from=1-1, to= 1-2]
\arrow ["F_1^{(m)}", from=1-2, to= 1-3]
\arrow ["\theta _m", from=1-3, to= 1-4]
\arrow ["\theta _n ^{-1}"', from=2-1, to= 2-2]
\arrow ["F_1^{(n)}"', from=2-2, to= 2-3]
\arrow ["\theta _n"', from=2-3, to= 2-4]
\arrow [""{name=0, anchor=center, inner sep=0},"{Mf}"', from=1-1, to=2-1]
\arrow [""{name=3, anchor=center, inner sep=0},"{Nf}", from=1-4, to=2-4]
\arrow [""{name=1, anchor=center, inner sep=0},"{M_1^{(f)}}"{description, inner sep=1pt}, from=1-2, to=2-2]
\arrow [""{name=2, anchor=center, inner sep=0},"{N_1^{(f)}}"{description, inner sep=1pt}, from=1-3, to=2-3]
\arrow ["\widetilde {F}_m", popup=1em, from=1-1, to=1-4]
\arrow ["\widetilde {F}_n"', popdown=1em, from=2-1, to=2-4]
\arrow ["\circlearrowright "{marking, allow upside down}, draw=none, from=0, to=1]
\arrow ["\circlearrowright "{marking, allow upside down}, draw=none, from=1, to=2]
\arrow ["\circlearrowright "{marking, allow upside down}, draw=none, from=2, to=3]
\end {tikzcd}
\end{center}
\item{}The pseudo-naturality squares for the active maps \(\mathsf {a}_n : [1] \to  [n]\) are recovered from the pseudo-functoriality cells of \(F\):

  \begin{center}
\begin {tikzcd}
	{M_n} & [+1em] {N_n} \\
  {M_1^{(n)}} & {N_1^{(n)}} \\
	{M_1} & {N_1}
	\arrow ["{\widetilde {F}_n}", from=1-1, to=1-2]
	\arrow [""{name=0, anchor=center, inner sep=0}, "{\theta _n ^{-1}}"', from=1-1, to=2-1]
	\arrow [""{name=1, anchor=center, inner sep=0}, "{M \mathsf {a}_n}"', popleft=3em, from=1-1, to=3-1]
	\arrow [""{name=2, anchor=center, inner sep=0}, "{\theta _n ^{-1}}", from=1-2, to=2-2]
	\arrow [""{name=3, anchor=center, inner sep=0}, "{N \mathsf {a}_n}", popright=3em, from=1-2, to=3-2]
	\arrow ["{F_1^{(n)}}"{description,inner sep=1pt}, from=2-1, to=2-2]
	\arrow [""{name=4, anchor=center, inner sep=0}, "{\mathsf {comp}_n}"', from=2-1, to=3-1]
	\arrow [""{name=5, anchor=center, inner sep=0}, "{\mathsf {comp}_n}", from=2-2, to=3-2]
	\arrow ["{F_1}"', from=3-1, to=3-2]
	\arrow ["\circlearrowright "{marking, allow upside down}, draw=none, from=1, to=2-1]
	\arrow ["\circlearrowright "{marking, allow upside down}, draw=none, from=0, to=2]
	\arrow ["\circlearrowright "{marking, allow upside down}, draw=none, from=3, to=2-2]
	\arrow ["{F_\mathsf {comp}}", shift right=2, shorten <=13pt, shorten >=13pt, Rightarrow, from=4, to=5]
\end {tikzcd}
  \end{center}

These in fact determine the pseudo-naturality 2-cells of \(\widetilde {F}\) for all other morphisms. First, by the fact that every morphism in \(\Delta \) factors as an active followed by an inert, it suffices to define the pseudo-naturality 2-cells on the actives. Consider now the data required for the pseudo-naturality square for an active morphism \(a : [n] \to  [m]\):

  \begin{center}
\begin {tikzcd}
	{M_m} & {M_n} \\
	{N_m} & {N_n}
	\arrow ["Ma", from=1-1, to=1-2]
	\arrow [""{name=0, anchor=center, inner sep=0}, "{\widetilde {F}_m}"', from=1-1, to=2-1]
	\arrow [""{name=1, anchor=center, inner sep=0}, "{\widetilde {F}_n}", from=1-2, to=2-2]
	\arrow ["Na"', from=2-1, to=2-2]
	\arrow ["{\widetilde {F}_a}", shift right=2, shorten <=13pt, shorten >=13pt, Rightarrow, from=0, to=1]
\end {tikzcd}
\end{center}

By the universal property of \(N_n\), to provide such a 2-cell \(\widetilde {F}_a\) it suffices to specify the post-compositions of \(\widetilde {F}_a\) by the projections \(N \mathsf {i}_j\), which by the strictness of \(\widetilde {F}\) at inerts must be the pseudo-naturality cells for the composites \(a\;\mathsf {i}_j\):

  \begin{center}
\begin {tikzcd}
	{M_m} & {M_n} & {M_1} \\
	{N_m} & {N_n} & {N_1}
	\arrow ["Ma", from=1-1, to=1-2]
	\arrow [""{name=0, anchor=center, inner sep=0}, "{\widetilde {F}_m}"', from=1-1, to=2-1]
	\arrow ["{M \mathsf {i}_j}", from=1-2, to=1-3]
	\arrow [""{name=1, anchor=center, inner sep=0}, "{\widetilde {F}_n}"{description, inner sep=1pt}, from=1-2, to=2-2]
	\arrow [""{name=2, anchor=center, inner sep=0}, "{F_1}", from=1-3, to=2-3]
	\arrow ["Na"', from=2-1, to=2-2]
	\arrow ["{N \mathsf {i}_j}"', from=2-2, to=2-3]
	\arrow ["{\widetilde {F}_a}", shift right=2, shorten <=13pt, shorten >=13pt, Rightarrow, from=0, to=1]
\arrow ["\circlearrowright "{marking, allow upside down}, draw=none, from=1, to=2]
\arrow ["M \left ( a\; \mathsf {i}_j \right )", popup=1em, from=1-1, to=1-3]
\arrow ["N \left ( a\; \mathsf {i}_j \right )"', popdown=1em, from=2-1, to=2-3]
\end {tikzcd}
\end{center}

The composite \([1] \xrightarrow {\mathsf {i}_j} [n] \xrightarrow {a} [m]\) has an active-inert factorisation \([1] \xrightarrow {\mathsf {a}_{k_j}} [k_j] \xrightarrow {i_j} [m]\), so by the strictness of \(M\), \(N\) and \(\widetilde {F}\) on inerts, the pasting of the squares above is equally given by the pasting:

  \begin{center}
\begin {tikzcd}
	{M_m} & {M_{k_j}} & {M_1} \\
	{N_m} & {N_{k_j}} & {N_1}
	\arrow ["Mi_j", from=1-1, to=1-2]
	\arrow [""{name=0, anchor=center, inner sep=0}, "{\widetilde {F}_m}"', from=1-1, to=2-1]
	\arrow ["{M \mathsf {a}_{k_j}}", from=1-2, to=1-3]
	\arrow [""{name=1, anchor=center, inner sep=0}, "{\widetilde {F}_{k_j}}"{description, inner sep=1pt}, from=1-2, to=2-2]
	\arrow [""{name=2, anchor=center, inner sep=0}, "{F_1}", from=1-3, to=2-3]
	\arrow ["Ni_j"', from=2-1, to=2-2]
	\arrow ["{N \mathsf {a}_{k_j}}"', from=2-2, to=2-3]
	\arrow ["\circlearrowright "{marking, allow upside down}, draw=none, from=0, to=1]
	\arrow ["{\widetilde {F} \mathsf {a}_{k_j}}", shift right=2, shorten <=13pt, shorten >=13pt, Rightarrow, from=1, to=2]
\end {tikzcd}
\end{center}

defined in terms of the pseudo-naturality cells at morphisms of the form \(\mathsf {a}_k\). These 2-cells will lift to define the 2-cell \(\widetilde {F}_a\) if they are coherent with respect to the maps \(N\mathsf {d}_{0}, N\mathsf {d}_{1} : N_1 \to  N_0\), i.e. if the following pasting diagrams agree:

  \begin{center}
\begin {tikzcd}
	{M_m} & {M_{k_j}} & {M_1} & {M_0}\\
	{N_m} & {N_{k_j}} & {N_1} & {N_0}
\arrow ["Mi_j", from=1-1, to=1-2]
\arrow ["M \mathsf {d}_{0}", from=1-3, to=1-4]
\arrow [""{name=end, anchor=center, inner sep=0},"F_0", from=1-4, to=2-4]
	\arrow [""{name=0, anchor=center, inner sep=0}, "{\widetilde {F}_m}"', from=1-1, to=2-1]
	\arrow ["{M \mathsf {a}_{k_j}}", from=1-2, to=1-3]
	\arrow [""{name=1, anchor=center, inner sep=0}, "{\widetilde {F}_{k_j}}"{description, inner sep=1pt}, from=1-2, to=2-2]
	\arrow [""{name=2, anchor=center, inner sep=0}, "{F_1}", from=1-3, to=2-3]
	\arrow ["Ni_j"', from=2-1, to=2-2]
	\arrow ["{N \mathsf {a}_{k_j}}"', from=2-2, to=2-3]
	\arrow ["{N\mathsf {d}_{0}}"', from=2-3, to=2-4]
	\arrow ["\circlearrowright "{marking, allow upside down}, draw=none, from=0, to=1]
	\arrow ["\circlearrowright "{marking, allow upside down}, draw=none, from=2, to=end]
	\arrow ["{\widetilde {F} \mathsf {a}_{k_j}}", shift right=2, shorten <=13pt, shorten >=13pt, Rightarrow, from=1, to=1-|2]
\end {tikzcd}
\qquad 
\begin {tikzcd}
	{M_m} & {M_{k_{j+1}}} & {M_1} &{M_0}\\
	{N_m} & {N_{k_{j+1}}} & {N_1} & {N_0}
\arrow ["{Mi_{j+1}}", from=1-1, to=1-2]
\arrow ["M \mathsf {d}_{1}", from=1-3, to=1-4]
\arrow [""{name=end, anchor=center, inner sep=0},"F_0", from=1-4, to=2-4]
	\arrow [""{name=0, anchor=center, inner sep=0}, "{\widetilde {F}_m}"', from=1-1, to=2-1]
	\arrow ["{M \mathsf {a}_{k_{j+1}}}", from=1-2, to=1-3]
	\arrow [""{name=1, anchor=center, inner sep=0}, "{\widetilde {F}_{k_{j+1}}}"{description, inner sep=1pt}, from=1-2, to=2-2]
	\arrow [""{name=2, anchor=center, inner sep=0}, "{F_1}", from=1-3, to=2-3]
	\arrow ["{N_{j+1}}"', from=2-1, to=2-2]
	\arrow ["{N \mathsf {a}_{k_{j+1}}}"', from=2-2, to=2-3]
	\arrow ["{N\mathsf {d}_{0}}"', from=2-3, to=2-4]
	\arrow ["\circlearrowright "{marking, allow upside down}, draw=none, from=0, to=1]
	\arrow ["{\widetilde {F} \mathsf {a}_{k_{j+1}}}", shift right=2, shorten <=16pt, shorten >=16pt, Rightarrow, from=1, to=1-|2]
\end {tikzcd}
\end{center}

which they do because the \(F_\mathsf {comp}\) cells — and thus the \(\widetilde {F}{\mathsf {a}_k}\) cells — are vertical with respect to \(N\mathsf {d}_{0}\) and \(N\mathsf {d}_{1}\), and the boundary of both diagrams is:

  \begin{center}
\begin {tikzcd}
	{M_m} & {M_0} \\
	{N_m} & {N_0}
	\arrow ["{M\mathsf {i}_{j+1}}", from=1-1, to=1-2]
	\arrow ["{\widetilde {F}_m}"', from=1-1, to=2-1]
	\arrow ["{F_0}", from=1-2, to=2-2]
	\arrow ["{N\mathsf {i}_{j}}"', from=2-1, to=2-2]
\end {tikzcd}
\end{center}

where \(\mathsf {i}_{j} : [0] \to  [m]\) is the map picking out object \(j-1\). This diagram commutes because \(\mathsf {u}_j\) is inert.

These 2-cells do indeed satisfy the property that \(N \mathsf {d}_{0} \; \widetilde {F} \mathsf {a}_{k_j} = N \mathsf {d}_{1}\; \widetilde {F} \mathsf {a}_{k_{j+1}}\), and so suffice to determine the pseudo-naturality cell for \(a\). One can then verify by a similar argument that for composites:
\begin{equation}[m] \xrightarrow {i} [n] \xrightarrow {a} [l]\end{equation} with \(i\) inert and \(a\) active, and with active–inert factorisation \(ai = i'a'\) that the following two pastings agree:

  \begin{center}
\begin {tikzcd}
	{M_l} & {M_{n}} & {M_m} \\
	{N_l} & {N_{n}} & {N_m}
	\arrow ["Ma", from=1-1, to=1-2]
	\arrow [""{name=0, anchor=center, inner sep=0}, "{\widetilde {F}_l}"', from=1-1, to=2-1]
	\arrow ["{Mi}", from=1-2, to=1-3]
	\arrow [""{name=1, anchor=center, inner sep=0}, "{\widetilde {F}_n}"{description, inner sep=1pt}, from=1-2, to=2-2]
	\arrow [""{name=2, anchor=center, inner sep=0}, "{F_m}", from=1-3, to=2-3]
	\arrow ["Na"', from=2-1, to=2-2]
	\arrow ["{Ni}"', from=2-2, to=2-3]
	\arrow ["\circlearrowright "{marking, allow upside down}, draw=none, from=1, to=2]
	\arrow ["{\widetilde {F} a}", shift right=2, shorten <=13pt, shorten >=13pt, Rightarrow, from=0, to=1]
\end {tikzcd}
\qquad 
\begin {tikzcd}
	{M_l} & {M_{n'}} & {M_m} \\
	{N_l} & {N_{n'}} & {N_m}
	\arrow ["Mi'", from=1-1, to=1-2]
	\arrow [""{name=0, anchor=center, inner sep=0}, "{\widetilde {F}_l}"', from=1-1, to=2-1]
	\arrow ["{Ma'}", from=1-2, to=1-3]
	\arrow [""{name=1, anchor=center, inner sep=0}, "{\widetilde {F}_n'}"{description, inner sep=1pt}, from=1-2, to=2-2]
	\arrow [""{name=2, anchor=center, inner sep=0}, "{F_m}", from=1-3, to=2-3]
	\arrow ["Ni'"', from=2-1, to=2-2]
	\arrow ["{Na'}"', from=2-2, to=2-3]
	\arrow ["\circlearrowright "{marking, allow upside down}, draw=none, from=0, to=1]
	\arrow ["{\widetilde {F} a'}", shift right=2, shorten <=13pt, shorten >=13pt, Rightarrow, from=1, to=2]
\end {tikzcd}
\end{center}

by showing that they agree under post-whiskering by the \(N \mathsf {i}_j\) maps,
and thus that the putative pseudo-naturality 2-cells assemble into a well-defined pseudo-\(\mathscr {F}\)-transformation.
\end{enumerate}
It is then straight-forward to check that this construction of a pseudo-\(\mathscr {F}\)-transformation from a pseudo double functor is clearly inverse to \(\phi  \mapsto  \widehat {\phi }\) and preserves composition.

  \par{}
  Consider now the 2-cells. A 2-cell between pseudo double functors \(F,G : M \to  N\) is given by natural transformations \(\Gamma _0 : F_0 \Rightarrow   G_0\) and \(\Gamma _1 : F_1 \Rightarrow   G_1\) satisfying \(N \mathsf {d}_{i} \; \Gamma _1 = \Gamma _0 \; M\mathsf {d}_{i}\) for \(i = 0,1\):

  \begin{center}
\begin {tikzcd}[row sep=1em]
	{M_1} & {N_1} \\
	{M_0} & {N_0}
	\arrow [""{name=0, anchor=center, inner sep=0}, "{G_1}"', from=1-1, to=1-2]
	\arrow [""{name=1, anchor=center, inner sep=0}, "{F_1}", popup=1em, from=1-1, to=1-2]
	\arrow ["{M\mathsf {d}_{i}}"', from=1-1, to=2-1]
	\arrow ["{N\mathsf {d}_{i}}", from=1-2, to=2-2]
	\arrow [""{name=2, anchor=center, inner sep=0}, "{F_0}", from=2-1, to=2-2]
	\arrow [""{name=3, anchor=center, inner sep=0}, "{G_0}"', popdown=1em, from=2-1, to=2-2]
	\arrow ["{\Gamma _1}", shorten <=5pt, shorten >=5pt, Rightarrow, from=1, to=1|-0]
	\arrow ["{\Gamma _0}", shorten <=5pt, shorten >=5pt, Rightarrow, from=2, to=2|-3]
\end {tikzcd}
  \end{center}

These 2-cells must additionally satisfy coherence conditions with respect to the composition cells for \(F\) and \(G\) expressed as the equality of the below two pasting diagrams:

  \begin{center}
\begin {tikzcd}
	{M_1^{(n)}} & {N_1^{(n)}} \\
	{M_1} & {N_1}
	\arrow [""{name=0, anchor=center, inner sep=0}, "{G_1^{(n)}}"', from=1-1, to=1-2]
	\arrow [""{name=1, anchor=center, inner sep=0}, "{F_1^{(n)}}", popup=1em, from=1-1, to=1-2]
	\arrow [""{name=2, anchor=center, inner sep=0}, "{\mathsf {comp}_n}"', from=1-1, to=2-1]
	\arrow [""{name=3, anchor=center, inner sep=0}, "{\mathsf {comp}_n}", from=1-2, to=2-2]
	\arrow ["{G_1}"', from=2-1, to=2-2]
	\arrow ["{\Gamma _1^{(n)}}", shorten <=6pt, shorten >=6pt, Rightarrow, from=1, to=1|-0]
	\arrow ["{G\mathsf {comp}_n}", shift right=3, shorten <=6pt, shorten >=6pt, Rightarrow, from=2, to=3]
  \end {tikzcd}
  \qquad 
  \begin {tikzcd}
	{M_1^{(n)}} & {N_1^{(n)}} \\
	{M_1} & {N_1}
	\arrow [""{name=1, anchor=center, inner sep=0}, "{F_1^{(n)}}", from=1-1, to=1-2]
	\arrow [""{name=2, anchor=center, inner sep=0}, "{\mathsf {comp}_n}"', from=1-1, to=2-1]
	\arrow [""{name=3, anchor=center, inner sep=0}, "{\mathsf {comp}_n}", from=1-2, to=2-2]
	\arrow [""{name=5, anchor=center, inner sep=0},"{F_1}", from=2-1, to=2-2]  
	\arrow [""{name=4, anchor=center, inner sep=0},"{G_1}"', popdown=1em, from=2-1, to=2-2]
	\arrow ["{\Gamma _1}", shorten <=4pt, shorten >=4pt, Rightarrow, from=5, to=5|-4]
	\arrow ["{F\mathsf {comp}_n}", shorten <=6pt, shorten >=6pt, Rightarrow, from=2, to=3]
\end {tikzcd}
  \end{center}

The data of \(\Gamma _0\) and \(\Gamma _1\) are a subset of the data of a modification between the corresponding pseudo-\(\mathscr {F}\)-transformations, but we once again observe that we can reconstruct the missing data — the \(\Gamma _n : F_n \Rightarrow   G_n : M_n \to  N_n\) maps — from just \(\Gamma _0\) and \(\Gamma _1\). Indeed, by the pseudo-model property for \(N\), such a 2-cell \(\Gamma _n\) is determined by its post-composition with the maps \(N \mathsf {i}_j : N_n \to  N_1\), and for the \(\Gamma _n\)'s to be natural in \(n\) it must be the case that \(N \mathsf {i}_j \; \Gamma _n = \Gamma _1 \; M \mathsf {i}_j\), and so \(\Gamma _n\) is given by:

  \begin{center}
\begin {tikzcd}
	{M_n} & {M_1^{(n)}} & {N_1^{(n)}} & {N_n}
	\arrow ["{\theta ^{-1}}", from=1-1, to=1-2]
	\arrow [""{name=0, anchor=center, inner sep=0}, "{F_1^{(n)}}", popup=0.5em, from=1-2, to=1-3]
	\arrow [""{name=1, anchor=center, inner sep=0}, "{G_1^{(n)}}"', popdown=0.5em, from=1-2, to=1-3]
	\arrow ["{\theta _n}", from=1-3, to=1-4]
	\arrow ["{\Gamma _1^{(n)}}", shift right=2, shorten <=4pt, shorten >=4pt, Rightarrow, from=0, to=1]
\end {tikzcd}
  \end{center}

A similar argument using projections \(N \mathsf {i}_j\) shows that the pseudo-naturality of the \(\Gamma _n\)'s with respect to active maps follows from pseudo-naturality with respect to the maps \(\mathsf {a}_k\), which is equivalent to the coherence of \(\Gamma \) with respect to the composition cells \(F \mathsf {comp}_n\) and \(G \mathsf {comp}_n\) described above. Thus the map from modifications between pseudo-\(\mathscr {F}\)-transformations and vertical transformations between the corresponding  pseudo double functors given by restricting to just the component at \([0]\) and \([1]\) has a unique section, and so is bijective. Restriction to the components at \([0]\) and \([1]\) clearly preserves vertical composition as well as pre and post whiskering by 1-cells.
\end{proof}
\subsection{Model opfibrations are closed under composition}\label{jrb-002D}\par{}
In \cref{jrb-001I} it is shown that postcomposition by a \emph{discrete} model opfibration preserves and reflects the property of being a model opfibration. To see that the discreteness condition is necessary, consider the case where the locally discrete \(\mathscr {F}\)-sketch \(\mathsf {A}\) has no marked cones and is inchordate as an \(\mathscr {F}\)-category. In this case, a model opfibration over \(\mathsf {A}\) is just an opfibration in \(\mathcal {C}\mathsf {at}\), and opfibrations in \(\mathcal {C}\mathsf {at}\) aren't right-cancellative. For example, the unique morphism from a category \(\mathsf {C}\) to the terminal category is an opfibration if and only if \(\mathsf {C}\) is a groupoid, but not all maps between groupoids are opfibrations.
\par{}
However, arbitrary opfibrations \emph{are} closed under composition, so one might expect the same to be true for model opfibrations. Indeed, the proof that post-composing a model opfibration by a discrete model opfibration \(q\) in \cref{jrb-001I-proof} only appealed to the discreteness of \(q\) in demonstrating the "lifting of 2-cells" property. The discreteness hypothesis does simplify the proof of this property but it is not essential, as we now demonstrate.

\begin{proposition}[{Model opfibrations are closed under composition}]\label{jrb-002E}
\par{}
Given a locally discrete \(\mathscr {F}\)-sketch, \(\mathsf {A}\), a model opfibration \(p : \mathsf {B} \to  \mathsf {A}\) and another model opfibration \(q : \mathsf {C} \to  \mathsf {B}\) with respect to the induced  \(\mathscr {F}\)-sketch structure on \(\mathsf {B}\), the composite \(p\,q : \mathsf {C} \to  \mathsf {A}\) is a model opfibration.

\begin{proof}[{proof of \cref{jrb-002E}.}]\label{jrb-002E-proof}
\par{}
As noted above, the argument given in \cref{jrb-001I-proof} for the existence of a unique diagonal filler for the square below, where \(S : \mathsf {J} \to  x \downarrow  {\mathsf {A}}_{\mathsf {inert}}\) is a marked cone of \(\mathsf {A}\) and \(L : \mathsf {J} \to  \mathsf {C}\) some lift along \(p\,q\), does not depend on the discreteness of \(p\).

  \begin{center}
\begin {tikzcd}[row sep=1em]
	{\mathsf {J}} & {\mathsf {C}} \\
	& {\mathsf {B}} \\
	{x \downarrow  {\mathsf {A}}_{\mathsf {inert}}} & {\mathsf {A}}
	\arrow [""{name=0, anchor=center, inner sep=0}, "L", from=1-1, to=1-2]
	\arrow ["S"', from=1-1, to=3-1]
	\arrow ["q", from=1-2, to=2-2]
	\arrow ["{p}", from=2-2, to=3-2]
	\arrow [""{name=1, anchor=center, inner sep=0}, "{\mathsf {cod}}"', from=3-1, to=3-2]
	\arrow ["\circlearrowright "{marking, allow upside down}, draw=none, from=0, to=1]
\end {tikzcd}
\end{center}\par{}
It therefore remains only to show that for two diagonals, \(K,L : x \downarrow  {\mathsf {A}}_{\mathsf {inert}} \to  \mathsf {C}\) and a \(p\,q\)-vertical 2-cell \(\alpha  : KS \Rightarrow   LS\), as shown below left, there exists a unique \(p\,q\)-vertical \(\beta  : K \Rightarrow   L\) such that \(\beta  \,S = \alpha \). To this end, first let \(l \in  \mathsf {B}\) be the image of \(1_x \in  x \downarrow  {\mathsf {A}}_{\mathsf {inert}}\) under \(q\,L\). The map \(l \downarrow  p : l \downarrow  {\mathsf {B}}_{\mathsf {inert}} \to  x \downarrow  {\mathsf {A}}_{\mathsf {inert}}\) is an isomorphism by \cref{jrb-001R}, so we can make the following assignments:
\begin{equation}S' = \left ( l \downarrow  p \right )^{-1}\,S \qquad  K' = K\,\left ( l \downarrow  p \right ) \qquad  L' = L\,\left ( l \downarrow  p \right )\end{equation}
as shown on the right below.

  \begin{center}
\begin {tikzcd}[row sep=1em]
	{\mathsf {J}} & {\mathsf {C}} \\
	& {\mathsf {B}} \\
	{x \downarrow  {\mathsf {A}}_{\mathsf {inert}}} & {\mathsf {A}}
	\arrow [""{name=0, anchor=center, inner sep=0}, "KS"', from=1-1, to=1-2]
	\arrow [""{name=1, anchor=center, inner sep=0}, "LS", popup=1em, from=1-1, to=1-2]
	\arrow ["S"', from=1-1, to=3-1]
	\arrow ["q", from=1-2, to=2-2]
	\arrow ["p", from=2-2, to=3-2]
  \arrow ["L"{description}, shift left=2, from=3-1, to=1-2]
	\arrow ["K"{description}, shift right=2, from=3-1, to=1-2]
	\arrow [""{name=2, anchor=center, inner sep=0}, "{\mathsf {cod}_\mathsf {A}}"', from=3-1, to=3-2]
	\arrow ["{\alpha }"{outer sep=1pt},shorten <=5pt, shorten >=5pt, Rightarrow, from=0, to=0|-1]
\end {tikzcd}
\qquad  $\leadsto $
\qquad 
\begin {tikzcd}[row sep=1em]
	{\mathsf {J}} & {\mathsf {C}} \\
	{l \downarrow  {\mathsf {B}}_{\mathsf {inert}}} & {\mathsf {B}} \\
	{x \downarrow  {\mathsf {A}}_{\mathsf {inert}}} & {\mathsf {A}}
	\arrow [""{name=0, anchor=center, inner sep=0}, "KS"{description}, from=1-1, to=1-2]
	\arrow [""{name=1, anchor=center, inner sep=0}, "LS", popup=1em, from=1-1, to=1-2]
	\arrow ["S'"', from=1-1, to=2-1]
	\arrow ["{l \downarrow  p}"', from=2-1, to=3-1]
	\arrow ["q", from=1-2, to=2-2]
	\arrow ["p", from=2-2, to=3-2]
  \arrow ["L'"{description}, shift left=1, from=2-1, to=1-2]
	\arrow ["K'"{description}, shift right=2.6, from=2-1, to=1-2]
	\arrow [""{name=2, anchor=center, inner sep=0}, "{\mathsf {cod}_{\mathsf {A}}}"', from=3-1, to=3-2]
	\arrow [""{name=3, anchor=center, inner sep=0}, "{\mathsf {cod}_{\mathsf {B}}}"', from=2-1, to=2-2]
  \arrow ["{\alpha }"{outer sep=1pt},shorten <=5pt, shorten >=5pt, Rightarrow, from=0, to=0|-1]
	\arrow ["\circlearrowright "{marking, allow upside down}, draw=none, from=2, to=3]
\end {tikzcd}  
\end{center}\par{}
Note that \(q\,L' = \mathsf {cod}_\mathsf {B}\), since both are morphisms of \(\mathscr {F}\)-fibrations which agree on \(1_x\).
\par{}
By the model opfibration property of \(p\) there exists a unique 2-cell \(\sigma  : qK \Rightarrow   qL\) satisfying \(\sigma  \,S = q\,\alpha \) and \(p\,\sigma  = 1_{\mathsf {cod}_\mathsf {A}}\). Let \(\sigma '= \sigma  \, \left ( l \downarrow  p \right )\). By the opfibration property of \(q\), we can take a \(q\)-opcartesian lift \(\tau  : K' \Rightarrow   \hat {L}\) of \(\sigma '\) along \(q\). Being a lift of \(\sigma '\), we must have in particular that \(q\hat {L} = qL' = \mathsf {cod}_\mathsf {B}\).
\par{}
The 2-cell \(\tau  \,S'\) is then also \(q\)-opcartesian, so we can factorise \(\alpha \) through \(\tau \,S'\) via a \(q\)-vertical 2-cell \(\hat {\alpha }\) as follows:

  \begin{center}
\begin {tikzcd}[column sep=4em]
\mathsf {J} & \mathsf {C}
\arrow [""{name=0, anchor=center, inner sep=0}, "KS=K'S'"', popdown=1em, from=1-1, to=1-2]
\arrow [""{name=1, anchor=center, inner sep=0}, "LS=L'S'", popup=1em, from=1-1, to=1-2]
\arrow ["{\alpha }"{outer sep=1pt},shorten <=5pt, shorten >=5pt, Rightarrow, from=0, to=0|-1]
\end {tikzcd}
\qquad  $=$ \qquad 
\begin {tikzcd}[column sep=4em]
\mathsf {J} & \mathsf {C}
\arrow [""{name=0, anchor=center, inner sep=0}, "K'S'"', popdown=1em, from=1-1, to=1-2]
\arrow [""{name=2, anchor=center, inner sep=0}, "L'S'", popup=1em, from=1-1, to=1-2]
\arrow [""{name=1, anchor=center, inner sep=0}, "\hat {L}S'"{description}, from=1-1, to=1-2]
\arrow ["{\tau \,S'}"{outer sep=1pt},shorten <=5pt, shorten >=5pt, Rightarrow, from=0, to=0|-1]
\arrow ["{\hat {\alpha }}"{outer sep=1pt},shorten <=5pt, shorten >=5pt, Rightarrow, from=1, to=1|-2]
\end {tikzcd}
\end{center}

We can then use the model opfibration property of \(q\) to factor \(\hat {\alpha }\) through \(S'\) as \(\hat {\alpha } = \tilde {\alpha } \, S'\).
\par{}
Our 2-cell \(\beta : K \Rightarrow  L\) is then defined as:
\begin{equation}K = K' \left ( l \downarrow  p \right )^{-1} \xRightarrow {\tau \, \left ( l \downarrow  p \right )^{-1}}
\hat {L} \left ( l \downarrow  p \right )^{-1} \xRightarrow { \tilde {\alpha }\,\left ( l \downarrow  p \right )^{-1}}
L' \left ( l \downarrow  p \right )^{-1} = L
\end{equation}
We first verify that it satisfies the necessary conditions:
\begin{equation}
\begin {aligned}
&p\,q\,\beta  =
p\,q\,\left ( \tilde {\alpha } \circ  \tau   \right )\,\left ( l \downarrow  p \right )^{-1} =
p\,\left ( 1_{\mathsf {cod}_\mathsf {B}} \circ  \sigma '  \right )\,\left ( l \downarrow  p \right )^{-1} =
p\, \sigma  =
1_{\mathsf {cod}_\mathsf {A}} \\
& \beta  \,S =
\left ( \tilde {\alpha } \circ  \tau   \right )\,\left ( l \downarrow  p \right )^{-1}\,S
= \left ( \tilde {\alpha } \circ  \tau   \right )\,S'
= \hat {\alpha } \circ  \tau \,S'
= \alpha 
\end {aligned}
\end{equation}\par{}
To demonstrate uniqueness, assume there were some other \(\gamma  : K \Rightarrow  L\) satisfying the properties above. Then \(q\,\gamma  = \sigma \) by the model-opfibration property of \(q\). Let \(\gamma ' = \gamma  \, \left ( l \downarrow  p \right )\) and let \(\hat {\gamma } : \hat {L} \Rightarrow   L'\) denote the unique \(q\)-vertical factorisation of \(\gamma '\) through \(\tau \). It follows that \(\hat {\gamma }\,S'\) is the unique \(q\)-vertical factorisation of \(\gamma '\,S' = \alpha \) through \(\tau \,S'\), which means \(\hat {\gamma }\,S' = \hat {\alpha }\). But \(\tilde {\alpha }\) is the unique \(q\)-vertical 2-cell with the property \(\tilde {\alpha } \,S' = \hat {\alpha }\), so \(\hat {\gamma } = \tilde {\alpha }\) and we have:
\begin{equation}\gamma  = \left ( \hat {\gamma } \, \tau  \right )\,\left ( l \downarrow  p \right )^{-1}
= \left ( \tilde {\alpha } \, \tau  \right )\,\left ( l \downarrow  p \right )^{-1} = \beta 
\end{equation}\end{proof}
\end{proposition}

\section{Declaration}
Funded by the Advanced Research + Invention Agency (ARIA) through
project codes MSAI-PR01-P14 and MSAI-PRO1-P15.

\printbibliography

\end{document}